%% file: paper.tex
\title{Real second-order freeness and the asymptotic real second-order freeness of several real matrix models}
\author{C.\ Emily I.\ Redelmeier\thanks{Research supported by a two-year Sophie Germain post-doctoral scholarship provided by the Fondation math\'{e}matique Jacques Hadarmard, held at the D\'{e}partement de Math\'{e}matiques, UMR 8628 Universit\'{e} Paris-Sud 11-CNRS, B\^{a}timent 425, Facult\'{e} des Sciences d'Orsay, Universit\'{e} Paris-Sud 11, F-91405 Orsay Cedex.  E-mail address: emily.redelmeier@math.u-psud.fr}}

\documentclass[11pt]{article}

\usepackage{graphicx}
\usepackage{graphics}
\usepackage{amsfonts}
\usepackage{amssymb}
\usepackage{amsthm}
\usepackage{amsmath}
\usepackage{color}
\usepackage{cite}

\newtheorem{theorem}{Theorem}[section]
\newtheorem{lemma}[theorem]{Lemma}

\theoremstyle{remark}
\newtheorem{remark}[theorem]{Remark}

\newtheorem{example}[theorem]{Example}

\theoremstyle{definition}
\newtheorem{definition}[theorem]{Definition}

\begin{document}

\maketitle

\begin{abstract}
We introduce real second-order freeness in second-order noncommutative probability spaces.  We demonstrate that under this definition, three real models of random matrices, namely real Ginibre matrices, Gaussian orthogonal matrices, and real Wishart matrices, are asymptotically second-order free.  These ensembles do not satisfy the complex definition of second-order freeness satisfied by their complex analogues.  We use a combinatorial approach to the matrix calculations similar to the genus expansion for complex random matrices, but in which nonorientable surfaces appear, demonstrating the commonality between the real models and the distinction from their complex analogues, motivating this distinct definition.   In the real case we find, in addition to the terms appearing in the complex case corresponding to annular spoke diagrams, an extra set of terms corresponding to annular spoke diagrams in which the two circles of the annulus are oppositely oriented, and in which the matrix transpose appears.
\end{abstract}

\section{Introduction}

In a noncommutative probability space, freeness is an analogue of independence in a classical probability space.  Many important random matrix models, both real and complex, are asymptotically free; that is, as the size of the matrix becomes large, independent matrices satisfy freeness conditions on the expected values of the traces of their products \cite{MR1094052}.  No modification of the definition is required for real random matrices.

In addition to considering the asymptotic moments of random matrices, that is, the expected values of traces of products in the large matrix limit, it is also possible to consider the distributions or fluctuations around these expected values (central limit-type theorems as opposed to law of large numbers-type theorems).  Many important matrix models (including all of those considered in this paper) are asymptotically Gaussian, so the most important quantity of these fluctuations is the appropriately rescaled asymptotic covariance of traces.  Second-order probability spaces, that is, noncommutative probability spaces equipped with a bilinear function modelling this covariance, were introduced in \cite{MR2216446} and further studied in \cite{MR2294222, MR2302524}.  Also given are definitions for second-order freeness and asymptotic second-order freeness, and it is shown that several important complex matrix models satisfy this definition.  Second-order freeness then has the role for fluctuations that first-order freeness has for moments.  In particular, if random matrices \(A\) and \(B\) are asymptotically second-order free, the asymptotic fluctuations of \(A+B\) and \(AB\) can be calculated from the asymptotic moments and fluctuations of \(A\) and \(B\).

Asymptotic second-order freeness as defined in \cite{MR2216446} is not generally satisfied by real ensembles of random matrices.  If random matrices \(A_{k,N}\) and \(B_{l,N}\), \(k,l=1,\ldots,p\) are elements of the algebra generated by a model studied in this paper, then the relation satisfied instead is
\begin{multline*}
\lim_{N\rightarrow\infty}\mathrm{cov}\left(\mathrm{Tr}\left(\mathaccent"7017{A}_{1,N}\cdots\mathaccent"7017{A}_{p,N}\right),\mathrm{Tr}\left(\mathaccent"7017{B}_{1,N}\cdots\mathaccent"7017{B}_{p,N}\right)\right)
\\=\sum_{k=0}^{p-1}\prod_{i=1}^{p}\lim_{N\rightarrow\infty}\mathbb{E}\left(\mathrm{tr}\left(\mathaccent"7017{A}_{i,N}\mathaccent"7017{B}_{k-i,N}\right)\right)+\sum_{k=0}^{p-1}\prod_{i=0}^{p}\lim_{N\rightarrow\infty}\mathbb{E}\left(\mathrm{tr}\left(\mathaccent"7017{A}_{i,N}\mathaccent"7017{B}_{k+i,N}^{T}\right)\right)
\end{multline*}
(where the circle above the random matrix terms indicates that they have been centred: \(\mathaccent"7017{X}:=X-\mathbb{E}\left(\mathrm{tr}\left(X\right)\right)\), cyclically adjacent terms in each entry come from algebras generated by independent ensembles, and all indices are taken modulo \(p\)), while in the limit covariances of cyclically alternating terms with different numbers of terms vanish, as in the complex case.  (Below we will generally suppress the index \(N\).)  The first sum on the right-hand side appears in the definition of complex second-order freeness, and its summands correspond to ``spoke diagrams'' on an annulus, as in the first row of Figure~\ref{spoke} and described in \cite{MR2216446}.  The second sum does not appear in the complex case.  Its summands correspond to the spoke diagrams in the second row of Figure~\ref{spoke} in which the two circles of the annulus are oppositely oriented.

\begin{figure}
\centering
\begin{tabular}{ccc}
\scalebox{0.5}{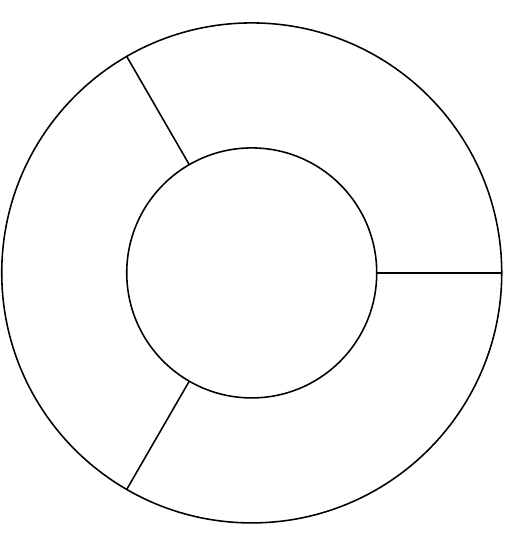}&\scalebox{0.5}{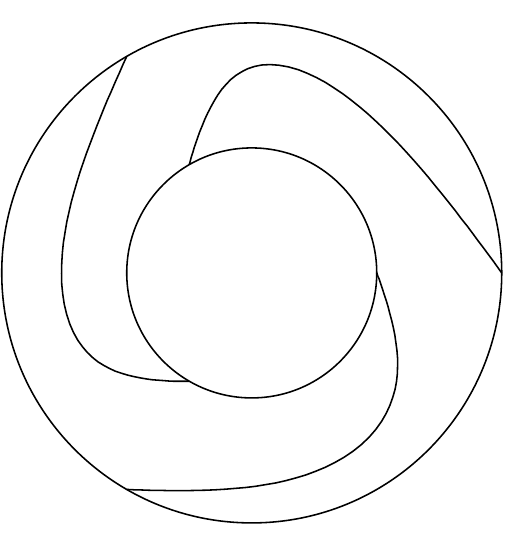}&\scalebox{0.5}{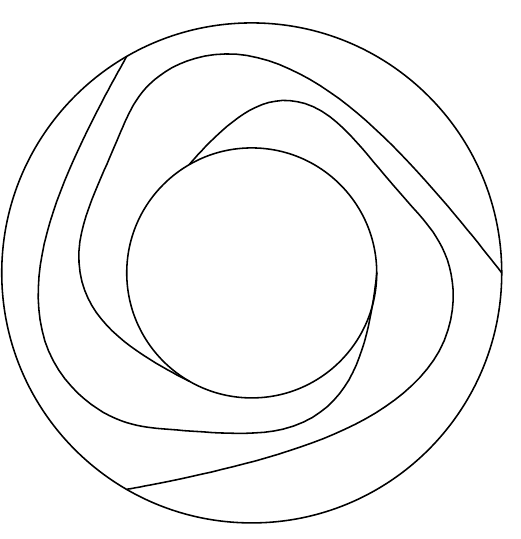}\\
\scalebox{0.5}{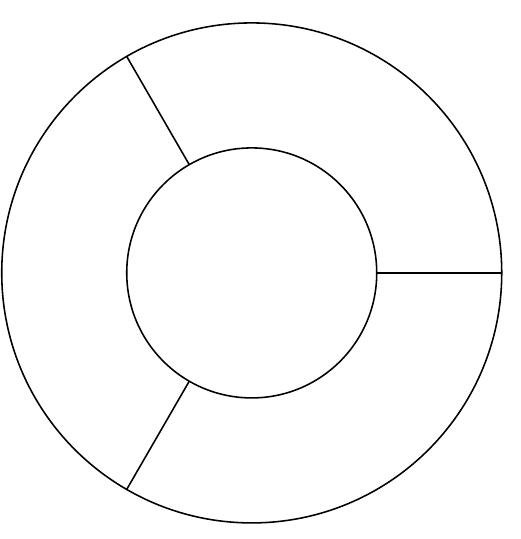}&\scalebox{0.5}{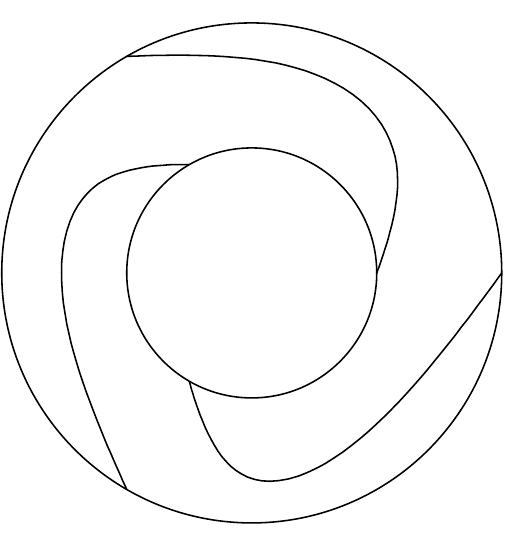}&\scalebox{0.5}{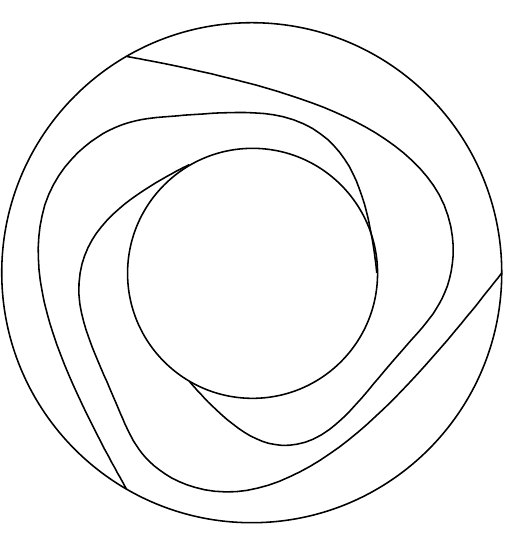}
\end{tabular}
\caption{Spoke diagrams on cycles with three terms.}
\label{spoke}
\end{figure}

We thus propose a definition of real second-order freeness of subalgebras of an abstract second-order probability space \(\left(A,\varphi_{1},\varphi_{2}\right)\) equipped with a linear involution \(a\mapsto a^{t}\) reversing the order of multiplication:
\[\varphi_{2}\left(a_{1}\cdots a_{p},b_{1}\cdots b_{p}\right)=\sum_{k=0}^{p-1}\prod_{i=1}^{p}\varphi_{1}\left(a_{i}b_{k-i}\right)+\sum_{k=0}^{p-1}\prod_{i=1}^{p}\varphi_{1}\left(a_{i}b_{k+i}^{t}\right)\]
(where the \(a_{j}\) and \(b_{j}\) are centred and from cyclically alternating subalgebras, and again all indices are taken modulo \(p\)), and similar quantities vanish when the number of terms in each argument of \(\varphi_{2}\) are different.

In Section~\ref{definitions}, we present the notation and definitions we will be using, including the relevant existing definitions in free probability, asymptotic freeness and second-order probability spaces.  In Section~\ref{ensembles}, we define the three matrix models which we will be considering.  We find that each model can be put into a common form resembling a genus expansion, but in which terms corresponding to both orientable and nonorientable surfaces appear.  We will use this form in the remainder of the paper, in which we will derive results which are applicable to each of the three models or any other model satisfying the same conditions.  In Section~\ref{combinatorics} we present several combinatorial expressions which are exact for all matrix dimensions \(N\).  We derive an expression for expected values of products of traces of several independent matrices and for the cumulants of such traces, and we find a combinatorial characterization of the terms which contribute when the terms are centred.  In Section~\ref{asymptotics}, we consider asymptotic behaviour.  We characterize combinatorially the highest order terms and show that all of the matrix models have second-order limit distributions, meaning that not only the moments but also their fluctuations are well-defined asymptotically.  In Section~\ref{freeness}, we show that the matrix ensembles are asymptotically free under the usual definition.  Using the machinery of \cite{MR2052516}, we show that highest order terms correspond to a certain type of annular-noncrossing diagram (which we generalize to the real-matrix context, in which a larger class of diagrams are possible).  We show that the highest order terms then correspond to a larger set of spoke diagrams, (those for both relative orientations of the two circles), and we define asymptotic real second-order freeness as satisfying such a relation in Section~\ref{real} where we will collect definitions for the real case.

\section{Notation and definitions}
\label{definitions}

\subsection{Combinatorics}

We denote the interval of integers \(\left\{1,\ldots,n\right\}\) by \(\left[n\right]\), and for \(m\leq n\) the interval \(\left\{m,\ldots,n\right\}\) by \(\left[m,n\right]\).

If \(I\) is a set of integers, then we define \(-I:=\left\{-k:k\in I\right\}\).  We define \(\pm I:=I\cup\left(-I\right)\).

\begin{definition}
For a set \(I\), a {\em partition of \(I\)} is a set \(\left\{V_{1},\ldots,V_{k}\right\}\) of subsets of \(I\) called {\em blocks} which are disjoint, nonempty, and have union \(I\).  We denote the set of partitions of \(I\) by \({\cal P}\left(I\right)\), and the set of partitions of \(\left[n\right]\) by \({\cal P}\left(n\right)\).

We define a partial order \(\preceq\) on \({\cal P}\left(n\right)\) by letting \(\pi\preceq\rho\) if each block of \(\pi\) is contained in a block of \(\rho\).

For a set \(I\), we denote the largest partition \(1_{I}:=\left\{I\right\}\).  If \(I=\left[n\right]\), then we let \(1_{n}:=\left\{\left[n\right]\right\}\).

We denote the join of two partitions \(\pi,\rho\in{\cal P}\left(n\right)\) by \(\pi\vee\rho\), where the join \(\pi\vee\rho\in{\cal P}\left(n\right)\) is the smallest partition such that \(\pi,\rho\preceq\pi\vee\rho\).

A partition is a pairing if each of its blocks contains exactly two elements.  We will denote the set of pairings on a set \(I\) by \({\cal P}_{2}\left(I\right)\), and the set of pairings on \(\left[n\right]\) by \({\cal P}_{2}\left(n\right)\).  If \(n\) is odd, then \({\cal P}_{2}\left(n\right)\) is the empty set, and any sum over \({\cal P}_{2}\left(n\right)\) is zero.
\end{definition}

We will denote the group of permutations on a finite set \(I\) by \(S\left(I\right)\), and the group of permutations on \(\left[n\right]\) by \(S_{n}\).

We will often use cycle notation for permutations, in which the permutation maps each element to the next element in the cycle (and the last element in a cycle to the first).  When we multiply permutations, the rightmost acts first, then the next to the left, and so on.  If we conjugate a permutation \(\pi\) by another permutation \(\rho\), then in cycle notation, the resulting permutation \(\rho\pi\rho^{-1}\) has the same cycle structure as \(\pi\), but with each \(k\) replaced by \(\rho\left(k\right)\).

We will also call a permutation a pairing if each of its cycles contains exactly two elements.  We note that a partition in \({\cal P}_{2}\left(I\right)\) uniquely defines a pairing in \(S\left(I\right)\) and {\em vice versa}.

We denote the number of cycles (orbits) of a permutation \(\pi\) by \(\#\left(\pi\right)\).  We will use the same notation to denote the number of orbits of any subgroup of a permutation group.  The subgroup generated by permutations \(\pi_{1},\ldots,\pi_{k}\) will be denoted \(\langle\pi_{1},\ldots,\pi_{k}\rangle\).  We can see that \(\#\left(\pi\right)=\#\left(\pi^{-1}\right)\) and \(\#\left(\pi\right)=\#\left(\rho\pi\rho^{-1}\right)\).

A permutation \(\pi\in S\left(I\right)\) is considered to act trivially on a \(k\notin I\).  The number of orbits of a permutation depends on the implied domain of the permutation, which we will state explicitly if it is not clear from the context.  If \(\pi\in S\left(I\right)\) and \(\rho\in S\left(J\right)\) for \(I\) and \(J\) disjoint, then \(\#\left(\pi\rho\right)=\#\left(\pi\right)+\#\left(\rho\right)\).

We will often consider the orbits of a permutation as a partition, and will use the permutation to denote this partition.

If \(J,K\subseteq I\), we say that a permutation or subgroup of \(S\left(I\right)\) connects \(J\) and \(K\) if it has an orbit which contains at least one element of both \(J\) and \(K\).  We will say that a permutation or subgroup with orbits given by partition \(\pi\) connects the blocks of partition \(\rho\) if \(\pi\vee\rho=1_{I}\).

\begin{definition}
If \(J\subseteq I\) are sets and \(\pi\in S\left(I\right)\), we define the {\em permutation induced on \(J\) by \(\pi\)}, denoted \(\left.\pi\right|_{J}\in S\left(J\right)\), by letting \(\left.\pi\right|_{J}\left(k\right)=\pi^{m}\left(k\right)\), where \(m\) is the smallest positive integer such that \(\pi^{m}\left(k\right)\in J\).  (In cycle notation, this amounts to deleting all elements not in \(J\).)
\end{definition}
While it is not in general true that \(\left(\left.\pi\right|_{J}\right)\left(\left.\rho\right|_{J}\right)=\left.\pi\rho\right|_{J}\), it is true if at least one of \(\pi\) or \(\rho\) does not connect \(J\) and \(I\setminus J\).  We can see that \(\left(\left.\pi\right|_{J}\right)^{-1}=\left.\pi^{-1}\right|_{J}\) and if \(K\subseteq J\), \(\left.\left.\pi\right|_{J}\right|_{K}=\left.\pi\right|_{K}\).

Although an induced permutation \(\left.\pi\right|_{J}\) is defined only on \(J\), we will find it useful in Lemma~\ref{centred} to define \(\left.\pi\right|_{J}\) for all \(k\in I\) in the same manner: \(\left.\pi\right|_{J}\left(k\right)=\pi^{m}\left(k\right)\) where \(m\) is the smallest positive integer such that \(\pi^{m}\left(k\right)\in J\) even if \(k\notin J\).

Throughout, we will let \(\delta:k\mapsto -k\).

\begin{definition}
For a set of nonzero integers \(I\) we call a permutation \(\pi\in S\left(\pm I\right)\) a {\em premap} if \(\pi\left(k\right)=-\pi^{-1}\left(-k\right)\) and no cycle contains both \(k\) and \(-k\) for all \(k\in\pm I\).  We denote the set of premaps on \(\pm I\) by \(PM\left(\pm I\right)\).
\end{definition}

We note that the first condition is equivalent to \(\delta\pi\delta=\pi^{-1}\); and given the first condition, to show the second, it is sufficient to show that there is no \(k\) such that \(\pi\left(k\right)=-k\) (since given a \(k\) with \(\pi^{2m}\left(k\right)=-k\), we must then have \(\pi^{m}\left(k\right)=\pi^{-m}\left(-k\right)=-\pi^{m}\left(k\right)\), which is impossible, and given a \(k\) with \(\pi^{2m+1}\left(k\right)=-k\), we must have \(\pi\pi^{m}\left(k\right)=\pi^{-m}\left(-k\right)=-\pi^{m}\left(k\right)\), which is disallowed).

\begin{remark}
\label{premaps}
The permutation in the above definition is only one component of a premap as defined elsewhere, such as in \cite{MR746795}, in which such permutations are used to describe surfaced hypergraphs.  (Here, a hypergraph is a graph with hyperedges, that is, edges which may connect any positive integer number of vertices, rather than just two, and can be interpreted as vertices which must alternate with the other vertices.)  In a surfaced hypergraph, not only vertices and faces but also hyperedges must have a cyclic order on the graph elements they are connected to.  A premap is the appropriate object to describe the faces, hyperedges, or vertices of an unoriented surface: the conditions ensure the consistency of the description on an orientable two-sheeted covering space (see, for example, \cite{MR1867354}, pages 234--235, for the construction of this two-sheeted cover).  See \cite{MR746795, MR2036721, MR0404045, MR2052516, MR2851244} for more information on how surfaced hypergraphs can be represented as sets of permutations.  This interpretation is often useful for understanding subsequent calculations, but is not necessary to the proofs.  See Remark~\ref{hypergraphs}.
\end{remark}

\begin{definition}
We call a cycle of a permutation \(\pi\in S\left(\pm I\right)\) {\em particular} if its element with smallest absolute value is positive, after \cite{MR2132270}.  We denote by \(\pi/2\) the set of all particular cycles of a premap \(\pi\).  We will use \(\pi/2\) to denote both the set of all elements appearing in particular cycles of \(\pi\) and the permutation on that set by those cycles.
\end{definition}

If \(\pi\in S\left(I\right)\) where \(I\) does not contain both \(k\) and \(-k\) for any \(k\), we denote by \(\pi_{+}\) the permutation \(\pi\) considered as a permutation on \(\pm I\) (where it acts trivially on any element not in \(I\)).  We define \(\pi_{-}:=\delta\pi_{+}\delta\).  Then \(\pi_{+}\pi_{-}^{-1}\) is a premap with \(\#\left(\pi_{+}\pi_{-}^{-1}\right)=2\#\left(\pi\right)\).

We note that the inverse of a premap is a premap.  We also note that the permutation \(\gamma_{-}^{-1}\pi\gamma_{+}\) is a premap:
\begin{lemma}
Let \(\gamma\in S\left(I\right)\) for some set \(I\) that does not contain both \(k\) and \(-k\) for any \(k\), and let \(\pi\in PM\left(\pm I\right)\).  Then \(\gamma_{-}^{-1}\pi\gamma_{+}\in PM\left(\pm I\right)\).
\end{lemma}
\begin{proof}
Firstly, \(\delta\gamma_{-}^{-1}\pi\gamma_{+}\delta=\delta\gamma_{-}^{-1}\delta\delta\pi\delta\delta\gamma_{+}\delta=\gamma_{+}^{-1}\pi^{-1}\gamma_{-}\).

Secondly, if, for any \(k\), \(\gamma_{-}^{-1}\pi\gamma_{+}\left(k\right)=-k\), then
\[\pi\left(\gamma_{+}\left(k\right)\right)=\gamma_{-}\left(-k\right)=-\gamma_{+}\left(k\right)\textrm{.}\]
But since \(\pi\) is a premap, this cannot be the case.
\end{proof}

\subsection{Random variables and matrices}

\begin{definition}
We define the {\em classical cumulants} \(k_{1},k_{2},\ldots\) such that the \(n\)th cumulant \(k_{n}\) is an \(n\)-linear function on random variables satisfying the {\em moment-cumulant formula}, by which they are uniquely defined:
\begin{equation}
\label{moment-cumulant}
\mathbb{E}\left(X_{1}\cdots X_{n}\right)=\sum_{\pi\in{\cal P}\left(n\right)}\prod_{\left\{i_{1},\ldots,i_{m}\right\}\in\pi}k_{m}\left(X_{i_{1}},\ldots,X_{i_{m}}\right)\textrm{.}
\end{equation}
\end{definition}
In particular, the second cumulant is the covariance:
\[k_{2}\left(X,Y\right)=\mathbb{E}\left(XY\right)-\mathbb{E}\left(X\right)\mathbb{E}\left(Y\right)\textrm{.}\]
We also note that if any entry \(X_{i}\) of a cumulant \(k_{n}\left(X_{1},\ldots,X_{n}\right)\) with \(n>1\) is a constant random variable, then the cumulant vanishes.  (Briefly, we have
\begin{eqnarray*}
\mathbb{E}\left(X_{1}\cdots X_{n}\right)&=&k_{1}\left(X_{i}\right)\mathbb{E}\left(X_{1}\cdots X_{i-1}X_{i+1}\cdots X_{n}\right)\\
&=&k_{1}\left(X_{i}\right)\sum_{\pi\in{\cal P}\left(\left[n\right]\setminus\left\{i\right\}\right)}\prod_{\left\{i_{1},\ldots,i_{m}\right\}\in\pi}k_{m}\left(X_{i_{1}},\ldots,X_{i_{m}}\right)\textrm{;}
\end{eqnarray*}
that is, the moment is the sum over terms in which \(i\) is in its own block.  Inductively, we assume that any term in the moment-cumulant formula in which \(X_{i}\) is contained in a \(k_{m}\) with \(1<m<n\) vanishes, so \(k_{n}\left(X_{1},\ldots,X_{n}\right)\) also vanishes.)

We denote the usual trace by \(\mathrm{Tr}\left(X\right):=\sum_{i=1}^{N}X_{ii}\), and the normalized trace by \(\mathrm{tr}:=\frac{1}{N}\mathrm{Tr}\), where \(N\) is the size of the matrix.

When we wish to move the subscript of a matrix \(X_{k}\) to its superscript (such as when we have indices in the subscript), we will put it in brackets, and denote \(X_{k}\) by \(X^{\left(k\right)}\).  We will often find it convenient to denote a matrix \(X\) by \(X^{\left(1\right)}\) and its transpose by \(X^{\left(-1\right)}\).  If we want to combine these notations, we will denote \(X_{k}\) by \(X^{\left(k\right)}\) and \(X_{k}^{T}\) by \(X^{\left(-k\right)}\).

Let \(I=\left\{c_{1},\ldots,c_{m}\right\}\subseteq\pm\left[n\right]\), let \(m_{1}+\cdots+m_{k}=m\), and let \(\pi=\left(c_{1},\ldots,c_{m_{1}}\right)\allowbreak\cdots\allowbreak\left(c_{m_{1}+\cdots+m_{k-1}+1},\ldots,c_{m}\right)\in S\left(I\right)\).  We denote the trace along the cycles of \(\pi\):
\begin{multline*}
\mathrm{Tr}_{\pi}\left(X_{1},\ldots,X_{n}\right)\\:=\mathrm{Tr}\left(X^{\left(c_{1}\right)}\cdots X^{\left(c_{m_{1}}\right)}\right)\cdots\mathrm{Tr}\left(X^{\left(c_{m_{1}+\cdots+m_{k-1}+1}\right)}\cdots X^{\left(c_{m}\right)}\right)\textrm{.}
\end{multline*}
If the domain of \(\pi\) is restricted to \(\pm J\), where \(J\) is a finite set of positive integers \(J=\left\{j_{1},\ldots,j_{m}\right\}\) with \(j_{1}<\ldots<j_{m}\), then we may list in the arguments of the trace only the elements of \(J\) (in order): \(\mathrm{Tr}_{\pi}\left(X_{j_{1}},\ldots,X_{j_{m}}\right)\).  In each case, the domain will be specified if it is not clear from context.  The normalized trace along a permutation is defined analogously.

\subsection{Noncommutative probability spaces and freeness}

\begin{definition}
A {\em noncommutative probability space \(\left(A, \varphi_{1}\right)\)} is an unital algebra \(A\) equipped with a tracial functional \(\varphi_{1}\) (called an expectation) such that \(\varphi_{1}\left(1_{A}\right)=1\).  We will call elements of \(A\) {\em (noncommutative) random variables}.
\end{definition}
When \(X\) is a random matrix, we will let \(\varphi_{1}\left(X\right)=\mathbb{E}\left(\mathrm{tr}\left(X\right)\right)\).

\begin{definition}
We will say that a random variable \(a\in A\) is {\em centred} if \(\varphi_{1}\left(a\right)=0\).  We will denote the centred random variable \(\mathaccent"7017{a}:=a-\varphi_{1}\left(a\right)\).  In particular, if \(X\) is a random matrix, we will denote the centred random matrix by \(\mathaccent"7017{X}:=X-\mathbb{E}\left(\mathrm{tr}\left(X\right)\right)\).
\end{definition}

\begin{definition}
We say that a product of terms \(a_{1},\ldots,a_{p}\) taken from subalgebras \(A_{1},\ldots,A_{n}\subseteq A\) are {\em alternating} if \(a_{i}\in A_{k_{i}}\) and \(k_{1}\neq k_{2}\neq\ldots\neq k_{p}\), and they are {\em cyclically alternating} if, in addition, \(k_{p}\neq k_{1}\).  Similarly, we say that a word \(w:\left[p\right]\rightarrow\left[C\right]\) is alternating if \(w\left(1\right)\neq w\left(2\right)\neq\cdots\neq w\left(p\right)\), and cyclically alternating if in addition \(w\left(p\right)\neq w\left(1\right)\).  
\end{definition}

\begin{definition}
Let \(A_{1},\ldots,A_{n}\subseteq A\) be subalgebras of noncommutative probability space \(A\).  We say that \(A_{1},\ldots,A_{n}\) are {\em free} if
\[\varphi_{1}\left(a_{1},\ldots,a_{p}\right)=0\]
whenever the \(a_{i}\) are centred and alternating.
\end{definition}

Let \(\left(\Omega,\Sigma,\mathbb{P}\right)\) be a probability space.

\begin{definition}
For each colour \(c\in I\) for some index set \(I\) and for \(N=1,2,\ldots\), let \(\left\{X_{c}^{\left(\lambda\right)}:\Omega\rightarrow M_{N\times N}\left(\mathbb{F}\right)\right\}_{\lambda\in\Lambda_{c}}\) be a family of matrices.  (Here \(\mathbb{F}\) may be either \(\mathbb{R}\) or \(\mathbb{C}\).)  We say that the families are {\em asymptotically free} if, for any integer \(p>0\), alternating word \(w:\left[p\right]\rightarrow I\), and \(A_{i}\) in the algebra generated by the \(X_{w\left(i\right)}^{\left(\lambda\right)}\) for \(1\leq i\leq p\), we have
\[\lim_{N\rightarrow\infty}\mathbb{E}\left(\mathrm{tr}\left(\mathaccent"7017{A}_{1}\cdots\mathaccent"7017{A}_{p}\right)\right)=0\textrm{.}\]
\end{definition}

\begin{definition}
A {\em second-order probability space \(\left(A,\varphi_{1},\varphi_{2}\right)\)} is a noncommutative probability space \(\left(A,\varphi_{1}\right)\) equipped with a bilinear functional \(\varphi_{2}\) which is tracial in each argument and such that \(\varphi_{2}\left(1_{A},a\right)=\varphi_{2}\left(a,1_{A}\right)=0\) for all \(a\in A\).
\end{definition}
If \(X,Y\) are random matrices, we will let \(\varphi_{2}\left(X,Y\right)=\mathbb{E}\left(\mathrm{Tr}\left(X\right),\mathrm{Tr}\left(Y\right)\right)\).

\begin{definition}
Subalgebras \(A_{1},\ldots,A_{n}\) of a second-order noncommutative probability space \(\left(A,\varphi_{1},\varphi_{2}\right)\) are {\em complex second-order free} if they are free and
\[\varphi_{2}\left(a_{1}\cdots a_{p},b_{1}\cdots b_{p}\right)=\sum_{k=0}^{p-1}\prod_{i=1}^{p}\varphi_{1}\left(a_{i}b_{k-i}\right)\]
when the \(a_{1},\ldots,a_{p}\) and the \(b_{1},\ldots,b_{p}\) are centred and cyclically alternating, and
\[\varphi_{2}\left(a_{1}\cdots a_{p},b_{1}\cdots b_{q}\right)=0\]
when \(p\neq q\) and the \(a_{1},\ldots,a_{p}\) and the \(b_{1},\ldots,b_{q}\) are centred and either cyclically alternating or consist of a single term.

Here and elsewhere we take indices modulo the range over which they are defined.
\end{definition}

\begin{definition}
We say that random matrices \(\left\{X_{\lambda}:\Omega\rightarrow M_{N\times N}\left(\mathbb{F}\right)\right\}_{\lambda\in\Lambda}\) have a {\em second-order limit distribution} if there exists a second-order probability space \(\left(A,\varphi_{1},\varphi_{2}\right)\) with \(x_{\lambda}\in A\), \(\lambda\in\Lambda\), such that for \(i=1,2,\ldots\), any nonnegative integers \(n_{1},n_{2},\ldots\) and any polynomials \(p_{i}\) in \(n_{i}\) noncommuting variables (\(i=1,2,\ldots\)), and any \(\lambda_{i,1},\ldots,\lambda_{i,n}\in\Lambda\), we have
\[\lim_{N\rightarrow\infty}k_{1}\left(\mathrm{tr}\left(p_{1}\left(X_{\lambda_{1,1}},\ldots,X_{\lambda_{1,n_{1}}}\right)\right)\right)=\varphi_{1}\left(p_{1}\left(x_{\lambda_{1,1}},\ldots,x_{\lambda_{1,n_{1}}}\right)\right)\textrm{,}\]
\begin{multline*}
\lim_{N\rightarrow\infty}k_{2}\left(\mathrm{Tr}\left(p_{1}\left(X_{\lambda_{1,1}},\ldots,X_{\lambda_{1,n_{1}}}\right)\right),\mathrm{Tr}\left(p_{2}\left(X_{\lambda_{2,1}},\ldots,X_{\lambda_{2,n_{2}}}\right)\right)\right)\\=\varphi_{2}\left(p_{1}\left(x_{\lambda_{1,1}},\ldots,x_{\lambda_{1,n_{1}}}\right),p_{2}\left(x_{\lambda_{2,1}},\ldots,x_{\lambda_{2,n_{2}}}\right)\right)\textrm{,}
\end{multline*}
and for all \(r\geq 3\),
\[\lim_{N\rightarrow\infty}k_{r}\left(\mathrm{Tr}\left(p_{1}\left(X_{\lambda_{1,1}},\ldots,X_{\lambda_{1,n_{1}}}\right)\right),\ldots,\mathrm{Tr}\left(p_{r}\left(X_{\lambda_{r,1}},\ldots,X_{\lambda_{r,n_{r}}}\right)\right)\right)=0\textrm{.}\]

In addition, if there are \(\lambda\in\Lambda\) with \(-\lambda\in\Lambda\), then we require the existence of an involution \(x\mapsto x^{t}\) on \(A\) reversing the order of multiplication such that \(x_{-\lambda}=x_{\lambda}^{t}\) for any such \(\lambda\).
\end{definition}

\begin{definition}
For each colour \(c\in I\), let \(\left\{X^{\left(\lambda\right)}_{c}:\Omega\rightarrow M_{N\times N}\left(\mathbb{F}\right)\right\}_{\lambda\in\Lambda_{c}}\) be an ensemble of random matrices.  Let \(v:\left[p\right]\rightarrow I\) and \(w:\left[q\right]\rightarrow I\) be words in the colours.  Let \(A_{k}\) be a matrix in the algebra generated by the \(X_{v\left(k\right)}^{\left(\lambda\right)}\) and let \(B_{k}\) be a matrix in the algebra generated by the \(X_{w\left(k\right)}^{\left(\lambda\right)}\).  Then we say that the ensembles are {\em asymptotically complex second-order free} if they are free, have a second-order limit distribution, and the algebras generated by the elements of the second-order limit distribution are second-order free; that is, for \(v\) and \(w\) cyclically alternating and \(p=q\geq 2\),
\begin{multline*}
\lim_{N\rightarrow\infty}k_{2}\left(\mathrm{Tr}\left(\mathaccent"7017{A}_{1}\cdots\mathaccent"7017{A}_{p}\right),\mathrm{Tr}\left(\mathaccent"7017{B}_{1}\cdots\mathaccent"7017{B}_{p}\right)\right)\\=\sum_{k=0}^{p-1}\prod_{i=1}^{p}\lim_{N\rightarrow\infty}\left(\mathbb{E}\left(\mathrm{tr}\left(A_{i}B_{k-i}\right)\right)-\mathbb{E}\left(\mathrm{tr}\left(A_{i}\right)\right)\mathbb{E}\left(\mathrm{tr}\left(B_{k-i}\right)\right)\right)
\end{multline*}
and whenever \(p\neq q\) and \(v\) and \(w\) are either cyclically alternating or have length \(1\),
\[\lim_{N\rightarrow\infty}k_{2}\left(\mathrm{Tr}\left(\mathaccent"7017{A}_{1}\cdots\mathaccent"7017{A}_{p}\right),\mathrm{Tr}\left(\mathaccent"7017{B}_{1}\cdots\mathaccent"7017{B}_{q}\right)\right)=0\textrm{.}\]
\end{definition}

We collect the analogous definitions for real second-order probability in Section~\ref{real}.

\section{The matrix models}
\label{ensembles}

In this section, we present formulas for the expected values of products of traces of matrices from the three models we are considering.  Quantities involving several independent matrices can be calculated using Lemma~\ref{colour}.

We will use several lemmas to derive the formulas in this section, which we state here.  The following is a folklore result which can be proven by direct calculation:
\begin{lemma}
\label{trace}
Let \(\pi\in S\left(\left\{n_{1},\ldots,n_{m}\right\}\right)\), \(n_{1}<\cdots<n_{m}<n\), and let \(X_{n_{k}}\) be an \(N_{n_{k}}\times N_{\pi\left(n_{k}\right)}\) matrix for each \(k\).  Then
\[\mathrm{Tr}_{\pi}\left(X_{n_{1}},\ldots,X_{n_{m}}\right)=\sum_{\substack{1\leq i_{n_{k}}\leq N_{n_{k}}\\k=1,\ldots,m}}\prod_{k=1}^{m}X^{\left(n_{k}\right)}_{i_{n_{k}}i_{\pi\left(n_{k}\right)}}\textrm{.}\]
\end{lemma}
If \(N_{k}=N\) for all \(k\in I\), we can interpret this as a sum over all functions \(i:\left\{n_{1},\ldots,n_{m}\right\}\rightarrow\left[N\right]\) with \(k\mapsto i_{k}\).

The following lemma is known as the Wick Formula.  A proof can be found in \cite{MR2036721}, p. 164.

\begin{lemma}[Wick]
\label{Wick}
Let \(\left\{f_{\lambda}\right\}_{\lambda\in\Lambda}\) be components of a real Gaussian family of random variables for some index set \(\Lambda\).  Then, for \(\lambda_{1},\ldots,\lambda_{n}\in\Lambda\),
\[\mathbb{E}\left(f_{\lambda_{1}}\cdots f_{\lambda_{n}}\right)=\sum_{\pi\in{\cal P}_{2}\left(n\right)}\prod_{\left\{k,l\right\}\in\pi}\mathbb{E}\left(f_{\lambda_{k}}f_{\lambda_{l}}\right)\textrm{.}\]
\end{lemma}

We will find the following expression useful.  For \(\varepsilon:\left[n\right]\rightarrow\left\{1,-1\right\}\), we define \(\varepsilon\) on \(-\left[n\right]\) by \(\varepsilon\left(-k\right)=\varepsilon\left(k\right)\), and we define
\[\delta_{\varepsilon}:k\mapsto\varepsilon\left(k\right)k\textrm{.}\]
\begin{lemma}
\label{lemma: pairing}
The cycles of \(\delta_{\varepsilon}\pi\delta\pi\delta_{\varepsilon}\) are of the form \(\left(k,-\varepsilon\left(k\right)\varepsilon\left(l\right)l\right)\) and \(\left(-k,\varepsilon\left(k\right)\varepsilon\left(l\right)l\right)\), where \(\left(k,l\right)\) is a cycle of \(\pi\).  It is thus a premap.
\end{lemma}
\begin{proof}
The permutation \(\delta\) conjugated by \(\pi\) consists of cycles of the form \(\left(k,-l\right)\), where \(\left(k,l\right)\) is a cycle of \(\pi\).  Conjugating by \(\delta_{\varepsilon}\) may interchange \(k\) and \(-k\) as well as \(l\) and \(-l\).  If \(\varepsilon\left(k\right)=\varepsilon\left(l\right)\), then the cycle is unchanged; otherwise, the cycles become \(\left(k,l\right)\) and \(\left(-k,-l\right)\) (so no cycle contains both \(k\) and \(-k\), and hence it is a premap).
\end{proof}

We will also use the following lemma, outlined in \cite{MR2851244}.  We give a complete proof here.

\begin{lemma}
\label{genus}
Let \(X\) be an \(M\times N\) matrix with entries \(\frac{1}{\sqrt{N}}f_{ij}\), where the \(f_{ij}\) are independent \(N\left(0,1\right)\) random variables, and let \(Y_{1},\ldots Y_{n}\) be random matrices independent from \(X\) and appropriately sized so the matrix multiplication below is defined.  (We will take the normalized trace to be \(\frac{1}{N}\) times the usual trace regardless of whether it is applied to an \(N\times N\) matrix or an \(M\times M\) matrix; we will assume that \(M\) is of the order of \(N\).)

Let \(\gamma\in S_{n}\).  Then
\begin{multline*}
\mathbb{E}\left(\mathrm{tr}_{\gamma}\left(X^{\left(\varepsilon\left(1\right)\right)}Y_{1},\cdots,X^{\left(\varepsilon\left(n\right)\right)}Y_{n}\right)\right)
\\=\sum_{\pi\in{\cal P}_{2}\left(n\right)}N^{\#\left(\gamma_{-}^{-1}\delta_{\varepsilon}\pi\delta\pi\delta_{\varepsilon}\gamma_{+}\right)/2-\#\left(\gamma\right)-n/2}\\\times\mathbb{E}\left(\mathrm{tr}_{\gamma_{-}^{-1}\delta_{\varepsilon}\pi\delta\pi\delta_{\varepsilon}\gamma_{+}/2}\left(Y_{1},\ldots,Y_{n}\right)\right)\textrm{.}
\end{multline*}
\end{lemma}
\begin{proof}
To allow for the arbitrary occurrences of transposed copies of \(X\) encoded in \(\varepsilon\), in place of the usual indices \(i_{k}\) and \(j_{k}\) we will use indices \(\iota^{+}_{k}\) and \(\iota^{-}_{k}\).  We choose our indices so that the \(k\)th occurrence of the matrix \(X\) will have indices \(\iota^{+}_{k}\iota^{-}_{k}\):
\begin{eqnarray*}
&&\mathbb{E}\left(\mathrm{tr}_{\gamma}\left(X^{\left(\varepsilon\left(1\right)\right)}Y_{1},\cdots,X^{\left(\varepsilon\left(n\right)\right)}Y_{n}\right)\right)
\\&=&\sum_{\substack{\iota^{+}:\left[n\right]\rightarrow\left[N\right]\\\iota^{-}:\left[n\right]\rightarrow\left[M\right]}}N^{-\#\left(\gamma\right)}\mathbb{E}\left(X^{\left(\varepsilon\left(1\right)\right)}_{\iota^{\varepsilon\left(1\right)}_{1}\iota^{-\varepsilon\left(1\right)}_{1}}Y^{\left(1\right)}_{\iota^{-\varepsilon\left(1\right)}_{1}\iota^{\varepsilon\left(\gamma\left(1\right)\right)}_{\gamma\left(1\right)}}\cdots\right.\\&&\left.X^{\left(\varepsilon\left(n\right)\right)}_{\iota^{\varepsilon\left(n\right)}_{n}\iota^{-\varepsilon\left(n\right)}_{n}}Y^{\left(n\right)}_{\iota^{-\varepsilon\left(n\right)}_{n}\iota^{\varepsilon\left(\gamma\left(n\right)\right)}_{\gamma\left(n\right)}}\right)
\\&=&\sum_{\substack{\iota^{+}:\left[n\right]\rightarrow\left[N\right]\\\iota^{-}:\left[n\right]\rightarrow\left[M\right]}}N^{-\#\left(\gamma\right)-n/2}\mathbb{E}\left(Y^{\left(1\right)}_{\iota^{-\varepsilon\left(1\right)}_{1}\iota^{\varepsilon\left(\gamma\left(1\right)\right)}_{\gamma\left(1\right)}}\cdots Y^{\left(n\right)}_{\iota^{-\varepsilon\left(n\right)}_{n}\iota^{\varepsilon\left(\gamma\left(n\right)\right)}_{\gamma\left(n\right)}}\right)\\&&\times\mathbb{E}\left(f_{\iota^{+}_{1}\iota^{-}_{1}}\cdots f_{\iota^{+}_{n}\iota^{-}_{n}}\right)\textrm{.}
\end{eqnarray*}

We now apply Lemma~\ref{Wick} to the expected value expression:
\[\mathbb{E}\left(f_{\iota^{+}_{1}\iota^{-}_{1}}\cdots f_{\iota^{+}_{n}\iota^{-}_{n}}\right)=\sum_{\pi\in{\cal P}_{2}\left(n\right)}\prod_{\left\{k,l\right\}\in\pi}\mathbb{E}\left(f_{\iota^{+}_{k}\iota^{-}_{k}}f_{\iota^{+}_{l}\iota^{-}_{l}}\right)\textrm{.}\]
If \(\iota^{+}_{k}\neq\iota^{+}_{l}\) or \(\iota^{-}_{k}\neq\iota^{-}_{l}\) (that is, if \(\iota^{+}\left(k\right)\neq\iota^{+}\pi\left(k\right)\) or \(\iota^{-}\left(k\right)\neq\iota^{-}\pi\left(k\right)\)) then \(\mathbb{E}\left(f_{\iota^{+}_{k}\iota^{-}_{k}}f_{\iota^{+}_{l}\iota^{-}_{l}}\right)=0\), so a product including this term will vanish.  Thus the term corresponding to a given pairing \(\pi\) is
\[\prod_{\left\{k,l\right\}\in\pi}\mathbb{E}\left(f_{\iota^{+}_{k}\iota^{-}_{k}}f_{\iota^{+}_{l}\iota^{-}_{l}}\right)=\left\{\begin{array}{ll}1,&\iota^{\pm}=\iota^{\pm}\circ\pi\\0,&\textrm{otherwise}\end{array}\right.\textrm{.}\]

Inserting this back into our expression, we get:
\[\sum_{\substack{\iota^{+}:\left[n\right]\rightarrow\left[N\right]\\\iota^{-}:\left[n\right]\rightarrow\left[M\right]}}N^{-\#\left(\gamma\right)-n/2}\mathbb{E}\left(Y^{\left(1\right)}_{\iota^{-\varepsilon\left(1\right)}_{1}\iota^{\varepsilon\left(\gamma\left(1\right)\right)}_{\gamma\left(1\right)}}\cdots Y^{\left(n\right)}_{\iota^{-\varepsilon\left(n\right)}_{n}\iota^{\varepsilon\left(\gamma\left(n\right)\right)}_{\gamma\left(n\right)}}\right)\sum_{\pi\in{\cal P}_{2}\left(n\right):\iota^{\pm}\circ\pi=\iota^{\pm}}1\textrm{.}\]
Reversing the order of summation, we get:
\[=\sum_{\pi\in{\cal P}_{2}}\sum_{\substack{\iota^{+}:\left[n\right]\rightarrow\left[N\right]\\\iota^{-}:\left[n\right]\rightarrow\left[M\right]\\\iota^{\pm}\circ\pi=\iota^{\pm}}}N^{-\#\left(\gamma\right)-n/2}\mathbb{E}\left(Y^{\left(1\right)}_{\iota^{-\varepsilon\left(1\right)}_{1}\iota^{\varepsilon\left(\gamma\left(1\right)\right)}_{\gamma\left(1\right)}}\cdots Y^{\left(n\right)}_{\iota^{-\varepsilon\left(n\right)}_{n}\iota^{\varepsilon\left(\gamma\left(n\right)\right)}_{\gamma\left(n\right)}}\right)\textrm{.}\]

By Lemma~\ref{lemma: pairing}, \(\delta_{\varepsilon}\pi\delta\pi\delta_{\varepsilon}\) and therefore \(\gamma_{-}^{-1}\delta_{\varepsilon}\pi\delta\pi\delta_{\varepsilon}\gamma_{+}\) are premaps.  We give an expression for the element of \(Y_{k}\) which appears, regardless of the sign appearing on \(k\) in \(\gamma_{-}^{-1}\delta_{\varepsilon}\pi\delta\pi\delta_{\varepsilon}\gamma_{+}\): for \(k>0\),
\[Y^{\left(k\right)}_{\iota^{-\mathrm{sgn}\left(k\right)\varepsilon\left(\gamma_{-}\left(k\right)\right)}_{\left|\gamma_{-}\left(k\right)\right|}\iota^{\mathrm{sgn}\left(k\right)\varepsilon\left(\gamma_{+}\left(k\right)\right)}_{\left|\gamma_{+}\left(k\right)\right|}}=Y^{\left(k\right)}_{\iota^{-\varepsilon\left(k\right)}_{k}\iota^{\varepsilon\left(\gamma\left(k\right)\right)}_{\gamma\left(k\right)}}\]
and
\[Y^{\left(-k\right)}_{\iota^{-\mathrm{sgn}\left(-k\right)\varepsilon\left(\gamma_{-}\left(-k\right)\right)}_{\left|\gamma_{-}\left(-k\right)\right|}\iota^{\mathrm{sgn}\left(-k\right)\varepsilon\left(\gamma_{+}\left(-k\right)\right)}_{\left|\gamma_{+}\left(-k\right)\right|}}=Y^{\left(-k\right)}_{\iota^{\varepsilon\left(\gamma\left(k\right)\right)}_{\gamma\left(k\right)}\iota^{-\varepsilon\left(k\right)}_{k}}=Y^{\left(k\right)}_{\iota^{-\varepsilon\left(k\right)}_{k}\iota^{\varepsilon\left(\gamma\left(k\right)\right)}_{\gamma\left(k\right)}}\textrm{.}\]
Since exactly one of \(k\) and \(-k\) appears in the permutation \(\gamma_{-}^{-1}\delta_{\varepsilon}\pi\delta\pi\delta_{\varepsilon}\gamma_{+}/2\), we can write the expression:
\[\sum_{\pi\in{\cal P}_{2}}\sum_{\substack{\iota^{+}:\left[n\right]\rightarrow\left[N\right]\\\iota^{-}:\left[n\right]\rightarrow\left[M\right]\\\iota^{\pm}\circ\pi=\iota^{\pm}}}N^{-\#\left(\gamma\right)-n/2}\mathbb{E}\left(\prod_{k\in\gamma_{-}^{-1}\delta_{\varepsilon}\pi\delta\pi\delta_{\varepsilon}\gamma_{+}/2}Y^{\left(k\right)}_{\iota^{-\mathrm{sgn}\left(k\right)\varepsilon\left(\gamma_{-}\left(k\right)\right)}_{\left|\gamma_{-}\left(k\right)\right|}\iota^{\mathrm{sgn}\left(k\right)\varepsilon\left(\gamma_{+}\left(k\right)\right)}_{\left|\gamma_{+}\left(k\right)\right|}}\right)\textrm{.}\]

We show that the condition \(\iota^{\pm}\circ\pi=\iota^{\pm}\) always pairs a first index with a second index: specifically, the second index of \(Y_{k}\) is paired with the first index of \(Y_{\gamma_{-}^{-1}\delta_{\varepsilon}\pi\delta\pi\delta_{\varepsilon}\gamma_{+}\left(k\right)}\), which must also appear, since \(\gamma_{-}^{-1}\delta_{\varepsilon}\pi\delta\pi\delta_{\varepsilon}\gamma_{+}\left(k\right)\) appears in the same (particular) cycle of \(\gamma_{-}^{-1}\delta_{\varepsilon}\pi\delta\pi\delta_{\varepsilon}\gamma_{+}\).  The second index of \(Y_{k}\) is \(\iota^{\mathrm{sgn}\left(k\right)\varepsilon\left(\gamma_{+}\left(k\right)\right)}_{\left|\gamma_{+}\left(k\right)\right|}\), which is paired with \(\iota^{\mathrm{sgn}\left(k\right)\varepsilon\left(\gamma_{+}\left(k\right)\right)}_{\pi\left(\left|\gamma_{+}\left(k\right)\right|\right)}\).  By Lemma~\ref{lemma: pairing}, \(\delta_{\varepsilon}\pi\delta\pi\delta_{\varepsilon}\) contains the cycle
\[\left(\gamma_{+}\left(k\right),-\mathrm{sgn}\left(\gamma_{+}\left(k\right)\right)\varepsilon\left(\gamma_{+}\left(k\right)\right)\varepsilon\left(\pi\left(\left|\gamma_{+}\left(k\right)\right|\right)\right)\pi\left(\left|\gamma_{+}\left(k\right)\right|\right)\right)\textrm{,}\]
so \(\pi\left(\left|\gamma_{+}\left(k\right)\right|\right)=\left|\gamma_{-}\gamma_{-}^{-1}\delta_{\varepsilon}\pi\delta\pi\delta_{\varepsilon}\gamma_{+}\left(k\right)\right|\) and (since the \(\gamma_{\pm}\) do not change the sign) \(\mathrm{sgn}\left(\gamma_{-}^{-1}\delta_{\varepsilon}\pi\delta\pi\delta_{\varepsilon}\gamma_{+}\left(k\right)\right)=-\mathrm{sgn}\left(k\right)\varepsilon\left(\gamma_{+}\left(k\right)\right)\allowbreak\varepsilon\left(\delta_{\varepsilon}\pi\delta\pi\delta_{\varepsilon}\gamma_{+}\left(k\right)\right)\).  Since each \(\iota^{\pm}_{k}\) appears at most once in the expression, there are no further constraints.

Renaming the first indices and replacing the second indices with the index they are constrained to be equal to, \(Y_{k}\) appears with indices \(i_{k}\) and \(i_{\gamma_{-}^{-1}\delta_{\varepsilon}\pi\delta\pi\delta_{\varepsilon}\gamma_{+}\left(k\right)}\).  We recognize the resulting sum as a product over the cycles of \(\gamma_{-}^{-1}\delta_{\varepsilon}\pi\delta\pi\delta_{\varepsilon}\gamma_{+}/2\) as in Lemma~\ref{trace}, and expressed in terms of the normalized trace, the result follows.
\end{proof}

\begin{remark}
\label{hypergraphs}
It is possible to interpret the summands in these and subsequent calculations as unoriented surfaced hypergraphs (see Remark~\ref{premaps}).  Briefly, the cycles of \(\gamma\) can be thought of as encoding face information, and \(\gamma_{+}\gamma_{-}^{-1}\) face information on the two-sheeted covering space which allows for a consistent orientation.  The cycles of a premap \(\pi\) encode hyperedge information on this covering space.  Vertex information on the covering space is given by the permutation \(\gamma_{+}\pi^{-1}\gamma_{-}^{-1}\), and the traces of the \(Y_{k}\) matrices are given by its inverse.  See \cite{MR1492512, MR2036721, MR2052516, MR2851244} for examples of how matrix integrals may be interpreted in terms of surfaced hypergraphs.  The definition of the Euler characteristic of \(\gamma\) and \(\pi\) which follows is the natural one in this interpretation (\cite{MR0404045} gives a compatible definition of genus).
\end{remark}

\begin{definition}
Let \(I\) be a finite set of integers which does not contain both \(k\) and \(-k\) for any \(k\).  For a \(\gamma\in S\left(I\right)\) and a premap \(\pi\in PM\left(\pm I\right)\), we define
\[\chi\left(\gamma,\pi\right):=\#\left(\gamma_{+}\gamma_{-}^{-1}\right)/2+\#\left(\pi\right)/2+\#\left(\gamma_{+}^{-1}\pi^{-1}\gamma_{-}\right)/2-\left|I\right|\textrm{.}\]
\end{definition}
We note that if \(\pm I_{1}\) and \(\pm I_{2}\) are disjoint, and \(\gamma_{i}\in S\left(I_{i}\right)\) and \(\pi_{i}\in PM\left(\pm I_{i}\right)\) for \(i=1,2\), then
\[\chi\left(\gamma_{1},\pi_{1}\right)+\chi\left(\gamma_{2},\pi_{2}\right)=\chi\left(\gamma_{1}\gamma_{2},\pi_{1}\pi_{2}\right)\textrm{.}\]

\subsection{Real Ginibre matrices}

\begin{definition}
Let \(f_{ij}\) be independent \(N\left(0,1\right)\) random variables, for \(1\leq i,j\leq N\).  Let \(Z:\Omega\rightarrow M_{N\times N}\left(\mathbb{R}\right)\) be a matrix-valued random variable such that the \(ij\)th entry of \(Z\) is \(\frac{1}{\sqrt{N}}f_{ij}\).  Then \(Z\) is a {\em real Ginibre matrix} (see \cite{MR0173726}).
\end{definition}

\begin{lemma}
If \(Z\) is a real Ginibre matrix, \(\gamma\in S_{n}\), \(\varepsilon:\left[n\right]\rightarrow\left\{1,-1\right\}\), and \(Y_{1},\ldots,Y_{n}\) are random matrices independent from \(Z\), then
\begin{multline*}
\mathbb{E}\left(\mathrm{tr}_{\gamma}\left(Z^{\left(\varepsilon\left(1\right)\right)}Y_{1},\ldots,Z^{\left(\varepsilon\left(n\right)\right)}Y_{n}\right)\right)\\=\sum_{\pi\in\left\{\rho\delta\rho:\rho\in{\cal P}_{2}\left(n\right)\right\}}N^{\chi\left(\gamma,\delta_{\varepsilon}\pi\delta_{\varepsilon}\right)-2\#\left(\gamma\right)}\mathbb{E}\left(\mathrm{tr}_{\gamma_{-}^{-1}\delta_{\varepsilon}\pi\delta_{\varepsilon}\gamma_{+}/2}\left(Y_{1},\ldots,Y_{n}\right)\right)\textrm{.}
\end{multline*}
\end{lemma}
\begin{proof}
By Lemma~\ref{lemma: pairing}, \(\delta_{\varepsilon}\pi\delta_{\varepsilon}=\delta_{\varepsilon}\rho\delta\rho\delta_{\varepsilon}\) is a pairing, so \(\#\left(\delta_{\varepsilon}\pi\delta_{\varepsilon}\right)=n\), and the lemma follows immediately from Lemma~\ref{genus}.
\end{proof}

\subsection{Gaussian orthogonal matrices}

\begin{definition}
Let \(X:\Omega\rightarrow M_{N\times N}\left(\mathbb{R}\right)\) be a random matrix such that \(X_{ij}=\frac{1}{\sqrt{N}}f_{ij}\), where the \(f_{ij}\) are independent \(N\left(0,1\right)\) random variables.  Then a {Gaussian orthogonal matrix} (or {\em Gaussian orthogonal ensemble} or {\em GOE matrix}) is a matrix \(T:=\frac{1}{\sqrt{2}}\left(X+X^{T}\right)\).  (This definition is equivalent, up to normalization, to more standard definitions given in, e.g., \cite{MR2129906}.  See e.g. \cite{MR1677884, MR2036721} for a demonstration of the equivalence in the complex case; the real case is similar.)
\end{definition}

\begin{lemma}
If \(T\) is a GOE matrix, \(\gamma\in S_{n}\), and \(Y_{1},\ldots,Y_{n}\) are random matrices independent from \(T\), then
\begin{multline*}
\mathbb{E}\left(\mathrm{tr}_{\gamma}\left(TY_{1},\ldots,TY_{n}\right)\right)
\\=\sum_{\pi\in PM\left(\pm\left[n\right]\right)\cap{\cal P}_{2}\left(\pm\left[n\right]\right)}N^{\chi\left(\gamma,\pi\right)-2\#\left(\gamma\right)}\mathbb{E}\left(\mathrm{tr}_{\gamma_{-}^{-1}\pi\gamma_{+}/2}\left(Y_{1},\ldots,Y_{n}\right)\right)\textrm{.}
\end{multline*}
\end{lemma}
\begin{proof}
We expand the left-hand side in terms of \(X\) and \(X^{T}\) and apply Lemma~\ref{genus} to each summand:
\begin{eqnarray*}
&&\mathbb{E}\left(\mathrm{tr}_{\gamma}\left(TY_{1},\cdots,TY_{n}\right)\right)
\\&=&\frac{1}{2^{n/2}}\sum_{\varepsilon:\left[n\right]\rightarrow\left\{1,-1\right\}}\mathbb{E}\left(\mathrm{tr}_{\gamma}\left(X^{\left(\varepsilon\left(1\right)\right)}Y_{1},\cdots,X^{\left(\varepsilon\left(n\right)\right)}Y_{n}\right)\right)
\\&=&\frac{1}{2^{n/2}}\sum_{\varepsilon:\left[n\right]\rightarrow\left\{1,-1\right\}}\sum_{\rho\in{\cal P}_{2}\left(n\right)}N^{\#\left(\gamma_{-}^{-1}\delta_{\varepsilon}\rho\delta\rho\delta_{\varepsilon}\gamma_{+}\right)/2-\#\left(\gamma_{+}\gamma_{-}^{-1}\right)/2-n/2}\\&&\times\mathbb{E}\left(\mathrm{tr}_{\gamma_{-}^{-1}\delta_{\varepsilon}\rho\delta\rho\delta_{\varepsilon}\gamma_{+}/2}\left(Y_{1},\ldots,Y_{n}\right)\right)\textrm{.}
\end{eqnarray*}

We now show that the map \(\left(\rho,\varepsilon\right)\mapsto\delta_{\varepsilon}\rho\delta\rho\delta_{\varepsilon}\) is \(2^{n/2}\)-to-one and its image is \(PM\left(\pm\left[n\right]\right)\cap{\cal P}_{2}\left(\pm \left[n\right]\right)\).  By Lemma~\ref{lemma: pairing}, \(\delta_{\varepsilon}\rho\delta\rho\delta_{\varepsilon}\) is a pairing and a premap.  For any \(\pi\in PM\left(\pm\left[n\right]\right)\cap{\cal P}_{2}\left(\pm\left[n\right]\right)\), we can find a \(\rho\in{\cal P}_{2}\left(n\right)\) and an appropriate \(\varepsilon\) such that we construct \(\pi\) in this manner.  In fact, we have two choices for values of \(\varepsilon\) on the elements of each cycle of \(\pi\), or \(2^{n/2}\) such \(\varepsilon\).

Since \(\pi\) is a pairing, \(\#\left(\pi\right)=n\), giving us the correct exponent on \(N\), and the result follows.
\end{proof}

\subsection{Real Wishart matrices}

\begin{definition}
\label{definition: Wishart}
Let \(\left\{f_{ij}\right\}\) be independent \(N\left(0,1\right)\) random variables, for \(1\leq i\leq M\) and \(1\leq j\leq N\).  Let \(X:\Omega\rightarrow M_{M\times N}\left(\mathbb{R}\right)\) be a matrix-valued random variable such that the \(ij\)th entry of \(X\) is \(\frac{1}{\sqrt{N}}f_{ij}\).  Then a {\em real Wishart matrix} is a matrix \(W=X^{T}DX\) for some \(M\times M\) matrix \(D\).  We will assume that \(M\) is of the order of \(N\), and each \(D\) is part of a family \(D_{\lambda}^{\left(M\right)}\in M_{M\times M}\left(\mathbb{C}\right)\), \(\lambda\in\Lambda\) such that \(\lim_{N\rightarrow\infty}\mathrm{tr}\left(D_{\lambda_{1}}^{\left(M\right)}\cdots D_{\lambda_{n}}^{\left(M\right)}\right)\) exists for all \(n\) and all tuples \(\lambda_{1},\ldots,\lambda_{n}\).
\end{definition}

\begin{remark}
Wishart matrices were first studied in \cite{Wishart}, and are variously defined in such papers as \cite{MR1990662, HSS, MR2052516, MR2132270, MR2439565, MR2480549}.  They are generally defined as matrices of the form \(AX^{T}BXC\) with various requirements on the constant matrices \(A\), \(B\), and \(C\), such as that they must be symmetric, positive definite, or the identity matrix, or various relationships between them.  Here we require \(A\) and \(C\) to be the identity matrix, but we do not put any requirements on \(B\).  The possible asymmetry of \(B\) and therefore the Wishart matrices motivates the appearances of the linear involution in the main definition \ref{real second-order freeness}, and the possibility of different \(B\) matrices motivates the order in which terms appear.
\end{remark}

The following lemma expresses the moments of real Wishart matrices in a form which we will find convenient.  The pairings on \(\left[2n\right]\) which appear in Lemma~\ref{genus} are replaced by premaps on \(\pm\left[n\right]\) under a bijection in which each pair of adjacent elements is compressed into one element.  This is possible because the transposes appear in a regular pattern; in some sense, sign takes the place of parity in this expression.  Geometrically, this calculation can be represented as in \cite{MR2851244}, except we consider the alternate vertices which here contain the \(Y_{k}\) matrices as hyperedges, encoded in the permutation \(\pi\), and the vertices induced in this unoriented hypermap are the vertices containing the \(Y_{k}\) matrices.  An example with a diagram is given after the proof.

\begin{lemma}
\label{Wishart}
If \(\left\{W_{\lambda}:\Omega\rightarrow M_{N\times N}\left(\mathbb{C}\right)\right\}_{\lambda\in\Lambda}\) are real Wishart matrices such that \(W_{\lambda_{k}}=X^{T}D_{\lambda_{k}}X\) for \(1\leq k\leq n\), \(\gamma\in S_{n}\) and \(Y_{1},\ldots,Y_{n}\) are random matrices independent from the \(W_{\lambda_{k}}\), then
\begin{multline*}
\mathbb{E}\left(\mathrm{tr}_{\gamma}\left(W_{\lambda_{1}}Y_{1},\cdots,W_{\lambda_{n}}Y_{n}\right)\right)\\=\sum_{\pi\in PM\left(\pm\left[n\right]\right)}N^{\chi\left(\gamma,\pi\right)-2\#\left(\gamma\right)}\mathrm{tr}_{\pi^{-1}}\left(D_{\lambda_{1}},\ldots,D_{\lambda_{n}}\right)\\\times\mathbb{E}\left(\mathrm{tr}_{\gamma_{-}^{-1}\pi\gamma_{+}/2}\left(Y_{1},\ldots,Y_{n}\right)\right)\textrm{.}
\end{multline*}
\end{lemma}
\begin{proof}
In order to put the left-hand-side of our expression into the form of Lemma~\ref{genus}, we define two bijections from the indices \(\left[n\right]\) to the new indices of the \(X^{T}\) and \(X\) terms (respectively) that appear when each Wishart term is expanded:
\[\theta_{-}:k\mapsto 2k-1\textrm{.}\]
and
\[\theta_{+}:k\mapsto 2k\]
We then define, for \(k\in\left[2n\right]\),
\[\gamma^{\prime}\left(k\right)=\left\{\begin{array}{ll}\theta_{+}\theta_{-}^{-1}\left(k\right)&\textrm{\(k\) odd}\\\theta_{-}\gamma\theta_{+}^{-1}\left(k\right)&\textrm{\(k\) even}\end{array}\right.\textrm{,}\]
\(\varepsilon:k\mapsto\left(-1\right)^{k}\), and \(\delta^{\prime}=\delta_{\varepsilon}\).  Then
\begin{eqnarray*}
&&\mathbb{E}\left(\mathrm{tr}_{\gamma}\left(W_{\lambda_{1}}Y_{1},\cdots,W_{\lambda_{n}}Y_{n}\right)\right)\\
&=&\mathbb{E}\left(\mathrm{tr}_{\gamma^{\prime}}\left(X^{T}D_{\lambda_{1}},XY_{1},\ldots,X^{T}D_{\lambda_{n}},XY_{n}\right)\right)\\
&=&\sum_{\rho\in{\cal P}_{2}\left(2n\right)}N^{\#\left(\gamma_{-}^{\prime-1}\delta^{\prime}\rho\delta\rho\delta^{\prime}\gamma_{+}^{\prime}\right)/2-\#\left(\gamma^{\prime}_{+}\gamma^{\prime-1}_{-}\right)/2-n}\\&&\times\mathrm{tr}_{\gamma_{-}^{\prime-1}\delta^{\prime}\rho\delta\rho\delta^{\prime}\gamma_{+}^{\prime}/2}\left(D_{\lambda_{1}},Y_{1},\ldots,D_{\lambda_{n}},Y_{n}\right)\textrm{.}
\end{eqnarray*}

The permutation \(\gamma_{-}^{\prime-1}\delta^{\prime}\rho\delta\rho\delta^{\prime}\gamma_{+}^{\prime}\) acts separately on the odd and the even integers: by Lemma~\ref{lemma: pairing}, \(\delta^{\prime}\rho\delta\rho\delta^{\prime}\) reverses either the parity or the sign.  In the former case, exactly one of \(\gamma_{+}^{\prime}\) and \(\gamma_{-}^{\prime-1}\) acts, and in the latter both or neither do.  In either case, the parity is preserved.

We extend the \(\theta_{\pm}\) as odd functions (so they commute with \(\delta\), \(\delta^{\prime}\), and absolute values).  We define \(\pi\in PM\left(\pm\left[n\right]\right)\) by \(\pi^{-1}=\theta_{-}^{-1}\gamma_{-}^{\prime-1}\delta^{\prime}\rho\delta\rho\delta^{\prime}\gamma_{+}^{\prime}\theta_{-}\).  Then the contribution of the odd cycles is \(\mathrm{tr}_{\pi^{-1}/2}\left(D_{\lambda_{1}},\ldots,D_{\lambda_{n}}\right)\).  We now show that the contribution of the even cycles is \(\mathbb{E}\left(\mathrm{tr}_{\gamma_{-}^{-1}\pi\gamma_{+}/2}\left(Y_{1},\ldots,Y_{n}\right)\right)\).  We note that \(\gamma^{\prime 2}=\theta_{\pm}\gamma\theta_{\pm}^{-1}\) for whichever of \(\theta_{\pm}^{-1}\) is defined, and \(\theta_{+}\theta_{-}^{-1}=\gamma^{\prime}_{+}\gamma^{\prime}_{-}\).  Then
\[\theta_{+}\gamma_{-}^{-1}\pi\gamma_{+}\theta_{+}^{-1}=\gamma_{-}^{\prime-2}\theta_{+}\theta_{-}^{-1}\gamma_{+}^{\prime-1}\delta^{\prime}\rho\delta\rho\delta^{\prime}\gamma_{-}^{\prime}\theta_{-}\theta_{+}^{-1}\gamma_{+}^{\prime 2}=\gamma^{\prime-1}_{-}\delta^{\prime}\rho\delta\rho\delta^{\prime}\gamma^{\prime}_{+}\textrm{.}\]
We can see that \(\#\left(\gamma^{\prime}\right)=\#\left(\gamma\right)\), since odd integers always occur in cycles with even integers, and on even integers, \(\gamma^{\prime 2}=\theta_{+}\gamma\theta_{+}^{-1}\), so the exponent on \(N\) is correct.

We now show that this map from \({\cal P}_{2}\left(2n\right)\) to \(PM\left(\pm\left[n\right]\right)\) is a bijection.  For a given \(\pi\in PM\left(\pm\left[n\right]\right)\), we can reconstruct from \(\pi\) and \(\gamma_{-}^{-1}\pi\gamma_{+}\) the permutation \(\gamma^{\prime-1}_{-}\delta^{\prime}\rho\delta\rho\delta^{\prime}\gamma^{\prime}_{+}\) and hence \(\delta^{\prime}\rho\delta\rho\delta^{\prime}\), and by Lemma~\ref{lemma: pairing} we can reconstruct \(\rho\).  Thus the map is injective.

Conversely, for \(\pi\in PM\left(\pm\left[n\right]\right)\), we can define \(\rho\) on \(k\in\left[2n\right]\) by letting
\[\rho\left(\theta_{\pm}\left(k\right)\right)=\left|\theta_{\mp\mathrm{sgn}\left(\pi^{\mp 1}\left(k\right)\right)}\pi^{\mp 1}\left(k\right)\right|\textrm{.}\]
(All integers in \(\left[n\right]\) are in the range of exactly one of the \(\theta_{\pm 1}\).)  If \(\rho\left(\theta_{\pm}\left(k\right)\right)=\theta_{\pm 1}\left(k\right)\), then \(\mathrm{sgn}\left(\pi^{\mp 1}\left(k\right)\right)=-1\) and \(\left|\pi^{\mp 1}\left(k\right)\right|=k\), which is excluded since \(\pi\) is a premap; thus \(\rho\) has no fixed points.  We calculate:
\begin{multline*}
\rho^{2}\left(\theta_{\pm 1}\left(k\right)\right)\\
=\left|\theta_{\pm\mathrm{sgn}\left(\pi^{\mp 1}\left(k\right)\right)\mathrm{sgn}\left(\pi^{\pm\mathrm{sgn}\left(\pi^{\mp 1}\left(k\right)\right)}\left(\left|\pi^{\mp 1}\left(k\right)\right|\right)\right)}\pi^{\pm\mathrm{sgn}\left(\pi^{\mp 1}\left(k\right)\right)}\left(\left|\pi^{\mp 1}\left(k\right)\right|\right)\right|\textrm{.}
\end{multline*}
Compensating for the absolute value signs, we note that
\begin{multline*}
\pi^{\pm\mathrm{sgn}\left(\pi^{\mp 1}\left(k\right)\right)}\left(\left|\pi^{\mp 1}\left(k\right)\right|\right)=\pi^{\pm\mathrm{sgn}\left(\pi^{\mp 1}\left(k\right)\right)}\left(\mathrm{sgn}\left(\pi^{\mp 1}\left(k\right)\right)\pi^{\mp 1}\left(k\right)\right)\\=\mathrm{sgn}\left(\pi^{\mp 1}\left(k\right)\right)\pi^{\pm\mathrm{sgn}\left(\pi^{\mp 1}\left(k\right)\right)^{2}}\pi^{\mp}\left(k\right)=\mathrm{sgn}\left(\pi^{\mp 1}\left(k\right)\right)k\textrm{,}
\end{multline*}
we can see that \(\rho^{2}\) is the identity.  Thus \(\rho\) is a pairing.

Finally, we show that \(\theta_{-}^{-1}\gamma_{+}^{\prime-1}\delta^{\prime}\rho\delta\rho\delta^{\prime}\gamma_{-}^{\prime}\theta_{-}\) is premap \(\pi\).  Since the permutation \(\gamma_{+}^{\prime-1}\delta^{\prime}\rho\delta\rho\delta^{\prime}\gamma_{-}^{\prime}\) takes odd numbers to odd numbers, \(\gamma^{\prime}_{-}\theta_{-}=\theta_{-\mathrm{sgn}\left(k\right)}\) and \(\theta_{-}^{-1}\gamma^{\prime-1}_{+}=\theta_{\pm}^{-1}\) for whichever \(\theta_{\pm}^{-1}\) is defined.  We have
\begin{multline*}
\rho\left(\left|\theta_{-\mathrm{sgn}\left(k\right)}\left(k\right)\right|\right)=\left|\theta_{\mathrm{sgn}\left(k\right)\mathrm{sgn}\left(\pi^{\mathrm{sgn}\left(k\right)}\left(\left|k\right|\right)\right)}\pi^{\mathrm{sgn}\left(k\right)}\left(\left|k\right|\right)\right|\\=\left|\theta_{\mathrm{sgn}\left(k\right)^{2}\mathrm{sgn}\left(\pi\left(k\right)\right)}\pi\left(k\right)\right|=\left|\theta_{\mathrm{sgn}\left(\pi\left(k\right)\right)}\pi\left(k\right)\right|
\end{multline*}
so by Lemma~\ref{lemma: pairing}, \(\delta^{\prime}\rho\delta\rho\delta^{\prime}\) takes \(\theta_{-\mathrm{sgn}\left(k\right)}\left(k\right)\) to
\begin{multline*}
-\mathrm{sgn}\left(\theta_{-\mathrm{sgn}\left(k\right)}\left(k\right)\right)\varepsilon\left(\theta_{-\mathrm{sgn}\left(k\right)}\left(k\right)\right)\varepsilon\left(\rho\left(\left|\theta_{-\mathrm{sgn}\left(k\right)}\left(k\right)\right|\right)\right)\rho\left(\left|\theta_{-\mathrm{sgn}\left(k\right)}\left(k\right)\right|\right)\\=-\mathrm{sgn}\left(k\right)\varepsilon\left(\theta_{-\mathrm{sgn}\left(k\right)}\left(k\right)\right)\varepsilon\left(\theta_{\mathrm{sgn}\left(\pi\left(k\right)\right)}\pi\left(k\right)\right)\left|\theta_{\mathrm{sgn}\left(\pi\left(k\right)\right)}\pi\left(k\right)\right|\\=\mathrm{sgn}\left(k\right)^{2}\mathrm{sgn}\left(\pi\left(k\right)\right)\left|\theta_{\mathrm{sgn}\left(\pi\left(k\right)\right)}\pi\left(k\right)\right|=\theta_{\mathrm{sgn}\left(k\right)\mathrm{sgn}\left(\pi\left(k\right)\right)}\pi\left(k\right)\textrm{.}
\end{multline*}
The result follows.
\end{proof}

\begin{figure}
\centering
\scalebox{0.5}{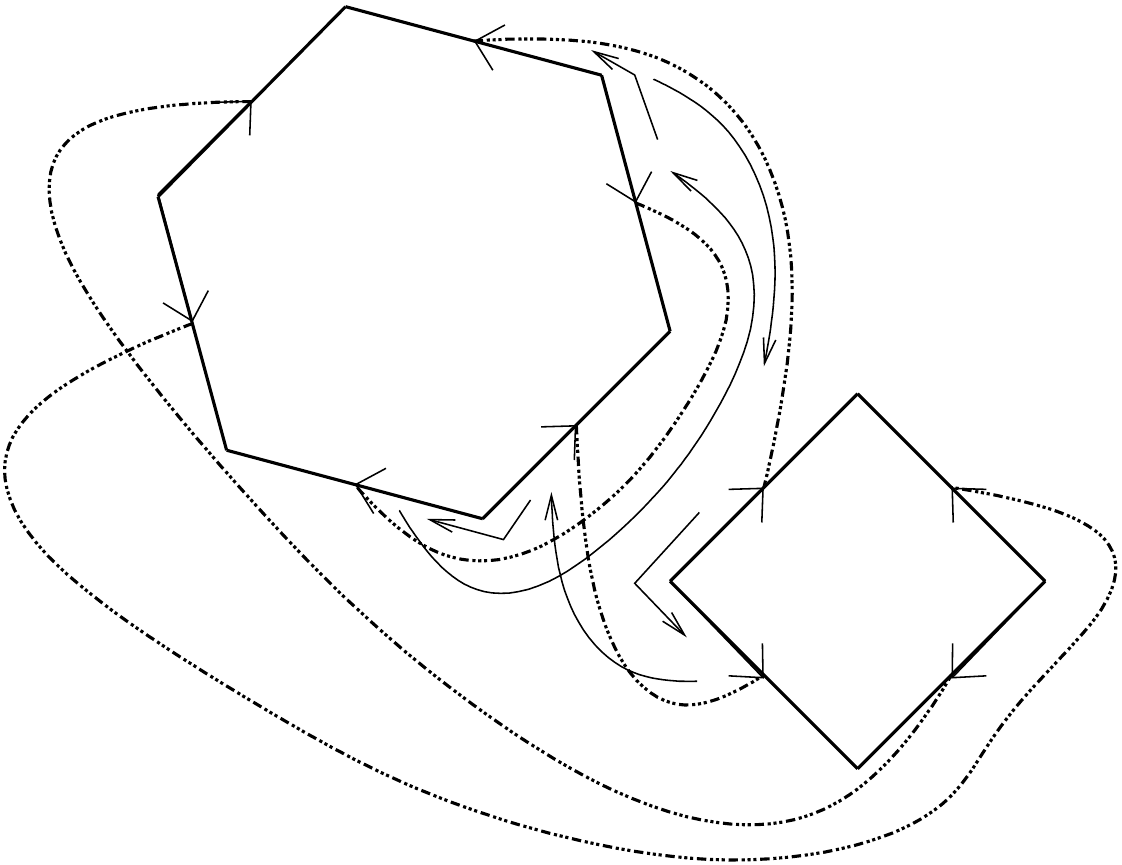}
\scalebox{0.5}{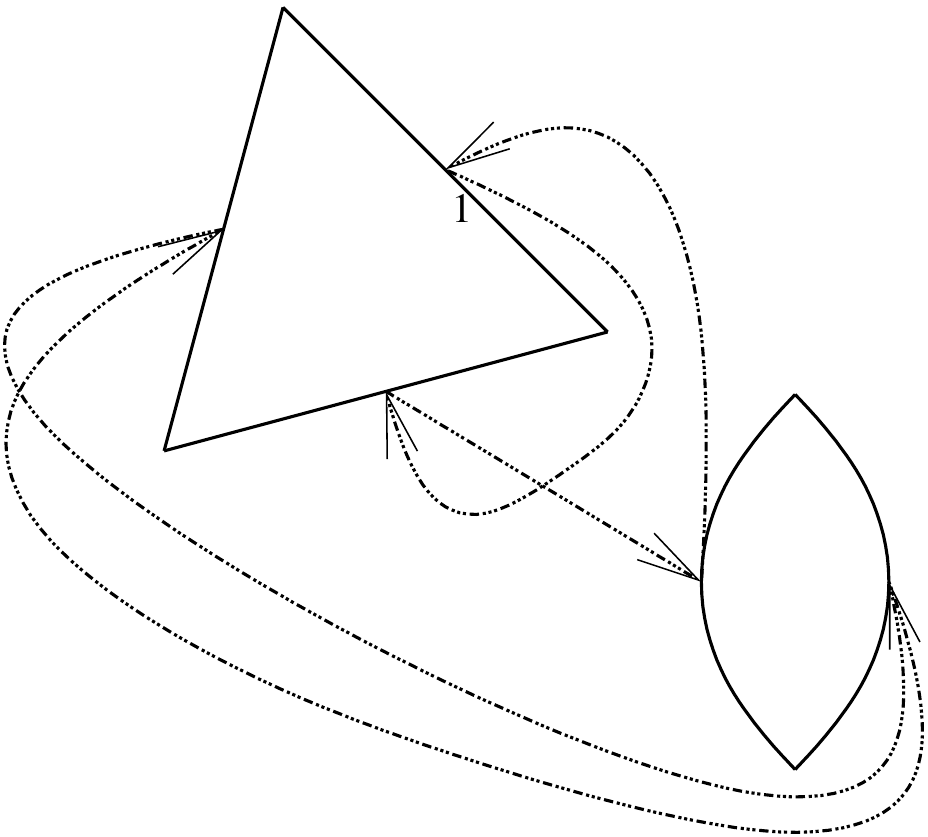}
\caption{A pairing (left) and its equivalent premap (right).  ``Twists'' in the cycles of the premap, which denote that that edge is identified in the opposite direction, correspond to changes in sign.}
\label{hyperedge}
\end{figure}

\begin{example}
If we wish to calculate the value of
\begin{multline*}
\mathbb{E}\left(\mathrm{tr}\left(W_{1}Y_{1}W_{2}Y_{2}W_{3}Y_{3}\right)\mathrm{tr}\left(W_{4}Y_{4}W_{5}Y_{5}\right)\right)\\=\mathbb{E}\left(\mathrm{tr}\left(X^{T}D_{1}XY_{1}X^{T}D_{2}XY_{2}X^{T}D_{3}XY_{3}\right)\mathrm{tr}\left(X^{T}D_{4}XY_{4}X^{T}D_{5}XY_{5}\right)\right)\textrm{,}
\end{multline*}
Lemma~\ref{genus} gives us the expression
\[\sum_{\rho\in{\cal P}_{2}\left(\left[10\right]\right)}N^{\#\left(\gamma_{-}^{\prime-1}\delta^{\prime}\rho\delta\rho\delta^{\prime}\gamma_{+}^{\prime}\right)/2-7}\mathbb{E}\left(\mathrm{tr}_{\gamma_{-}^{\prime-1}\delta^{\prime}\rho\delta\rho\delta^{\prime}\gamma_{+}^{\prime}/2}\left(D_{1},Y_{1},\ldots,D_{5},Y_{5}\right)\right)\]
where
\[\gamma=\left(1,2,3,4,5,6\right)\left(7,8,9,10\right)\textrm{.}\]
Each term in the sum can be represented as a face gluing (see \cite{MR2851244}, Section~3, for more detail).  The pairing
\[\rho=\left(1,5\right)\left(2,7\right)\left(3,9\right)\left(4,10\right)\left(6,8\right)\]
is shown in Figure~\ref{hyperedge}, left.  We calculate that
\begin{multline*}
\gamma_{-}^{\prime-1}\delta^{\prime}\rho\delta\rho\delta^{\prime}\gamma_{+}^{\prime}=\\\left(1,7,-5\right)\left(5,-7,-1\right)\left(2,-8,6,-4,10\right)\left(-10,4,-6,8,-2\right)\left(3,-9\right)\left(9,-3\right)\textrm{.}
\end{multline*}
The region that becomes the vertex \(\left(1,7,-5\right)\left(5,-7,-1\right)\), which becomes \(\mathrm{tr}\left(D_{1}D_{4}D_{3}^{T}\right)\), is marked with arrows.  The contribution of this term is then
\[N^{-4}\mathrm{tr}\left(D_{1}D_{4}D_{3}^{T}\right)\mathrm{tr}\left(D_{2}D_{5}^{T}\right)\mathbb{E}\left(\mathrm{tr}\left(Y_{1}Y_{4}^{T}Y_{3}Y_{2}^{T}Y_{5}\right)\right)\textrm{.}\]

As shown in Figure~\ref{hyperedge}, right, we can ``pinch'' the two edges corresponding to each \(W_{k}\) matrix together.  We now have
\[\gamma=\left(1,2,3\right)\left(4,5\right)\textrm{,}\]
and Lemma~\ref{Wishart} gives us that
\begin{multline*}
\mathbb{E}\left(\mathrm{tr}\left(W_{1}Y_{1}W_{2}Y_{2}W_{3}Y_{3}\right)\mathrm{tr}\left(W_{4}Y_{4}W_{5}Y_{5}\right)\right)\\=\sum_{\pi\in PM\left(\pm\left[5\right]\right)}N^{\chi\left(\gamma,\pi\right)-4}\mathrm{tr}_{\pi^{-1}/2}\left(D_{1},\ldots,D_{5}\right)\mathbb{E}\left(\mathrm{tr}_{\gamma_{-}^{-1}\pi\gamma_{+}/2}\left(Y_{1},\ldots,Y_{5}\right)\right)
\end{multline*}
The above pairing \(\rho\) corresponds to premap
\[\pi=\left(1,-3,4\right)\left(-4,3,-1\right)\left(2,-5\right)\left(5,-2\right)\textrm{.}\]
The cycles \(\left(1,-3,4\right)\left(-4,3,-1\right)\) are the hyperedge corresponding to the marked vertex in the previous diagram.  This premap has vertex permutation
\[\gamma_{-}^{-1}\pi\gamma_{+}=\left(1,-4,3,-2,5\right)\left(-5,2,-3,4,-1\right)\textrm{.}\]
We calculate that
\[\chi\left(\gamma,\pi\right)=2+2+1-5=0\textrm{,}\]
so \(\pi\) gives the same contribution as \(\rho\) did above.
\end{example}

\section{Combinatorial calculations}
\label{combinatorics}

Let \(\left\{X_{\lambda}\right\}_{\lambda\in\Lambda}\) be real Ginibre, GOE, or real Wishart matrices.  (We consider a family of matrices in order to accommodate Wishart matrices \(W_{\lambda_{k}}=X^{T}D_{\lambda_{k}}X\) with the same matrix \(X\) but possibly distinct deterministic matrices \(D_{\lambda_{k}}\); in the Ginibre and GOE cases there is only one matrix \(X_{\lambda}\).  Independent ensembles will be handled by Lemma~\ref{colour}.)  We note that for all \(n\), \(\gamma\in S_{n}\), \(\varepsilon:\left[n\right]\rightarrow\left\{1,-1\right\}\), and \(Y_{k}\) random matrices independent from the \(X_{\lambda}\), the \(X_{\lambda}\) satisfy
\begin{multline}
\label{general}
\mathbb{E}\left(\mathrm{tr}_{\gamma}\left(X_{\lambda_{1}}^{\left(\varepsilon\left(1\right)\right)}Y_{1},\cdots,X_{\lambda_{n}}^{\left(\varepsilon\left(n\right)\right)}Y_{n}\right)\right)
\\=\sum_{\pi\in PM_{c}\left(\pm\left[n\right]\right)}N^{\chi\left(\gamma,\delta_{\varepsilon}\pi\delta_{\varepsilon}\right)-2\#\left(\gamma\right)}f_{c}\left(\pi\right)\mathbb{E}\left(\mathrm{tr}_{\gamma_{-}^{-1}\delta_{\varepsilon}\pi\delta_{\varepsilon}\gamma_{+}/2}\left(Y_{1},\ldots,Y_{n}\right)\right)
\end{multline}
where \(PM_{c}\left(\pm\left[n\right]\right)\) is a subset of the premaps on \(\pm\left[n\right]\) and \(f_{c}\) is a function on those premaps.  In each case, for each finite set of positive integers \(I\), \(PM_{c}\left(\pm I\right)\subseteq PM\left(\pm I\right)\) is a subset of the premaps on \(\pm I\) such that for any \(J\subseteq I\), the \(\pi\in PM_{c}\left(\pm I\right)\) which do not connect \(\pm J\) and \(\pm\left(I\setminus J\right)\) are the product of a \(\pi_{1}\in PM_{c}\left(\pm J\right)\) and \(\pi_{2}\in PM_{c}\left(\pm\left(I\setminus J\right)\right)\), and \(f_{c}:\bigcup_{I\subseteq\mathbb{N},\left|I\right|<\infty}PM_{c}\left(\pm I\right)\rightarrow\mathbb{C}\) is a function such that \(\lim_{N\rightarrow\infty}f_{c}\left(\pi\right)\) exists.  Furthermore, if \(\pi\in PM_{c}\left(I\right)\) does not connect \(\pm J\) and \(\pm\left(I\setminus J\right)\), then \(f_{c}\left(\pi\right)=f_{c}\left(\left.\pi\right|_{\pm J}\right)f_{c}\left(\left.\pi\right|_{\pm\left(I\setminus J\right)}\right)\).  (This last condition is the only one not satisfied by Haar-distributed orthogonal matrices (appearing in future work).  As it is only required for some of the results, we will note when it is needed.)

Specifically, for real Ginibre matrices, \(PM_{c}\left(\pm I\right)=\left\{\rho\delta\rho:\rho\in{\cal P}_{2}\left(I\right)\right\}\) and \(f_{c}\left(\pi\right)=1\).  For GOE matrices, \(PM_{c}\left(\pm I\right)=PM\left(\pm I\right)\cap{\cal P}_{2}\left(\pm I\right)\) and \(f_{c}\left(\pi\right)=1\).  For real Wishart matrices, \(PM_{c}\left(\pm I\right)=PM\left(\pm I\right)\) and, if \(X_{\lambda_{k}}=W_{\lambda_{k}}=X^{T}D_{\lambda_{k}}X\) and \(I=\left\{i_{1},\ldots,i_{m}\right\}\) with \(i_{1}<\ldots<i_{m}\), then \(f_{c}\left(\pi\right)=\mathrm{tr}_{\pi^{-1}/2}\left(D_{\lambda_{i_{1}}},\ldots,D_{\lambda_{i_{m}}}\right)\).

\subsection{Moments and cumulants}

We now give a formula for the traces of products of matrices from several independent ensembles of matrices satisfying (\ref{general}).  It does not depend on the last condition (that \(f\) be multiplicative).  This formula can be used to calculate expressions with several independent Ginibre, GOE, or Wishart matrices, or any combination of such matrices.

\begin{remark}
Lemma~\ref{colour} is the fairly intuitive result that, given a product of traces of several independent ensembles of random matrices, we construct faces as before and glue the edges belonging to each independent ensemble according to the rules of that particular ensemble.  An example with a diagram is given after the proof.
\end{remark}

\begin{lemma}
\label{colour}
For each colour \(c\in\left[C\right]\), let \(\left\{X_{c}^{\left(\lambda\right)}\right\}_{\lambda\in\Lambda}\) be a set of random matrices satisfying (\ref{general}) with subsets of the premaps \(PM_{c}\left(\pm I\right)\) and function \(f_{c}:\bigcup_{I\subseteq\mathbb{N},\left|I\right|<\infty}PM_{c}\left(\pm I\right)\rightarrow\mathbb{C}\), and assume the set associated with each colour \(c\) is independent from every other set.  Let \(w:\left[n\right]\rightarrow\left[C\right]\) be a word in the set of colours \(\left[C\right]\).  Then, for \(\gamma\in S_{n}\) and \(\varepsilon:\left[n\right]\rightarrow\left\{1,-1\right\}\),
\begin{multline}
\label{complement}
\mathbb{E}\left(\mathrm{tr}_{\gamma}\left(X_{w\left(1\right)}^{\left(\varepsilon\left(1\right)\lambda_{1}\right)},\ldots,X_{w\left(n\right)}^{\left(\varepsilon\left(n\right)\lambda_{n}\right)}\right)\right)
\\=\sum_{\substack{\pi=\pi_{1}\ldots\pi_{C}\\\pi_{c}\in PM_{c}\left(\pm w^{-1}\left(c\right)\right)}}N^{\chi\left(\gamma,\delta_{\varepsilon}\pi\delta_{\varepsilon}\right)-2\#\left(\gamma\right)}f_{1}\left(\pi_{1}\right)\cdots f_{C}\left(\pi_{C}\right)\textrm{.}
\end{multline}
\end{lemma}
\begin{proof}
We prove this lemma by induction on the number of colours \(C\).  For \(C=1\), Formula~(\ref{complement}) reduces to (\ref{general}), where \(Y_{1}=\cdots=Y_{n}=I_{N}\).  We now assume Formula~(\ref{complement}) for \(C-1\) colours.

The left hand side of (\ref{complement}) is of the form of (\ref{general}) in colour \(C\), where \(\gamma^{\prime}:=\left.\gamma\right|_{w^{-1}\left(C\right)}\) takes the place of \(\gamma\).  (We may cycle any matrices appearing in a trace before the first occurrence of a matrix of colour \(C\) to the end without changing the value, and any traces which do not contain any matrices of colour \(C\) may be treated as a constant multiplicative factor within the expectation.)  Let \(w^{-1}\left(C\right)=\left\{k_{1},\ldots,k_{\left|w^{-1}\left(C\right)\right|}\right\}\), \(k_{1}<\ldots<k_{\left|w^{-1}\left(C\right)\right|}\).  We call \(Y_{k_{i}}\) the product of the matrices appearing (cyclically) between \(X_{C}^{\left(\varepsilon\left(k_{i}\right)\lambda_{k_{i}}\right)}\) and \(X_{C}^{\left(\varepsilon\left(k_{\gamma^{\prime}\left(i\right)}\right)\lambda_{k_{\gamma^{\prime}\left(i\right)}}\right)}\).  (This may be no matrices, in which case we let \(Y_{k_{i}}:=I_{N}\).)  For any trace which does not contain any matrices of colour \(C\), we call the matrix in the trace \(Z_{k}\), \(1\leq k\leq\#\left(\gamma\right)-\#\left(\gamma^{\prime}\right)\).  Then:
\begin{eqnarray*}
&&\mathbb{E}\left(\mathrm{tr}_{\gamma}\left(X_{w\left(1\right)}^{\left(\varepsilon\left(1\right)\lambda_{1}\right)},\ldots,X_{w\left(n\right)}^{\left(\varepsilon\left(n\right)\lambda_{n}\right)}\right)\right)
\\&=&\mathbb{E}\left(\mathrm{tr}_{\gamma^{\prime}}\left(X_{C}^{\left(\varepsilon\left(k_{1}\right)\lambda_{k_{1}}\right)}Y_{k_{1}},\ldots,X_{C}^{\left(\varepsilon\left(k_{\left|w^{-1}\left(C\right)\right|}\right)\lambda_{k_{\left|w^{-1}\left(C\right)\right|}}\right)}Y_{k_{\left|w^{-1}\left(C\right)\right|}}\right)\right.\\&&\left.\times\mathrm{tr}\left(Z_{1}\right)\cdots\mathrm{tr}\left(Z_{\#\left(\gamma\right)-\#\left(\gamma^{\prime}\right)}\right)\right)
\\&=&\sum_{\pi_{C}\in PM_{C}\left(\pm w^{-1}\left(C\right)\right)}N^{\chi\left(\gamma^{\prime},\delta_{\varepsilon}\pi_{C}\delta_{\varepsilon}\right)-2\#\left(\gamma^{\prime}\right)}f_{C}\left(\pi_{C}\right)\\&&\times\mathbb{E}\left(\mathrm{tr}_{\gamma_{-}^{\prime-1}\delta_{\varepsilon}\pi_{C}\delta_{\varepsilon}\gamma_{+}^{\prime}/2}\left(Y_{k_{1}},\ldots,Y_{k_{\left|w^{-1}\left(C\right)\right|}}\right)\mathrm{tr}\left(Z_{1}\right)\cdots\mathrm{tr}\left(Z_{\#\left(\gamma\right)-\#\left(\gamma^{\prime}\right)}\right)\right)\textrm{.}
\end{eqnarray*}
We note that each matrix of one of the colours in \(\left[C-1\right]\) appears exactly once inside the expected value expression, possibly transposed.  Let us denote by \(I\) the set of signed integers representing the indices of the matrices which appear.  By the induction hypothesis, for some permutation \(\gamma^{\prime\prime}\in S\left(I\right)\), counting cycles as permutations on \(\pm w^{-1}\left(\left[C-1\right]\right)\),
\begin{multline*}
\mathbb{E}\left(\mathrm{tr}_{\gamma_{-}^{\prime-1}\delta_{\varepsilon}\pi_{C}\delta_{\varepsilon}\gamma_{+}^{\prime}/2}\left(Y_{k_{1}},\ldots,Y_{k_{\left|w^{-1}\left(C\right)\right|}}\right)\mathrm{tr}\left(Z_{1}\right)\cdots\mathrm{tr}\left(Z_{\#\left(\gamma\right)-\#\left(\gamma^{\prime}\right)}\right)\right)
\\=\sum_{\substack{\pi^{\prime}=\pi_{1}\cdots\pi_{C-1}\\\pi_{c}\in PM_{c}\left(\pm w^{-1}\left(c\right)\right)}}N^{\chi\left(\gamma^{\prime\prime},\delta_{\varepsilon}\pi^{\prime}\delta_{\varepsilon}\right)-2\#\left(\gamma^{\prime\prime}\right)}f_{1}\left(\pi_{1}\right)\cdots f_{C-1}\left(\pi_{C-1}\right)\textrm{.}
\end{multline*}
(We note that \(PM_{c}\left(\pm w^{-1}\left(c\right)\right)\) is still the appropriate subset of the premaps: if \(-k\in I\) for \(k>0\), then \(-k\) would appear instead of \(k\) in the premaps in \(PM_{c}\left(\pm w^{-1}\left(c\right)\right)\), but the sign of \(\varepsilon\left(k\right)\) and hence \(\delta_{\varepsilon}\left(k\right)\) would be reversed, so the permutation \(\delta_{\varepsilon}\pi_{c}\delta_{\varepsilon}\) is unchanged.)

Substituting this expression into the previous one, it only remains to confirm that the exponent on \(N\) is correct.  To do so, we examine the permutation \(\gamma^{\prime\prime}\).  We show that \(\gamma^{\prime\prime}=\left.\delta_{\varepsilon}\pi_{C}\delta_{\varepsilon}\gamma_{+}\gamma_{-}^{-1}\right|_{I}\).

For \(i\) in a cycle of \(\gamma\) which does not contain colour \(C\), \(\gamma^{\prime\prime}\left(i\right)=\gamma\left(i\right)=\delta_{\varepsilon}\pi_{C}\delta_{\varepsilon}\gamma_{+}\gamma_{-}^{-1}\left(i\right)\). 

If a term with index \(i\) appears in \(Y_{k}\) (\(k\) signed, and hence \(i\) of the same sign as \(k\)), then the next term, if there is one, has index \(\gamma_{+}\gamma_{-}^{-1}\left(i\right)\) (the correct permutation acts given the sign of \(i\) and hence the order in which the terms of \(Y_{k}\) appear).  If there is such a next term, \(w\left(\gamma_{+}\gamma_{-}^{-1}\left(i\right)\right)\neq C\), so \(\delta_{\varepsilon}\pi_{C}\delta_{\varepsilon}\) acts trivially on \(\gamma_{+}\gamma_{-}^{-}\left(i\right)\), and hence \(\gamma^{\prime\prime}\left(i\right)=\left.\delta_{\varepsilon}\pi_{C}\delta_{\varepsilon}\gamma_{+}\gamma_{-}^{-1}\right|_{I}\left(i\right)\).

On the other hand, if there is no such next term, then \(\gamma^{\prime\prime}\left(i\right)\) is the first term of \(Y_{\left(\gamma_{-}^{\prime-1}\delta_{\varepsilon}\pi_{C}\delta_{\varepsilon}\gamma_{+}^{\prime}\right)^{m}\left(k\right)}\), where \(m\) is the smallest positive integer such that this term is nontrivial.  If \(i\) is the last index appearing in \(Y_{k}\), then \(k=\gamma_{+}^{\prime-1}\gamma_{+}\gamma_{-}^{-1}\left(i\right)\) (\(\gamma_{+}^{\prime-1}\) acts depending on the sign of \(k\), i.e.\ on whether \(k\) is to be found before or after the indices appearing in \(Y_{k}\)).  The first term of \(Y_{\gamma_{-}^{\prime-1}\delta_{\varepsilon}\pi_{C}\delta_{\varepsilon}\gamma_{+}^{\prime}\left(k\right)}\) (if it exists) has index \(\gamma_{+}\gamma_{-}^{-1}\gamma_{-}^{\prime}\gamma_{-}^{\prime-1}\delta_{\varepsilon}\pi_{C}\delta_{\varepsilon}\gamma_{+}\gamma_{-}^{-1}\left(i\right)=\gamma_{+}\gamma_{-}^{-1}\delta_{\varepsilon}\pi_{C}\delta_{\varepsilon}\gamma_{+}\gamma_{-}^{-1}\left(i\right)\).  If, however, this index has colour \(C\), it is nonetheless the index encountered after the last term of \(Y_{\gamma_{-}^{\prime-1}\delta_{\varepsilon}\pi_{C}\delta_{\varepsilon}\gamma_{+}^{\prime}\left(k\right)}\), in which case we may apply permutation \(\gamma_{+}\gamma_{-}^{-1}\delta_{\varepsilon}\pi_{C}\delta_{\varepsilon}\) until we have an index not of colour \(C\).  Since \(w\left(\left|l\right|\right)\neq C\) exactly when \(w\left(\left|\delta_{\varepsilon}\pi_{C}\delta_{\varepsilon}\left(l\right)\right|\right)\neq C\) (and then \(\delta_{\varepsilon}\pi_{C}\delta_{\varepsilon}\) acts trivially), the index of the next term is indeed \(\left.\delta_{\varepsilon}\pi_{C}\delta_{\varepsilon}\gamma_{+}\gamma_{-}^{-1}\right|_{I}\left(i\right)\).

On \(-I\), \(\gamma_{-}^{\prime\prime-1}=\left.\delta\delta_{\varepsilon}\pi_{C}\delta_{\varepsilon}\gamma_{+}\gamma_{-}^{-1}\delta\right|_{-I}^{-1}=\left.\gamma_{+}\gamma_{-}^{-1}\delta_{\varepsilon}\pi_{C}\delta_{\varepsilon}\right|_{-I}\).  Since \(\delta_{\varepsilon}\pi_{C}\delta_{\varepsilon}\) acts trivially on \(\pm w^{-1}\left(\left[C-1\right]\right)\), this is equal to \(\left.\gamma_{-}^{\prime\prime-1}=\delta_{\varepsilon}\pi_{C}\delta\gamma_{+}\gamma_{-}^{-1}\right|_{-I}\).  Thus \(\gamma_{+}^{\prime\prime}\gamma_{-}^{\prime\prime-1}=\left.\delta_{\varepsilon}\pi_{C}\delta_{\varepsilon}\gamma_{+}\gamma_{-}^{-1}\right|_{\pm I}\), so \(\delta_{\varepsilon}\pi^{\prime}\delta_{\varepsilon}\gamma_{+}^{\prime\prime}\gamma_{-}^{\prime\prime-1}=\left.\delta_{\varepsilon}\pi\delta_{\varepsilon}\gamma_{+}\gamma_{-}^{-1}\right|_{\pm I}\).

We now show that the cycles of \(\delta_{\varepsilon}\pi\delta_{\varepsilon}\gamma_{+}\gamma_{-}^{-1}\) which do not appear in the permutation it induces on \(\pm I=\pm w^{-1}\left(\left[C-1\right]\right)\), that is, those which are entirely colour \(C\), are in one-to-one correspondence with cycles of \(\gamma_{-}^{\prime-1}\delta_{\varepsilon}\pi\delta_{\varepsilon}\gamma_{+}^{\prime}\) consisting entirely of \(k\) with \(Y_{k}\) trivial (that is, those for which \(w\left(\left|k\right|\right)=w\left(\gamma\left(\left|k\right|\right)\right)=C\)), which we count.  We note that \(\gamma\left(\left|k\right|\right)=\left|\gamma_{+}\gamma_{-}\left(k\right)\right|\).

If a cycle of \(\delta_{\varepsilon}\pi\delta_{\varepsilon}\gamma_{+}\gamma_{-}^{-1}\) consists entirely of \(k\) with \(w\left(\left|k\right|\right)=C\), then it is the factor \(\pi_{C}\) of \(\pi\) which acts, so \(w\left(\left|\gamma_{+}\gamma_{-}^{-1}\left(k\right)\right|\right)=C\).  Since only one of \(\gamma_{+}\) and \(\gamma_{-}^{-1}\) acts nontrivially, \(w\left(\left|\gamma_{+}\left(k\right)\right|\right)=w\left(\left|\gamma_{-}^{-1}\left(k\right)\right|\right)=C\), and so \(\gamma_{+}^{\prime}\left(k\right)=\gamma_{+}\left(k\right)\) and \(\gamma_{-}^{\prime-1}\left(k\right)=\gamma_{-}\left(k\right)\).  Thus this cycle also appears as a cycle of \(\delta_{\varepsilon}\pi_{C}\delta_{\varepsilon}\gamma_{+}^{\prime}\gamma_{-}^{\prime-1}\).

If we map each integer in the cycles of \(\delta_{\varepsilon}\pi_{C}\delta_{\varepsilon}\gamma_{+}^{\prime}\gamma_{-}^{\prime-1}\) under \(\gamma_{-}^{\prime-1}\), they become the cycles of \(\gamma_{-}^{\prime-1}\delta_{\varepsilon}\pi_{C}\delta_{\varepsilon}\gamma_{+}^{\prime}\).  If \(k\) is in a cycle of \(\delta_{\varepsilon}\pi\delta_{\varepsilon}\gamma_{+}\gamma_{-}\) which is entirely of colour \(C\), then \(w\left(\left|\gamma_{-}^{\prime-1}\left(k\right)\right|\right)=C\) and \(w
\left(\left|\gamma_{+}\gamma_{-}\gamma_{-}^{\prime-1}\left(k\right)\right|\right)=w\left(\left|\gamma_{+}\left(k\right)\right|\right)=C\).  Thus \(\gamma_{-}^{\prime-1}\left(k\right)\) in the corresponding cycle of \(\delta_{\varepsilon}\pi_{C}\delta_{\varepsilon}\gamma_{+}^{\prime}\gamma_{-}^{\prime-1}\) has \(Y_{\gamma_{-}^{\prime-1}\left(k\right)}\) trivial.

Conversely, cycles of \(\gamma_{-}^{\prime-1}\delta_{\varepsilon}\pi_{C}\delta_{\varepsilon}\gamma_{+}^{\prime}\) correspond to cycles of \(\delta_{\varepsilon}\pi_{C}\delta_{\varepsilon}\gamma_{+}^{\prime}\gamma_{-}^{\prime-1}\) with each integer mapped under \(\gamma_{-}^{\prime}\).  If \(Y_{k}\) is trivial, that is, if \(w\left(\left|k\right|\right)=w\left(\left|\gamma_{+}\gamma_{-}\left(k\right)\right|\right)=C\), then \(w\left(\left|\gamma_{-}^{\prime}\left(k\right)\right|\right)=w\left(\left|\gamma_{-}\left(k\right)\right|\right)=C\).  Thus a cycle in \(\gamma_{-}^{\prime-1}\delta_{\varepsilon}\pi_{C}\delta_{\varepsilon}\gamma_{+}^{\prime}\) of \(k\) with \(Y_{k}\) trivial corresponds to a cycle of \(\delta_{\varepsilon}\pi_{C}\delta_{\varepsilon}\gamma_{+}^{\prime}\gamma_{-}^{\prime-1}\) of colour \(C\), which is equal to a cycle of \(\delta_{\varepsilon}\pi\delta_{\varepsilon}\gamma_{+}\gamma_{-}^{-1}\).

There are \(\#\left(\gamma\right)-\#\left(\gamma^{\prime}\right)\) cycles of \(\gamma^{\prime\prime}\) which correspond to traces of \(Z_{k}\) matrices and hence \(\#\left(\gamma^{\prime\prime}\right)+\#\left(\gamma^{\prime}\right)-\#\left(\gamma\right)\) which contain the \(Y_{k}\) and hence correspond to cycles of \(\gamma_{-}^{\prime-1}\delta_{\varepsilon}\pi_{C}\delta_{\varepsilon}\gamma_{+}^{\prime}/2\).  Thus \(\#\left(\gamma_{-}^{-1}\delta_{\varepsilon}\pi\delta_{\varepsilon}\gamma_{+}\right)/2=\#\left(\left.\delta_{\varepsilon}\pi\delta_{\varepsilon}\gamma_{+}\gamma_{-}^{-1}\right|_{\pm w^{-1}\left(\left[C-1\right]\right)}\right)/2+\#\left(\gamma_{-}^{\prime-1}\delta_{\varepsilon}\pi_{C}\delta_{\varepsilon}\gamma_{+}^{\prime}\right)/2-\#\left(\gamma^{\prime\prime}\right)-\#\left(\gamma^{\prime}\right)+\#\left(\gamma\right)\).  The result follows.
\end{proof}

\begin{figure}
\centering
\scalebox{0.5}{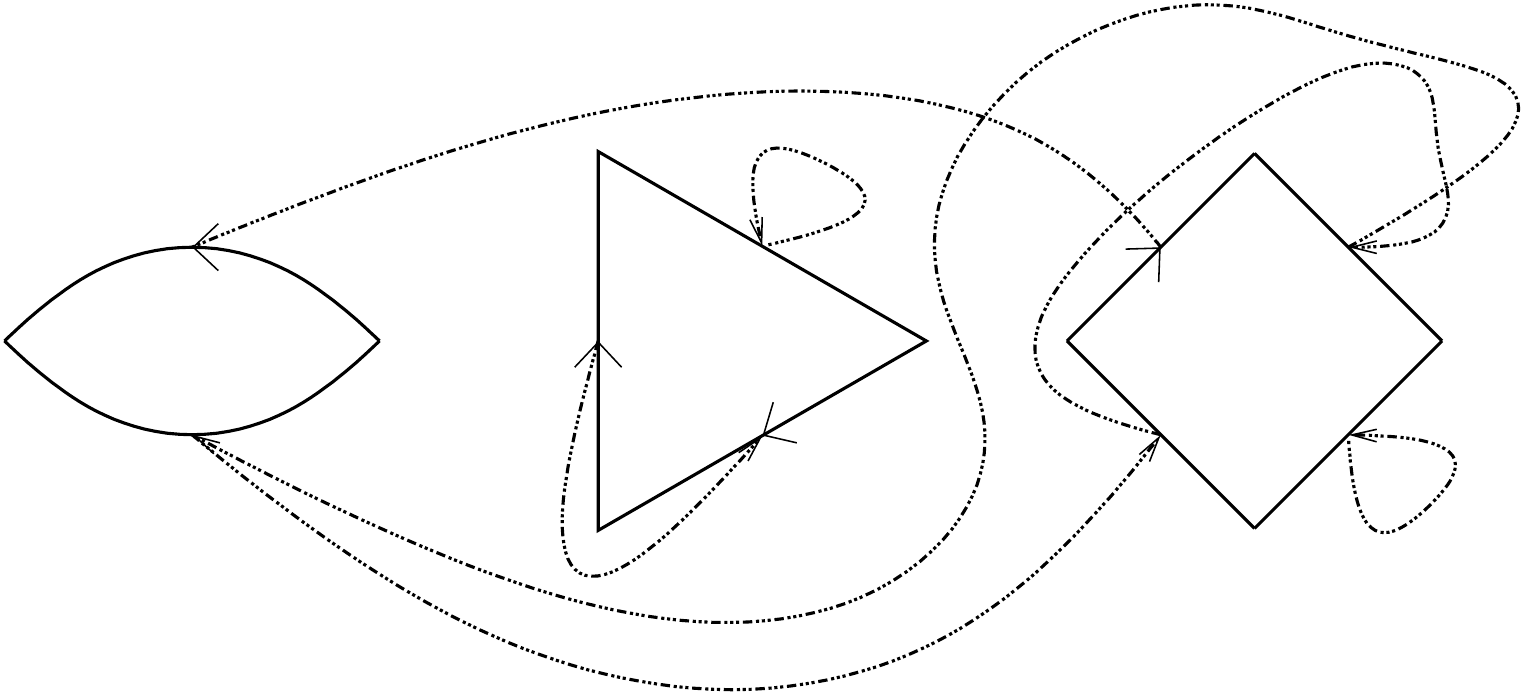}
\caption{The faces associated the calculation in Example~\ref{example: several matrices} involving several independent ensembles and an example of an edge-identification contributing to the result.}
\label{figure: several matrices}
\end{figure}

\begin{example}
\label{example: several matrices}
Let \(C=3\).  Associate with colour \(1\) the set of Wishart matrices
\[\left\{W_{1}^{\left(1\right)}=X_{1}^{T}A_{1}X_{1},W_{1}^{\left(2\right)}=X_{1}^{T}A_{2}X_{1},\ldots\right\}\]
where \(X_{1}\) is a random matrix as in Definition~\ref{definition: Wishart}; with colour \(2\) an independent set of Wishart matrices
\[\left\{W_{2}^{\left(1\right)}=X_{2}^{T}B_{1}X_{2},W_{2}^{\left(2\right)}=X_{2}^{T}B_{2}X_{2},\ldots\right\}\]
where \(X_{2}\) is another random matrix as in Definition~\ref{definition: Wishart} independent from \(X_{1}\); and with colour \(3\) the Ginibre matrix \(Z\), independent from the Wishart matrices.

If we wish to calculate the quantity
\[\mathbb{E}\left(\mathrm{tr}\left(ZW_{2}^{\left(\lambda_{2}\right)}\right)\mathrm{tr}\left(W_{1}^{\left(\lambda_{3}\right)}Z^{T}Z^{T}\right)\mathrm{tr}\left(W_{2}^{\left(\lambda_{6}\right)}Z^{T}W_{2}^{\left(\lambda_{8}\right)}W_{1}^{\left(\lambda_{9}\right)}\right)\right)\textrm{,}\]
for some integers \(\lambda_{2}\), \(\lambda_{3}\), \(\lambda_{6}\), \(\lambda_{8}\) and \(\lambda_{9}\), then we let
\[\gamma=\left(1,2\right)\left(3,4,5\right)\left(6,7,8,9\right)\textrm{,}\]
we let \(w\left(3\right)=w\left(9\right)=1\), \(w\left(2\right)=w\left(6\right)=w\left(8\right)=2\), and \(w\left(1\right)=w\left(4\right)=w\left(5\right)=w\left(7\right)=3\), and we let \(\varepsilon\left(1\right)=\varepsilon\left(2\right)=\varepsilon\left(3\right)=\varepsilon\left(6\right)=\varepsilon\left(8\right)=\varepsilon\left(9\right)=1\) and \(\varepsilon\left(4\right)=\varepsilon\left(5\right)=\varepsilon\left(7\right)=-1\).  We construct faces shown in Figure~\ref{figure: several matrices}.  We then consider the hyperedges we can construct on edges \(3\) and \(9\), the hyperedges we can construct on edges \(2\), \(6\) and \(8\), and the ways in which edges \(1\), \(4\), \(5\) and \(7\) may be identified pairwise in the directions shown.  We show an example of a choice for each corresponding to
\[\pi_{1}=\left(3\right)\left(-3\right)\left(9\right)\left(-9\right)\]
(contributing a factor of \(\mathrm{tr}\left(A_{\lambda_{3}}\right)\mathrm{tr}\left(A_{\lambda_{9}}\right)\)),
\[\pi_{2}=\left(2,8,-6\right)\left(6,-8,-2\right)\]
(contributing a factor of \(\mathrm{tr}\left(B_{\lambda_{2}}B_{\lambda_{6}}^{T}B_{\lambda_{8}}\right)\)),
and
\[\pi_{3}=\left(1,-7\right)\left(-1,7\right)\left(4,-5\right)\left(-4,5\right)\textrm{,}\]
or
\[\delta_{\varepsilon}\pi_{3}\delta_{\varepsilon}=\left(1,7\right)\left(-1,-7\right)\left(4,-5\right)\left(-4,5\right)\textrm{.}\]
This gives us
\begin{multline*}
\delta_{\varepsilon}\pi\delta_{\varepsilon}\\=\left(1,7\right)\left(-1,-7\right)\left(2,8,-6\right)\left(6,-8,-2\right)\left(3\right)\left(-3\right)\left(4,-5\right)\left(-4,5\right)\left(9\right)\left(-9\right)\textrm{.}
\end{multline*}
We calculate that
\[\gamma_{-}^{-1}\delta_{\varepsilon}\pi\delta_{\varepsilon}\gamma_{+}=\left(1,8,9,-7,-2,6\right)\left(-6,2,7,-9,-8,-1\right)\left(3,-4,5\right)\left(-5,4,-3\right)\textrm{.}\]
The contribution of this term is then
\[N^{-5}\mathrm{tr}\left(A_{\lambda_{3}}\right)\mathrm{tr}\left(A_{\lambda_{9}}\right)\mathrm{tr}\left(B_{\lambda_{2}}B_{\lambda_{6}}^{T}B_{\lambda_{8}}\right)\textrm{.}\]
\end{example}

We now show that, if the last condition of (\ref{general}) (the multiplicativity of \(f\)) is satisfied, the cumulants of traces of products of matrices are sums over the premaps which connect blocks \(\pm I\), where \(I\) is an orbit of the permutation \(\gamma\).  Such terms can be thought of as the connected surfaces.

\begin{lemma}
\label{cumulant}
For each colour \(c\in\left[C\right]\), let \(\left\{X_{c}^{\left(\lambda\right)}\right\}_{\lambda\in\Lambda}\) be an ensemble of random matrices satisfying (\ref{general}) with subsets of the premaps \(PM_{c}\) and function \(f_{c}\), and assume that the set associated with colour \(c\) is independent from each other set.  Let \(w:\left[n\right]\rightarrow\left[C\right]\) be a word in the set of colours \(\left[C\right]\).  Let \(n_{1},\ldots,n_{r}\) be positive integers, let \(n:=n_{1}+\cdots+n_{r}\), and let \(I_{k}=\left[n_{1}+\cdots+n_{k-1}+1,n_{1}+\cdots+n_{k}\right]\), \(1\leq k\leq r\).  Let
\[\gamma=\left(1,\ldots,n_{1}\right)\cdots\left(n_{1}+\cdots+n_{r-1},\ldots,n\right)\]
and \(\varepsilon:\left[n\right]\rightarrow\left\{1,-1\right\}\).  For \(1\leq k\leq r\), define classical random variable
\[Y_{k}:=\mathrm{tr}\left(X_{w\left(n_{1}+\cdots+n_{k-1}+1\right)}^{\left(\varepsilon\left(n_{1}+\cdots+n_{k-1}+1\right)\lambda_{n_{1}+\cdots+n_{k-1}+1}\right)}\cdots X_{w\left(n_{1}+\cdots+n_{k}\right)}^{\left(\varepsilon\left(n_{1}+\cdots+n_{k}\right)\lambda_{n_{1}+\cdots+n_{k}}\right)}\right)\textrm{.}\]
Then
\[k_{r}\left(Y_{1},\ldots,Y_{r}\right)
\\=\sum_{\substack{\pi=\pi_{1}\cdots\pi_{C}\\\pi_{c}\in PM_{c}\left(\pm w^{-1}\left(c\right)\right)\\\pi\vee\left\{\pm I_{k}\right\}_{k=1}^{r}=1_{\pm\left[n\right]}}}N^{\chi\left(\gamma,\delta_{\varepsilon}\pi\delta_{\varepsilon}\right)-2r}f_{1}\left(\pi_{1}\right)\cdots f_{C}\left(\pi_{C}\right)\]
where the last condition under the summation sign means that \(\pi\) must connect the blocks \(\pm I_{1},\ldots,\pm I_{r}\).
\end{lemma}
\begin{proof}
We demonstrate this by showing that the moment-cumulant formula (\ref{moment-cumulant}) is satisfied by the conjectured expressions.

Consider the term corresponding to a \(\pi\) in the moment
\begin{eqnarray*}
\mathbb{E}\left(Y_{1}\cdots Y_{r}\right)&=&\mathbb{E}\left(\mathrm{tr}_{\gamma}\left(X_{w\left(1\right)}^{\left(\varepsilon\left(1\right)\lambda_{1}\right)},\ldots,X_{w\left(n\right)}^{\left(\varepsilon\left(n\right)\lambda_{n}\right)}\right)\right)\\
&=&\sum_{\substack{\pi=\pi_{1}\cdots\pi_{C}\\\pi_{c}\in PM_{c}\left(\pm w^{-1}\left(c\right)\right)}}N^{\chi\left(\gamma,\delta_{\varepsilon}\pi\delta_{\varepsilon}\right)-2r}f_{1}\left(\pi_{1}\right)\cdots f_{C}\left(\pi_{C}\right)\textrm{.}
\end{eqnarray*}
Any such \(\pi\) induces a partition \(\rho\in{\cal P}\left(r\right)\): we let \(\rho\) be the smallest partition such that \(s,t\in\left[r\right]\) are in the same block if \(\pm I_{s}\) and \(\pm I_{t}\) are in the same block of \(\pi\vee\left\{\pm I_{k}\right\}_{k=1}^{r}\).  Letting \(I_{V}=\bigcup_{k\in V}I_{k}\) for \(V\subseteq\left[r\right]\), the term corresponding to \(\pi\) can be expressed:
\begin{multline*}
N^{\chi\left(\gamma,\delta_{\varepsilon}\pi\delta_{\varepsilon}\right)-2r}f_{1}\left(\pi_{1}\right)\cdots f_{C}\left(\pi_{C}\right)\\=\prod_{V\in\rho}N^{\chi\left(\left.\gamma\right|_{I_{V}},\left.\delta_{\varepsilon}\pi\delta_{\varepsilon}\right|_{I_{V}}\right)-2\left|V\right|}f_{1}\left(\left.\pi_{1}\right|_{I_{V}}\right)\cdots f_{C}\left(\left.\pi_{C}\right|_{I_{V}}\right)\textrm{.}
\end{multline*}
We note that the induced permutations on \(I_{V}\) are in each case simply the restrictions, since \(\pm I_{V}\) is the union of orbits in each case.  In fact, for \(\pi\) inducing partition \(\rho\), the \(\left.\pi\right|_{I_{V}}\) are exactly the \(\pi^{\prime}=\pi^{\prime}_{1}\cdots\pi^{\prime}_{C}\), \(\pi^{\prime}_{c}\in PM_{c}\left(\pm w^{-1}\left(c\right)\cap\pm I_{V}\right)\) which connect the \(\pm I_{k}\), \(k\in V\).  The sum over all \(\pi\) inducing \(\rho\) is then
\[\prod_{V\in\rho}\sum_{\substack{\pi^{\prime}=\pi^{\prime}_{1}\cdots\pi^{\prime}_{C}\\\pi^{\prime}_{c}\in PM_{c}\left(\pm w^{-1}\left(c\right)\cap I_{V}\right)\\\pi^{\prime}\vee\left\{\pm I_{k}:k\in V\right\}=1_{I_{V}}}}N^{\chi\left(\left.\gamma\right|_{I_{V}},\left.\delta_{\varepsilon}\pi^{\prime}\delta_{\varepsilon}\right|_{I_{V}}\right)-2\left|V\right|}f_{1}\left(\pi^{\prime}_{1}\right)\cdots f_{C}\left(\pi^{\prime}_{C}\right)\]
where the last condition under the summation sign means that \(\pi^{\prime}\) must connect the \(\pm I_{k}\) for all \(k\) in \(V\).  If \(V=\left\{i_{1},\ldots,i_{s}\right\}\), then each term in the product is the conjectured expression for the cumulant \(k_{s}\left(Y_{i_{1}},\ldots,Y_{i_{s}}\right)\).  This demonstrates that the conjectured expression for cumulants satisfies the moment-cumulant formula.
\end{proof}

We will also find it useful to consider expressions in which terms in the algebras generated by each ensemble have been centred.  A cumulant of products of centred terms may be interpreted in terms of the Principle of Inclusion and Exclusion, a generalization to functions of the formulas for calculating the number of elements in certain subsets and not in others (see, e.g. \cite{MR1311922}, Chapter 5, for more details and a proof).

\begin{lemma}[The Principle of Inclusion and Exclusion]
\label{inclusion-exclusion}
Let 
\[f,g:\left\{I\subseteq\left[n\right]\right\}\rightarrow\mathbb{C}\textrm{.}\]
Then the following are equivalent:
\[g\left(I\right)=\sum_{J\supseteq I}f\left(J\right)\textrm{;}\]
\[f\left(I\right)=\sum_{J\supseteq I}\left(-1\right)^{\left|J\setminus I\right|}g\left(J\right)\textrm{.}\]
\end{lemma}

This allows us to write a cumulant of products of centred terms satisfying the last condition of (\ref{general}) (multiplicativity of \(f\)) as a sum over the permutations with certain connectedness properties.  In addition to being a connected surface, each interval of edges corresponding to a centred term must be connected to at least one other.

\begin{lemma}
\label{centred}
Let \(p_{1},\ldots,p_{r}\) be positive integers and let \(p:=p_{1}+\cdots+p_{r}\).

Let \(w:\left[p\right]\rightarrow\left[C\right]\) be a word in the set of colours \(\left[C\right]\).  Associate with each colour an independent matrix ensemble \(\left\{X_{c}^{\left(\lambda\right)}\right\}_{\lambda\in\Lambda}\) satisfying (\ref{general}).

Let \(n_{1},\ldots,n_{p}\) be positive integers, let \(n:=n_{1}+\cdots+n_{p}\), and let
\begin{multline*}
\gamma=\left(1,\ldots,n_{1}+\cdots+n_{p_{1}}\right)\left(n_{1}+\cdots+n_{p_{1}}+1,\ldots,n_{1}+\cdots+n_{p_{2}}\right)\cdots\\\left(n_{1}+\cdots+n_{p_{r-1}}+1,\ldots,n\right)\textrm{;}
\end{multline*}
let \(\varepsilon:\left[n\right]\rightarrow\left\{1,-1\right\}\),
and let
\[A_{k}=\prod_{i=n_{1}+\cdots+n_{k-1}+1}^{n_{1}+\cdots+n_{k}}X_{w\left(k\right)}^{\left(\varepsilon\left(i\right)\lambda_{i}\right)}\textrm{.}\]
For \(k\in\left[p\right]\), let
\[I_{k}:=\left[n_{1}+\cdots+n_{k-1}+1,n_{1}+\cdots+n_{k}\right]\textrm{.}\]
For \(K\subseteq\left[p\right]\), let \(I_{K}=\bigcup_{k\in K}I_{k}\).  Let \(V_{k}=I_{\left[p_{1}+\cdots+p_{k-1}+1,p_{1}+\cdots+p_{k}\right]}\).  Then
\begin{multline*}
k_{r}\left(\mathrm{tr}\left(\mathaccent"7017{A}_{1}\cdots \mathaccent"7017{A}_{p_{1}}\right),\ldots,\mathrm{tr}\left(\mathaccent"7017{A}_{p_{1}+\cdots+p_{r-1}+1}\cdots\mathaccent"7017{A}_{p}\right)\right)
\\=\sum_{\substack{\pi=\pi_{1}\cdots\pi_{C}\\\pi_{c}\in PM_{c}\left(\pm I_{w^{-1}\left(c\right)}\right)\\\pi\vee\left\{\pm V_{l}\right\}_{l=1}^{r}=1_{n}\\\pm I_{k}\notin\pi\vee\left\{\pm I_{k}\right\}_{k=1}^{p}}}N^{\chi\left(\gamma,\delta_{\varepsilon}\pi\delta_{\varepsilon}\right)-2r}f_{1}\left(\pi_{1}\right)\cdots f_{C}\left(\pi_{C}\right)
\end{multline*}
(the third line under the summation sign says that \(\pi\) connects the blocks \(\pm V_{k}\), and the fourth that it connects each block \(\pm I_{k}\) to at least one other).
\end{lemma}
\begin{proof}
Expanding the left-hand side expression, we get
\begin{multline*}
\sum_{K\subseteq\left[p\right]}\left(-1\right)^{\left|K\right|}\prod_{k\in K}\mathbb{E}\left(\mathrm{tr}\left(A_{k}\right)\right)\\\times k_{r}\left(\mathrm{tr}\left(\prod_{k\in\left[1,p_{1}\right]\setminus K}A_{k}\right),\ldots,\mathrm{tr}\left(\prod_{k\in\left[p_{1}+\cdots+p_{r-1}+1,p\right]\setminus K}A_{k}\right)\right)\textrm{.}
\end{multline*}
For each \(k\in\left[p\right]\), define the single cycle permutation
\[\gamma_{k}:=\left(n_{1}+\cdots+n_{k-1}+1,\ldots,1+\cdots+n_{k}\right)\textrm{.}\]
Applying Lemmas~\ref{colour} and \ref{cumulant} to our previous expression, we get
\begin{multline*}
\sum_{K\subseteq\left[p\right]}\left(-1\right)^{\left|K\right|}\prod_{k\in K}\sum_{\pi\in PM_{w\left(k\right)}\left(\pm I_{k}\right)}N^{\chi\left(\gamma_{k},\delta_{\varepsilon}\pi\delta_{\varepsilon}\right)-2}f_{w\left(k\right)}\left(\pi\right)\\\times\sum_{\substack{\pi=\pi_{1}\cdots\pi_{C}\\\pi_{c}\in PM_{c}\left(\pm I_{w^{-1}\left(c\right)\setminus K}\right)\\\pi\vee\left\{\pm V_{l}\right\}_{l=1}^{r}=1_{\pm\left[n\right]}}}N^{\chi\left(\left.\gamma\right|_{I_{\left[p\right]\setminus K}},\delta_{\varepsilon}\pi\delta_{\varepsilon}\right)-2r}f_{1}\left(\pi_{1}\right)\cdots f_{C}\left(\pi_{C}\right)\textrm{.}
\end{multline*}
(We note that if any \(\left[p_{1}+\cdots+p_{k-1}+1,p_{1}+\cdots+p_{k}\right]\setminus K\) is empty then one of the entries of the cumulant is \(\mathrm{tr}\left(I_{N}\right)=1\), so the cumulant is zero and the term can be ignored.  Thus we can assume that \(\#\left(\left.\gamma\right|_{I_{\left[p\right]\setminus K}}\right)=r\).)

For a given \(K\), the permutations in each sum act on disjoint sets, so we can express it as a sum over permutations on all of \(\pm\left[n\right]\).  Let
\[\gamma^{\prime}:=\left(\prod_{k\in K}\gamma_{k}\right)\left.\gamma\right|_{I_{\left[p\right]\setminus K}}\textrm{.}\]
Then
\begin{multline*}
\sum_{K\subseteq\left[p\right]}\left(-1\right)^{\left|K\right|}\prod_{k\in K}\sum_{\substack{\pi=\pi_{1}\cdots\pi_{C}\\\pi_{c}\in PM_{c}\left(\pm I_{w^{-1}\left(c\right)}\right)\\\pi\vee\left\{\pm V_{l}\right\}_{l=1}^{r}=1_{\pm\left[n\right]}\\\pm I_{k}\in\pi\vee\left\{\pm I_{k}:k\in K\right\}}}N^{\chi\left(\gamma^{\prime},\delta_{\varepsilon}\pi\delta_{\varepsilon}\right)-2\#\left(\gamma^{\prime}\right)}f_{1}\left(\pi_{1}\right)\cdots f_{C}\left(\pi_{C}\right)
\end{multline*}
(where the last condition under the summation sign means that \(\pi\) does not connect \(\pm I_{k}\) to any other \(\pm I_{l}\), for all \(k\in K\)).

We now show that the exponent on \(N\) is equal to the exponent if permutation \(\gamma^{\prime}\) were replaced with \(\gamma\).  We recall that multiplying a permutation \(\pi\) by a transposition \(\left(k,l\right)\) (on the left or right) joins the orbits of \(\pi\) containing \(k\) and \(l\) if they are in separate orbits, reducing the number of orbits by \(1\), and splits the orbit containing \(k\) and \(l\) if they are in the same orbit, increasing the number of orbits by \(1\).

We note that
\[\gamma=\gamma^{\prime}\left(\prod_{k\in K}\left(n_{1}+\cdots+n_{k},\left.\gamma^{-1}\right|_{\left[n\right]\setminus K}\left(n_{1}+\cdots+n_{k}\right)\right)\right)^{-1}\]
(straightforward calculation; the product of transpositions is in increasing order in \(K\) and inverting the product reverses the order of the transpositions).  Each transposition connects an originally disconnected interval to an orbit of \(\gamma^{\prime}\), so \(\#\left(\gamma\right)=\#\left(\gamma^{\prime}\right)-\left|K\right|\).

We then express \(\gamma_{-}^{-1}\delta_{\varepsilon}\pi\delta_{\varepsilon}\gamma_{+}\) as \(\gamma_{-}^{\prime-1}\delta_{\varepsilon}\pi\delta_{\varepsilon}\gamma_{+}^{\prime}\) left- and right-multiplied by transpositions.  The permutation \(\gamma_{-}^{\prime-1}\delta_{\varepsilon}\pi\delta_{\varepsilon}\gamma_{+}^{\prime}\) does not connect any \(I_{k}\) or \(-I_{k}\) for \(k\in K\) to any other set, so each of these transpositions reduces the number of cycles by \(1\).  Thus \(\#\left(\gamma_{-}^{-1}\delta_{\varepsilon}\pi\delta_{\varepsilon}\gamma_{+}\right)=\#\left(\gamma_{-}^{\prime-1}\delta_{\varepsilon}\pi\delta_{\varepsilon}\gamma_{+}^{\prime}\right)-2\left|K\right|\).

Changing the exponent on \(N\) accordingly, we have:
\[\sum_{K\subseteq\left[p\right]}\left(-1\right)^{\left|K\right|}\sum_{\substack{\pi=\pi_{1}\cdots\pi_{C}\\\pi_{c}\in PM_{c}\left(I_{w^{-1}\left(c\right)}\right)\\\pi\vee\left\{\pm V_{l}\right\}_{l=1}^{r}=1_{\pm\left[n\right]}\\\pm I_{k}\in\pi\vee\left\{\pm I_{k}:k\in K\right\}}}N^{\chi\left(\gamma,\delta_{\varepsilon}\pi\delta_{\varepsilon}\right)-2\#\left(\gamma\right)}f_{1}\left(\pi_{1}\right)\cdots f_{C}\left(\pi_{C}\right)\textrm{.}\]

We can then interpret this expression in terms of Lemma~\ref{inclusion-exclusion}.  In our case, we will let \(f,g:\left\{K\subseteq\left[p\right]\right\}\rightarrow\mathbb{C}\), with \(g\left(K\right)\) the sum of terms over \(\pi\) which do not connect the intervals \(\pm I_{k}\) to any other interval for all \(k\in K\), and \(f\left(K\right)\) the sum of terms over \(\pi\) which do not connect the intervals \(\pm I_{k}\) to any other interval for exactly the \(k\in K\) (i.e. any other interval is connected to another).  We note that \(f\) and \(g\) satisfy the hypotheses of Lemma~\ref{inclusion-exclusion}.  The desired quantity is equal to \(f\left(\emptyset\right)\), from which the result follows.
\end{proof}

\section{Asymptotic calculations}
\label{asymptotics}

We may find upper bounds on the order of \(N\) and characterize highest-order terms in many of the above formulas using the following well-known result (see, e.g. \cite{MR1396978}, and \cite{MR2052516} for a proof):

\begin{theorem}
\label{geodesic}
Let \(\pi,\rho\in S\left(I\right)\) for some finite set \(I\).  Then
\[\#\left(\pi\right)+\#\left(\pi\rho\right)+\#\left(\rho\right)\leq\left|I\right|+2\#\langle\pi,\rho\rangle\textrm{.}\]
\end{theorem}

We see that if \(\gamma\in S_{n}\) and \(\pi\in PM\left(\pm\left[n\right]\right)\) represent a connected surface (see Remarks~\ref{premaps} and \ref{hypergraphs}), then \(\chi\left(\gamma,\pi\right)\leq 2\), as we would expect of an Euler characteristic:

\begin{lemma}
\label{sphere}
Let \(\gamma\in S_{n}\), and let \(\left\{V_{1},\ldots,V_{r}\right\}\in{\cal P}\left(n\right)\) be the orbits of \(\gamma\).  If \(\pi\in PM\left(\pm\left[n\right]\right)\) connects the blocks of \(\left\{\pm V_{1},\ldots,\pm V_{r}\right\}\), then \(\chi\left(\gamma,\pi\right)\leq 2\).
\end{lemma}
\begin{proof}
If an orbit of \(\langle\gamma_{+}\gamma_{-}^{-1},\pi\rangle\) does not contain at least one of \(V_{k}\) or \(-V_{k}\), then \(\pi\) does not connect the block \(\pm V_{k}\).  Thus \(\#\langle\gamma_{+}\gamma_{-}^{-1},\pi\rangle\leq 2\).  The result follows.
\end{proof}

\subsection{Highest order terms and noncrossing conditions}
\label{noncrossing}

For a given \(\pi\), a \(\rho\) satisfying the equality (typically the ones which will contribute highest order terms in \(N\)) may often be interpreted as a noncrossing diagram on the cycles of \(\pi\).  For \(\pi\) with one or two cycles and \(\langle\pi,\rho\rangle\) transitive, we state conditions equivalent to satisfying the equality.  See the original references for proofs and diagrams.

\begin{remark}
We are following the conventions of \cite{MR0404045, MR2036721}, where \(\pi\) is thought of as enumerating the hyperedges, which should also be counterclockwise; as opposed to those of \cite{MR1475837, MR2052516}, where a standard \(\pi\) should follow the sense of \(\gamma\).  The \(\pi\) which here satisfy or violate the various conditions are the inverses of the \(\pi\) of the latter sources.
\end{remark}

\begin{definition}
Let \(\gamma\in S\left(I\right)\) for some finite set \(I\) be a permutation with a single cycle, and let \(\pi\in S\left(I\right)\).

We call \(\pi\) {\em disc nonstandard (relative to \(\gamma\))} if there are three distinct elements \(a,b,c\in I\) such that \(\left.\gamma\right|_{\left\{a,b,c\right\}}=\left(a,b,c\right)\) and \(\left.\pi\right|_{\left\{a,b,c\right\}}=\left(a,b,c\right)\).  We call \(\pi\) {\em disc standard (relative to \(\gamma\))} if there are no such elements.

We call \(\pi\) {\em disc crossing (relative to \(\gamma\))} if there are four distinct elements \(a,b,c,d\in I\) such that \(\left.\gamma\right|_{\left\{a,b,c,d\right\}}=\left(a,b,c,d\right)\) but \(\left.\pi\right|_{\left\{a,b,c,d\right\}}=\left(a,c\right)\left(b,d\right)\).  We call \(\pi\) {\em disc noncrossing (relative to \(\gamma\))} if it is neither disc nonstandard nor disc crossing.

We denote the set of disc-noncrossing permutations on \(I\) relative to \(\gamma\) by \(S_{\mathrm{disc-nc}}\left(\gamma\right)\).
\end{definition}

The following theorem is from \cite{MR1475837}:

\begin{theorem}[Biane]
\label{disc}
Let \(\gamma,\pi\in S\left(I\right)\) for some finite set \(I\), where \(\gamma\) has a single cycle.  Then \(\gamma\) and \(\pi\) satisfy the equality in Theorem~\ref{geodesic}, or in this case
\[\#\left(\pi\right)+\#\left(\gamma^{-1}\pi^{-1}\right)=\left|I\right|+1\textrm{,}\]
if and only if \(\pi\in S_{\mathrm{disc-nc}}\left(\gamma\right)\).
\end{theorem}

Similar conditions may be found for permutations connecting two cycles:
\begin{definition}
Let \(\gamma\in S\left(I\right)\) for some finite set \(I\) be a permutation with two cycles (which we will refer to as \(\gamma_{\mathrm{ext}}\) and \(\gamma_{\mathrm{int}}\)), and let \(\pi\in S\left(I\right)\).

We say that \(\pi\) is {\em annular nonstandard (relative to \(\gamma\))} if one of the two following conditions holds:
\begin{enumerate}
	\item there are \(a,b,c\in I\) such that \(\left.\gamma\right|_{\left\{a,b,c\right\}}=\left(a,b,c\right)\) and \(\left.\pi\right|_{\left\{a,b,c\right\}}=\left(a,b,c\right)\),
	\item there are \(a,b,c,d\in I\) such that \(\left.\gamma\right|_{\left\{a,b,c,d\right\}}=\left(a,b\right)\left(c,d\right)\) but \(\left.\pi\right|_{\left\{a,b,c,d\right\}}=\left(a,c,b,d\right)\).
\end{enumerate}
We call \(\pi\) {\em annular standard (relative to \(\gamma\))} if neither of these conditions holds.

Let \(x\in\gamma_{\mathrm{ext}}\) and \(y\in\gamma_{\mathrm{int}}\).  We define a permutation \(\lambda_{x,y}\) on \(I\setminus\left\{x,y\right\}\) by letting \(\lambda_{x,y}\left(\gamma^{-1}\left(x\right)\right)=\gamma\left(y\right)\) and \(\lambda_{x,y}\left(\gamma^{-1}\left(y\right)\right)=\gamma\left(x\right)\) (we will generally be assuming that \(\gamma_{\mathrm{ext}}\) and \(\gamma_{\mathrm{int}}\) each have at least two elements), and letting \(\lambda_{x,y}\left(a\right)=\gamma\left(a\right)\) otherwise.  We will say that \(\pi\) is {\em annular crossing (relative to \(\gamma\))} if one of the three following conditions holds:
\begin{enumerate}
	\item there are elements \(a,b,c,d\in I\) such that \(\left.\gamma\right|_{\left\{a,b,c,d\right\}}=\left(a,b,c,d\right)\) but \(\left.\pi\right|_{\left\{a,b,c,d\right\}}=\left(a,c\right)\left(b,d\right)\),
	\item there are elements \(a,b,c,x,y\in I\), \(x\in\gamma_{\mathrm{ext}}\) and \(y\in\gamma_{\mathrm{int}}\), such that \(\left.\lambda_{x,y}\right|_{\left\{a,b,c\right\}}=\left(a,b,c\right)\) and \(\left.\pi\right|_{\left\{a,b,c,x,y\right\}}=\left(a,b,c\right)\left(x,y\right)\),
	\item there are elements \(a,b,c,d,x,y\in I\), \(x\in\gamma_{\mathrm{ext}}\) and \(y\in\gamma_{\mathrm{int}}\), such that \(\left.\lambda_{x,y}\right|_{\left\{a,b,c,d\right\}}=\left(a,b,c,d\right)\) but \(\left.\pi\right|_{\left\{a,b,c,d,x,y\right\}}=\left(a,c\right)\left(b,d\right)\left(x,y\right)\).
\end{enumerate}
We call \(\pi\) {\em annular noncrossing (relative to \(\gamma\))} if it is neither annular nonstandard nor annular crossing.

We denote the set of connected annular-noncrossing permutations (those for which \(\langle\gamma,\pi\rangle\) is transitive) by \(S_{\mathrm{ann-nc}}\left(\gamma\right)\).
\end{definition}

The following theorem is from \cite{MR2052516}:

\begin{theorem}[Mingo and Nica]
\label{annulus}
Let \(\gamma\in S\left(I\right)\) for some finite set \(I\) be a permutation with two cycles, and let \(\pi\in S\left(I\right)\) with \(\pi\) connecting the cycles of \(\gamma\) (i.e. \(\#\langle\gamma,\pi\rangle=1\)).  Then \(\gamma\) and \(\pi\) satisfy the equality in Theorem~\ref{geodesic}, or in this case
\[\#\left(\pi\right)+\#\left(\gamma^{-1}\pi^{-1}\right)=\left|I\right|\textrm{,}\]
if and only if \(\pi\in S_{\mathrm{ann-nc}}\left(\gamma\right)\).
\end{theorem}

We generalize these results to the nonorientable case.

\begin{lemma}
\label{lemma: unoriented disc noncrossing}
Let \(\gamma\in S\left(I\right)\), where \(\gamma\) has a single cycle.  Let \(\pi\in PM\left(\pm I\right)\).  Then \(\chi\left(\gamma,\pi\right)=2\) if and only if \(\pi\) does not connect \(I\) and \(-I\) and \(\left.\pi\right|_{I}\in S_{\mathrm{disc-nc}}\left(\gamma\right)\).
\end{lemma}
\begin{proof}
If \(\chi\left(\gamma,\pi\right)=2\), then by Theorem~\ref{geodesic}, \(\#\langle\gamma_{+}\gamma_{-}^{-1},\pi\rangle\geq 2\) so \(\pi\) must not connect the orbits of \(\gamma_{+}\gamma_{-}^{-1}\), \(I\) and \(-I\).  Thus \(\left.\gamma_{+}^{-1}\pi^{-1}\gamma_{-}\right|_{I}=\gamma^{-1}\pi^{-1}\), and since \(\pi\) and \(\gamma_{+}^{-1}\pi^{-1}\gamma_{-}\) have the same number of cycles on \(I\) as on \(-I\), \(\#\left(\left.\pi\right|_{I}\right)+\#\left(\left.\gamma^{-1}\pi^{-1}\right|_{I}\right)=\left|I\right|+1\).  Thus \(\left.\pi\right|_{I}\in S_{\mathrm{disc-nc}}\left(\gamma\right)\).

Conversely, if \(\pi\) does not connect \(I\) and \(-I\) and \(\left.\pi\right|_{I}\in S_{\mathrm{disc-nc}}\left(\gamma\right)\), we calculate similarly that \(\chi\left(\gamma,\pi\right)=2\).
\end{proof}

\begin{lemma}
\label{lemma: unoriented annular noncrossing}
Let \(\gamma\in S\left(I\right)\), where \(\gamma\) has two orbits, \(V_{1}\) and \(V_{2}\).  Let \(\pi\in PM\left(\pm I\right)\) connect \(\pm V_{1}\) and \(\pm V_{2}\).  Then \(\chi\left(\gamma,\pi\right)=2\) if and only if, for some choice of sign \(\varepsilon=\pm 1\), \(\pi\) does not connect \(V_{1}\cup\varepsilon V_{2}\) to \(\left(-V_{1}\right)\cup\left(-\varepsilon V_{2}\right)\) and \(\left.\pi\right|_{V_{1}\cup\varepsilon V_{2}}\in S_{\mathrm{ann-nc}}\left(\left.\gamma_{+}\gamma_{-}^{-1}\right|_{V_{1}\cup\varepsilon V_{2}}\right)\).
\end{lemma}
\begin{proof}
Assume \(\chi\left(\gamma,\pi\right)=2\) and \(\pi\) connects \(\pm V_{1}\) and \(\pm V_{2}\).  Being a premap, the permutation \(\pi\) must connect \(V_{1}\) to either \(V_{2}\) or \(-V_{2}\) and \(-V_{1}\) to the other.  We can calculate from \(\chi\left(\gamma,\pi\right)\) and Theorem~\ref{geodesic}, however, that \(\#\langle\gamma_{+}\gamma_{-}^{-1},\pi\rangle\geq 2\), so \(\pi\), and hence \(\gamma_{+}^{-1}\pi^{-1}\gamma_{-}\), must not further connect these blocks.  We calculate that \(\#\left(\left.\pi\right|_{V_{1}\cup\left(\varepsilon V_{2}\right)}\right)+\#\left(\left.\gamma_{+}^{-1}\gamma_{-}\pi^{-1}\right|_{V_{1}\cup\left(\varepsilon V_{2}\right)}\right)=\left|I\right|\), so \(\left.\pi\right|_{V_{1}\cup\left(\varepsilon V_{2}\right)}\in S_{\mathrm{ann-nc}}\left(\left.\gamma_{+}\gamma_{-}^{-1}\right|_{V_{1}\cup\left(\varepsilon V_{2}\right)}\right)\).

Conversely, if \(\pi\) does not connect \(V_{1}\cup\left(\varepsilon V_{2}\right)\) to \(\left(-V_{1}\right)\cup\left(-\varepsilon V_{2}\right)\) and \(\left.\pi\right|_{V_{1}\cup\left(\varepsilon V_{2}\right)}\in S_{\mathrm{ann-nc}}\left(\left.\gamma_{+}\gamma_{-}^{-1}\right|_{V_{1}\cup\left(\varepsilon V_{2}\right)}\right)\), we calculate that \(\chi\left(\gamma,\pi\right)=2\).
\end{proof}

\subsection{Limit distributions}
\label{limit}

Asymptotically, the moments of the real ensembles are equal to those of their complex analogues.  This is not surprising, since intuitively, highest order terms correspond to spheres, which must be orientable, and thus must have untwisted edge-identifications, and hence these terms appear in the complex expansion as well.  The following lemma gives us an expression for the asymptotic value of the moments of any ensemble satisfying (\ref{general}) (not requiring the multiplicativity of \(f\)):

\begin{lemma}
\label{moment}
Let \(\left\{X_{\lambda}\right\}_{\lambda\in\Lambda}\) be an ensemble of random matrices satisfying (\ref{general}) with subset of the premaps \(PM_{c}\) and function \(f_{c}\).  Then for
\[\gamma=\left(1,\ldots,n\right)\in S_{n}\]
and \(\varepsilon:\left[n\right]\rightarrow\left\{1,-1\right\}\) we have:
\begin{multline*}
\lim_{N\rightarrow\infty}\mathbb{E}\left(\mathrm{tr}\left(X_{\lambda_{1}}^{\left(\varepsilon\left(1\right)\right)}\cdots X_{\lambda_{n}}^{\left(\varepsilon\left(n\right)\right)}\right)\right)\\=\sum_{\substack{\pi\in S_{\mathrm{disc-nc}}\left(\gamma\right)\\\delta_{\varepsilon}\pi_{+}\pi_{-}^{-1}\delta_{\varepsilon}\in PM_{c}\left(\pm\left[n\right]\right)}}\lim_{N\rightarrow\infty}f_{c}\left(\delta_{\varepsilon}\pi_{+}\pi_{-}^{-1}\delta_{\varepsilon}\right)\textrm{.}
\end{multline*}
\end{lemma}
\begin{proof}
We know that
\[\mathbb{E}\left(\mathrm{tr}\left(X_{\lambda_{1}}^{\left(\varepsilon\left(1\right)\right)}\cdots X_{\lambda_{n}}^{\left(\varepsilon\left(n\right)\right)}\right)\right)\\=\sum_{\rho\in PM_{c}\left(\pm\left[n\right]\right)}N^{\chi\left(\gamma,\delta_{\varepsilon}\rho\delta_{\varepsilon}\right)-2}f_{c}\left(\rho\right)\textrm{.}\]
From Lemma~\ref{sphere}, \(\chi\left(\gamma,\delta_{\varepsilon}\rho\delta_{\varepsilon}\right)\leq 2\), so the exponent on \(N\) is less than or equal to \(0\) and the limit exists as \(N\rightarrow\infty\).  By Lemma~\ref{lemma: unoriented disc noncrossing}, the terms which do not vanish as \(N\rightarrow\infty\) are those where \(\delta_{\varepsilon}\rho\delta_{\varepsilon}\) does not connect \(\left[n\right]\) and \(-\left[n\right]\) and \(\left.\delta_{\varepsilon}\rho\delta_{\varepsilon}\right|_{\left[n\right]}\in S_{\mathrm{disc-nc}}\left(\gamma\right)\).  Then \(\delta_{\varepsilon}\rho\delta_{\varepsilon}=\pi_{+}\pi_{-}^{-1}\), and the result follows.
\end{proof}

Mixed moments in independent matrices from the various ensembles may be expressed similarly using Lemma~\ref{colour}; however we will not need this result.

By similar arguments, fluctuations of any matrix ensemble satisfying (\ref{general}), including the multiplicativity of \(f\), may be expressed similarly:

\begin{lemma}
\label{fluctuation}
Let \(\left\{X_{\lambda}\right\}_{\lambda\in\Lambda}\) be an ensemble of random matrices satisfying (\ref{general}) with subset of the premaps \(PM_{c}\) and function \(f_{c}\).  Let
\[\gamma:=\left(1,\ldots,m\right)\left(m+1,\ldots,m+n\right)\textrm{,}\]
let
\[\gamma_{\mathrm{op}}:=\left(1,\ldots,m\right)\left(-m-n,\ldots,-m-1\right)\textrm{,}\]
and let \(\varepsilon:\left[m+n\right]\rightarrow\left\{1,-1\right\}\).  Then
\begin{multline*}
\lim_{N\rightarrow\infty}k_{2}\left(\mathrm{Tr}\left(X_{\lambda_{1}}^{\left(\varepsilon\left(1\right)\right)}\cdots X_{\lambda_{m}}^{\left(\varepsilon\left(m\right)\right)}\right),\mathrm{Tr}\left(X_{\lambda_{m+1}}^{\left(\varepsilon\left(m+1\right)\right)}\cdots X_{\lambda_{m+n}}^{\left(\varepsilon\left(m+n\right)\right)}\right)\right)\\=\sum_{\substack{\pi\in S_{\mathrm{ann-nc}}\left(\gamma\right)\\\delta_{\varepsilon}\pi_{+}\pi_{-}^{-1}\delta_{\varepsilon}\in PM_{c}\left(\pm\left[m+n\right]\right)}}\lim_{N\rightarrow\infty}f_{c}\left(\delta_{\varepsilon}\pi_{+}\pi_{-}^{-1}\delta_{\varepsilon}\right)\\+\sum_{\substack{\pi\in S_{\mathrm{ann-nc}}\left(\gamma_{\mathrm{op}}\right)\\\delta_{\varepsilon}\pi_{+}\pi_{-}^{-1}\delta_{\varepsilon}\in PM_{c}\left(\pm\left[m+n\right]\right)}}\lim_{N\rightarrow\infty}f_{c}\left(\delta_{\varepsilon}\pi_{+}\pi_{-}^{-1}\delta_{\varepsilon}\right)\textrm{.}
\end{multline*}
\end{lemma}
\begin{proof}
By Lemma~\ref{cumulant}, we have that
\begin{multline*}
k_{2}\left(\mathrm{Tr}\left(X_{\lambda_{1}}^{\left(\varepsilon\left(1\right)\right)}\cdots X_{\lambda_{m}}^{\left(\varepsilon\left(m\right)\right)}\right),\mathrm{Tr}\left(X_{\lambda_{m+1}}^{\left(\varepsilon\left(m+1\right)\right)}\cdots X_{\lambda_{m+n}}^{\left(\varepsilon\left(m+n\right)\right)}\right)\right)\\=N^{2}\sum_{\substack{\rho\in PM_{c}\left(\pm\left[m+n\right]\right)\\\rho\vee\left\{\pm\left[m\right],\pm\left[m+1,m+n\right]\right\}=1_{\pm\left[m+n\right]}}}N^{\chi\left(\gamma,\delta_{\varepsilon}\rho\delta_{\varepsilon}\right)-4}f_{c}\left(\rho\right)
\end{multline*}
where the second condition under the summation sign means that \(\rho\) must connect \(\pm\left[m\right]\) and \(\pm\left[m+1,m+n\right]\).  By Lemma~\ref{sphere}, \(\chi\left(\gamma,\delta_{\varepsilon}\rho\delta_{\varepsilon}\right)\leq 2\), so the limit exists as \(N\rightarrow\infty\).  By Lemma~\ref{lemma: unoriented annular noncrossing}, the terms which do not vanish are those where \(\delta_{\varepsilon}\rho\delta_{\varepsilon}\) does not connect \(V_{1}\cup\left(\varepsilon_{\rho}V_{2}\right)\) and \(\left(-V_{1}\right)\cup\left(-\varepsilon_{\rho}V_{2}\right)\) for some choice of sign \(\varepsilon_{\rho}\) and \(\left.\delta_{\varepsilon}\rho\delta_{\varepsilon}\right|_{V_{1}\cup\left(\varepsilon_{\rho}V_{2}\right)}\) is in either \(S_{\mathrm{ann-nc}}\left(\gamma\right)\) or \(S_{\mathrm{ann-nc}}\left(\gamma_{\mathrm{op}}\right)\).  The result follows.
\end{proof}

It follows that the three matrix ensembles, or any matrix satisfying (\ref{general}), including the multiplicativity of \(f\), has a second-order limit distribution:

\begin{lemma}
Let \(\left\{X_{\lambda}\right\}_{\lambda\in\Lambda}\) be an ensemble of random matrices satisfying (\ref{general}) with function \(f_{c}\) and subset of the premaps \(PM_{c}\).  Then this ensemble has second-order limit distribution \(\left(A,\varphi_{1},\varphi_{2}\right)\) where \(A\) is the algebra of noncommutative polynomials on indeterminates \(x_{\lambda}\), \(\lambda\in\pm\Lambda\), and \(\varphi_{1}\) and \(\varphi_{2}\) are defined by extending the following expressions by linearity:
\[\varphi_{1}\left(x_{\lambda_{1}}^{\left(\varepsilon\left(1\right)\right)}\cdots x_{\lambda_{n}}^{\left(\varepsilon\left(n\right)\right)}\right)=\sum_{\substack{\pi\in S_{\mathrm{disc-nc}}\left(\gamma\right)\\\delta_{\varepsilon}\pi_{+}\pi_{-}^{-1}\delta_{\varepsilon}\in PM_{c}\left(\pm\left[n\right]\right)}}\lim_{N\rightarrow\infty}f_{c}\left(\delta_{\varepsilon}\pi_{+}\pi_{-}^{-1}\delta_{\varepsilon}\right)\]
where \(\gamma=\left(1,\ldots,n\right)\) and
\begin{multline*}
\varphi_{2}\left(x_{\lambda_{1}}^{\left(\varepsilon\left(1\right)\right)}\cdots x_{\lambda_{m}}^{\left(\varepsilon\left(m\right)\right)},x_{\lambda_{m+1}}^{\left(\varepsilon\left(m+1\right)\right)}\cdots x_{\lambda_{m+n}}^{\left(\varepsilon\left(m+n\right)\right)}\right)\\=\sum_{\substack{\pi\in S_{\mathrm{ann-nc}}\left(\gamma\right)\\\delta_{\varepsilon}\pi_{+}\pi_{-}^{-1}\delta_{\varepsilon}\in PM_{c}\left(\pm\left[m+n\right]\right)}}\lim_{N\rightarrow\infty}f_{c}\left(\delta_{\varepsilon}\pi_{+}\pi_{-}^{-1}\delta_{\varepsilon}\right)\\+\sum_{\substack{\pi\in S_{\mathrm{ann-nc}}\left(\gamma_{\mathrm{op}}\right)\\\delta_{\varepsilon}\pi_{+}\pi_{-}^{-1}\delta_{\varepsilon}\in PM_{c}\left(\pm\left[m+n\right]\right)}}\lim_{N\rightarrow\infty}f_{c}\left(\delta_{\varepsilon}\pi_{+}\pi_{-}^{-1}\delta_{\varepsilon}\right)
\end{multline*}
where
\[\gamma=\left(1,\ldots,m\right)\left(m+1,\ldots,m+n\right)\]
and
\[\gamma_{\mathrm{op}}\left(1,\ldots,m\right)\left(-m-1,\ldots,-m-n\right)\textrm{.}\]
\end{lemma}
\begin{proof}
Extending the results of Lemmas~\ref{moment} and \ref{fluctuation} by linearity, we can see that \(\left(A,\varphi_{1},\varphi_{2}\right)\) meets the conditions on the first two cumulants of traces.

For \(r\geq 3\), the cumulant
\[k_{r}\left(\mathrm{Tr}\left(p_{1}\left(X_{\lambda_{1,1}},\ldots,X_{\lambda_{1,n_{1}}}\right)\right),\ldots\mathrm{Tr}\left(p_{r}\left(X_{\lambda_{r,1}},\ldots,X_{\lambda_{r,n_{r}}}\right)\right)\right)\]
is the sum of terms of order \(\chi\left(\gamma,\pi\right)-r\) in \(N\), for various permutations \(\gamma\) and \(\pi\).  By Lemma~\ref{sphere}, \(\gamma\) and \(\pi\) must satisfy \(\chi\left(\gamma,\pi\right)\leq 2\).  Thus \(N\) appears with a negative exponent on all terms, which therefore vanish as \(N\rightarrow\infty\).
\end{proof}

\begin{remark}
The moments and fluctuations of these ensembles can be calculated more explicitly using combinatorial expressions for the number of noncrossing diagrams on one or two cycles.  In the real Ginibre and GOE case, moments and fluctuations are calculated by counting appropriate diagrams.  In the Wishart case, terms corresponding to diagrams must be weighted by the trace of the matrices \(D_{k}\) along the appropriate permutation; however, if we take \(D_{k}=I_{M}\) for all \(k\), the calculation reduces to counting appropriate diagrams.

In each case, the value of the moments is equal to that in the complex case, since twisted identifications, or terms in which \(\delta_{\varepsilon}\pi\delta_{\varepsilon}\) connects positive and negative numbers, resulting in lower order terms.  However, noncrossing annular diagrams on both relative orientations of two circles contribute to the fluctuations, and the contribution of each may be different.  In the case of the GOE and Wishart matrices, the same number of diagrams are possible on both relative orientations, but in the Wishart case transposes of the \(D_{k}\) matrices may appear, so the values may be different.

In the real Ginibre case, the moments are given by counting noncrossing \(*\)-pairings, that is, those which pair untransposed terms with transposed terms.  Fluctuations are given by counting annular noncrossing \(*\)-pairings on the expressions in the two traces as well as those where we have reversed the order and the presence or absence of transposes in one of the expressions.  The number of such \(*\)-pairings depends on where the transposes appear.  For more on \(*\)-pairings, see \cite{MR2266879}.

Asymptotically, the GOE has \(n\)th moment given by the number of noncrossing pairings on \(n\) points, that is, \(0\) for \(n\) odd and \(C_{n/2}\) for \(n\) even, where \(C_{k}:=\frac{1}{k+1}\binom{2k}{k}\) is the \(k\)th Catalan number.  Fluctuations are given by noncrossing annular pairings, which are counted in \cite{MR0142470}.  Counting both relative orientations, the asymptotic value of a fluctuation is given by
\[\lim_{N\rightarrow\infty}\mathbb{E}\left(\mathrm{Tr}\left(T^{p}\right),\mathrm{Tr}\left(T^{q}\right)\right)=\left\{\begin{array}{ll}\frac{4}{p+q}\frac{p!}{\frac{p}{2}!\left(\frac{p}{2}-1\right)!}\frac{q!}{\frac{q}{2}!\left(\frac{q}{2}-1\right)!},&\textrm{\(p,q\) even}\\\frac{4}{p+q}\frac{p!}{\left(\frac{p-1}{2}!\right)^{2}}\frac{q!}{\left(\frac{q-1}{2}!\right)^{2}}&\textrm{\(p,q\) odd}\\0,&\textrm{otherwise}\end{array}\right.\textrm{.}\]

The moments of the Wishart ensemble with \(D_{k}=I_{M}\) for all \(k\) are given by counting noncrossing permutations.  These are in one-to-one correspondence with noncrossing pairings with twice as many points on each circle, so the \(n\)th moment is \(C_{n}\), and the fluctuations are given by:
\[\lim_{N\rightarrow\infty}\mathbb{E}\left(\mathrm{Tr}\left(W^{p}\right),\mathrm{Tr}\left(W^{q}\right)\right)=\frac{2}{p+q}\frac{\left(2p\right)!}{p!\left(p-1\right)!}\frac{\left(2q\right)!}{q!\left(q-1\right)!}\textrm{.}\]
\end{remark}

\section{Freeness}
\label{freeness}

\subsection{First-order freeness}

The real matrix ensembles considered in this paper satisfy the same definition of asymptotic freeness as do their complex analogues:

\begin{lemma}
\label{first-order}
Let \(\left[C\right]\) be a set of colours, and associate to each colour \(c\in C\) a matrix ensemble \(\left\{X_{c}^{\left(\lambda\right)}\right\}_{\lambda\in\Lambda_{c}}\) satisfying (\ref{general}), independent from the other ensembles.  Then these ensembles are asymptotically free.
\end{lemma}
\begin{proof}
Let \(w:\left[p\right]\rightarrow\left[C\right]\) be an alternating word in the colours, and let \(A_{k}\) be an element of the algebra generated by the ensemble associated with colour \(w\left(k\right)\).  Then for each \(k\in\left[p\right]\) we may write
\[A_{k}=\prod_{i=n_{1}+\cdots+n_{k-1}+1}^{n_{1}+\cdots+n_{k}}X_{w\left(k\right)}^{\left(\varepsilon\left(i\right)\lambda_{i}\right)}\]
for some positive integers \(n_{1},\ldots,n_{p}\), some \(\lambda_{n_{1}+\cdots+n_{k-1}+1},\ldots,\lambda_{n_{1}+\cdots+n_{k}}\in\Lambda_{w\left(k\right)}\), and some function \(\varepsilon:\left[n\right]\rightarrow\left\{1,-1\right\}\) where \(n:=n_{1}+\cdots+n_{p}\).  Let \(I_{k}=\left[n_{1}+\cdots+n_{k-1}+1,n_{1}+\cdots+n_{k}\right]\) and for \(K\subseteq\left[p\right]\), let \(I_{K}=\bigcup_{k\in K}I_{k}\).  Let
\[\gamma=\left(1,\ldots,n\right)\textrm{.}\]
By Lemma~\ref{centred}, we see that
\[\mathbb{E}\left(\mathrm{tr}\left(\mathaccent"7017{A}_{1}\cdots \mathaccent"7017{A}_{p}\right)\right)=\sum_{\substack{\pi=\pi_{1}\cdots\pi_{C}\\\pi_{c}\in PM_{c}\left(\pm I_{w^{-1}\left(c\right)}\right)\\\pm I_{k}\notin\pi\vee\left\{\pm I_{l}\right\}_{l=1}^{p}}}N^{\chi\left(\gamma,\delta_{\varepsilon}\pi\delta_{\varepsilon}\right)-2}f_{1}\left(\pi_{1}\right)\cdots f_{C}\left(\pi_{C}\right)\]
(where the last condition under the summation sign means that \(\pi\) connects each block \(\pm I_{k}\) to at least one other).  By Lemma~\ref{sphere}, we see that this limit exists, and by Lemma~\ref{lemma: unoriented disc noncrossing}, we see that the term associated to \(\pi\) vanishes when \(N\rightarrow\infty\) unless \(\delta_{\varepsilon}\pi\delta_{\varepsilon}/2\in S_{\mathrm{disc-nc}}\left(\gamma\right)\).

Consider a \(\pi\) satisfying the conditions under the summation sign such that \(\delta_{\varepsilon}\pi\delta_{\varepsilon}\in S_{\mathrm{disc-nc}}\left(\gamma\right)\).  Then \(\delta_{\varepsilon}\pi\delta_{\varepsilon}/2\) must have a cycle connecting two distinct intervals, that is, containing an \(a\) and a \(c\) such that \(a\in I_{k}\) and \(c\in I_{l}\) for some \(k\) and \(l\), \(k<l\).  We know that \(l\neq k+1\), since \(w\) is alternating.  The permutation \(\delta_{\varepsilon}\pi\delta_{\varepsilon}/2\) must have a cycle containing a \(b\in I_{k+1}\) and a \(d\in I_{m}\) for some \(m\neq k+1\).  Since \(\left.\delta_{\varepsilon}\pi\delta_{\varepsilon}\right|_{\left\{a,b,c,d\right\}}=\left(a,c\right)\left(b,d\right)\), we must not have \(\left.\gamma\right|_{\left\{a,b,c,d\right\}}=\left(a,b,c,d\right)\), so \(k+1<m<l\).  By induction, given any connected intervals, we can find a pair of connected intervals whose indices are closer together, and we derive a contradiction.  Thus there are no such \(\pi\), and the expression vanishes as \(N\rightarrow\infty\).
\end{proof}

The noncrossing conditions satisfied by a term of highest order in an expression such as in Lemma~\ref{centred} allow us to further characterize the terms contributing to the asymptotic value of the fluctuations, giving us the expression in the following theorem.  This theorem can be applied to independent algebras of any matrices satisfying (\ref{general}), including real Ginibre, GOE, and Wishart matrices.

\begin{theorem}
\label{main}
Let \(v:\left[p\right]\rightarrow\left[C\right]\) and \(w:\left[q\right]\rightarrow\left[C\right]\) be cyclically alternating words in a set of colours \(\left[C\right]\).  To each colour \(c\), associate an independent ensemble of random matrices \(\left\{X_{c}^{\left(\lambda\right)}\right\}_{\lambda\in\Lambda_{c}}\) satisfying (\ref{general}) with subsets of the premaps \(PM_{c}\) and function \(f_{c}\).  For integers \(m_{1},\cdots,m_{p}\) and \(n_{1},\cdots,n_{q}\), let \(m:=m_{1}+\cdots+m_{p}\) and \(n:=n_{1}+\cdots+n_{q}\).  Let \(\varepsilon:\left[m+n\right]\rightarrow\left\{1,-1\right\}\).  For each \(k\in\left[p\right]\) and \(i\in\left[m_{1}+\cdots+m_{k-1}+1,m_{1}+\ldots+m_{k}\right]\) let \(\lambda_{i}\in\Lambda_{v\left(k\right)}\) and for \(k\in\left[q\right]\) and \(i\in\left[m+n_{1}+\cdots+n_{k-1}+1,m+n_{1}+\cdots+n_{k}\right]\) let \(\lambda_{i}\in\Lambda_{w\left(k\right)}\).  Let
\[A_{k}=\prod_{i=m_{1}+\cdots+m_{k-1}+1}^{m_{1}+\cdots+m_{k}}X_{v\left(k\right)}^{\left(\varepsilon\left(i\right)\lambda_{i}\right)}\]
and
\[B_{k}=\prod_{i=m+n_{1}+\cdots+n_{k-1}+1}^{m+n_{1}+\cdots+n_{k}}X_{w\left(k\right)}^{\left(\varepsilon\left(i\right)\lambda_{i}\right)}\textrm{.}\]
Then if \(p\neq q\)
\[\lim_{N\rightarrow\infty}k_{2}\left(\mathrm{Tr}\left(\mathaccent"7017{A}_{1}\cdots\mathaccent"7017{A}_{p}\right),\mathrm{Tr}\left(\mathaccent"7017{B}_{1}\cdots\mathaccent"7017{B}_{q}\right)\right)=0\textrm{,}\]
and if \(p=q\geq 2\),
\begin{multline*}
\lim_{N\rightarrow\infty}k_{2}\left(\mathrm{Tr}\left(\mathaccent"7017{A}_{1}\cdots\mathaccent"7017{A}_{p}\right),\mathrm{Tr}\left(\mathaccent"7017{B}_{1}\cdots\mathaccent"7017{B}_{p}\right)\right)
\\=\sum_{k=0}^{p-1}\prod_{i=1}^{p}\lim_{N\rightarrow\infty}\left(\mathbb{E}\left(\mathrm{tr}\left(A_{i}B_{k-i}\right)\right)-\mathbb{E}\left(\mathrm{tr}\left(A_{i}\right)\right)\mathbb{E}\left(\mathrm{tr}\left(B_{k-i}\right)\right)\right)\\+\sum_{k=0}^{p-1}\prod_{i=1}^{p}\lim_{N\rightarrow\infty}\left(\mathbb{E}\left(\mathrm{tr}\left(A_{i}B_{k+i}^{T}\right)\right)-\mathbb{E}\left(\mathrm{tr}\left(A_{i}\right)\right)\mathbb{E}\left(\mathrm{tr}\left(B_{k+i}^{T}\right)\right)\right)\textrm{.}
\end{multline*}
\end{theorem}
\begin{proof}
For integer \(k\) (taken modulo \(p\)), let
\[I_{k}:=\left[m_{1}+\cdots+m_{k-1}+1,m_{1}+\cdots+m_{k}\right]\]
and for \(k\) (taken modulo \(q\)), let
\[J_{k}:=\left[m+n_{1}+\cdots+n_{k-1}+1,m+n_{1}+\cdots+n_{k}\right]\textrm{.}\]
For \(K\subseteq\left[p\right]\), let \(I_{K}=\bigcup_{k\in K}I_{k}\), and for \(K\subseteq\left[q\right]\), let \(J_{K}=\bigcup_{k\in K}J_{K}\).

Let
\[\gamma=\left(1,\ldots,m\right)\left(m+1,m+n\right)\textrm{.}\]
By Lemma~\ref{centred}, we know that
\begin{multline*}
k_{2}\left(\mathrm{Tr}\left(\mathaccent"7017{A}_{1}\cdots\mathaccent"7017{A}_{p}\right),\mathrm{Tr}\left(\mathaccent"7017{B}_{1}\cdots\mathaccent"7017{B}_{q}\right)\right)
\\=\sum_{\substack{\pi=\pi_{1}\cdots\pi_{C}\\\pi_{c}\in PM_{c}\left(I_{v^{-1}\left(c\right)}\cup J_{w^{-1}\left(c\right)}\right)\\\pi\vee\left\{\pm\left[m\right],\pm\left[m+1,m+n\right]\right\}=1_{\pm\left[m+n\right]}\\\pm I_{k},\pm J_{k}\notin\pi\vee\left\{\pm I_{l}\right\}_{l=1}^{p}\cup\left\{\pm J_{l}\right\}_{l=1}^{q}}}N^{\chi\left(\gamma,\delta_{\varepsilon}\pi\delta_{\varepsilon}\right)-2}f_{1}\left(\pi_{1}\right)\cdots f_{C}\left(\pi_{C}\right)
\end{multline*}
(the third line under the summation sign says that \(\pi\) connects the blocks \(\pm\left[m\right]\) and \(\pm\left[m+1,m+n\right]\), and the fourth says that \(\pi\) connects each block \(\pm I_{k}\) and \(\pm J_{k}\) to at least one other).  By Lemma~\ref{lemma: unoriented annular noncrossing}, we know that the terms that survive in the limit \(N\rightarrow\infty\) are those such that \(\delta_{\varepsilon}\pi\delta_{\varepsilon}\) doesn't connect \(\left[m\right]\cup\left(\varepsilon_{\pi}\left[m+1,\ldots,m+n\right]\right)\) to \(-\left[m\right]\cup\left(-\varepsilon_{\pi}\left[m+1,m+n\right]\right)\) for some choice of sign \(\varepsilon_{\pi}\), and \(\delta_{\varepsilon}\pi\delta_{\varepsilon}\in S_{\mathrm{ann-nc}}\left(\left.\gamma_{+}\gamma_{-}^{-1}\right|_{\left[m\right]\cup\left(\varepsilon_{\pi}\left[m+1,m+n\right]\right)}\right)\).

We now wish to show that for such a \(\pi\), \(\delta_{\varepsilon}\pi\delta_{\varepsilon}\) must connect intervals in ``spokes'' as shown in Figure~\ref{spoke}, where each spoke can be expanded to a noncrossing diagram which connects the those two intervals.  Let \(\gamma_{\pi}:=\left.\gamma_{+}\gamma_{-}^{-1}\right|_{\left[m\right]\cup\left(\varepsilon_{\pi}\left[m+1,m+n\right]\right)}\).

First we show that \(\delta_{\varepsilon}\pi\delta_{\varepsilon}\) cannot have an orbit connecting two distinct \(I_{k}\) or two distinct \(\varepsilon_{\pi} J_{k}\).  By the arguments in Lemma~\ref{first-order}, we can find an \(a\in I_{k}\), \(b\in I_{k+1}\), \(c\in I_{l}\) and \(d\) such that \(\left.\delta_{\varepsilon}\pi\delta_{\varepsilon}\right|_{\left\{a,b,c,d\right\}}=\left(a,c\right)\left(b,d\right)\), and here we must have \(\left.\gamma_{\pi}\right|_{\left\{a,b,c,d\right\}}=\left(a,b,c\right)\left(d\right)\).  Reversing the roles of \(a\) and \(c\), we may then find an \(x\in I_{l+1}\) sharing an orbit of \(\delta_{\varepsilon}\pi^{-1}\delta_{\varepsilon}\) with a \(y\) from a different interval, and we conclude similarly that \(\left.\gamma_{\pi}\right|_{\left\{a,c,x,y\right\}}=\left(a,c,x\right)\left(y\right)\).  Since \(\left.\gamma_{\pi}\right|_{\left\{a,b,c,x\right\}}=\left(a,b,c,x\right)\), \(\left.\delta_{\varepsilon}\pi\delta_{\varepsilon}\right|_{\left\{a,b,c,x\right\}}=\left(a,c\right)\left(b\right)\left(x\right)\) (first annular-noncrossing condition), so \(\left.\delta_{\varepsilon}\pi\delta_{\varepsilon}\right|_{\left\{a,b,c,d,x,y\right\}}=\left(a,c\right)\left(b,d\right)\left(x,y\right)\).  On the other hand, \(\left.\lambda_{x,y}\right|_{\left\{a,b,c,d\right\}}=\left(a,b,c,d\right)\), violating the third annular-noncrossing condition.  So there must be no such orbit.

We now show that \(\delta_{\varepsilon}\pi\delta_{\varepsilon}\) may not connect an interval in one orbit of \(\gamma_{\pi}\) to two distinct intervals in the another.  If so, let \(c\) and \(y\) belong to the same interval (in the orbit of \(\gamma_{\pi}\) which we will call \(\gamma_{\mathrm{ext}}\)) and share cycles of \(\delta_{\varepsilon}\pi\delta_{\varepsilon}\) with \(a\) and \(x\) respectively, where \(a\) and \(x\) belong to distinct intervals in the other cycle of \(\gamma_{\pi}\) which we will call \(\gamma_{\mathrm{int}}\) (thus these cycles of \(\delta_{\varepsilon}\pi\delta_{\varepsilon}\) must be distinct), and such that if we apply \(\gamma_{\pi}\) repeatedly to \(y\) we will find \(c\) before we find an element of another interval.  If we apply \(\gamma_{\pi}\) repeatedly to \(a\), we must find various \(b\) belonging to another colour before we find \(x\).  At least one of these \(b\) must be connected by a cycle of \(\delta_{\varepsilon}\pi\delta_{\varepsilon}\) to another interval, which by the above arguments must be in \(\gamma_{\mathrm{ext}}\), and because it is of a different colour from \(a\), \(c\), \(x\) and \(y\), in an interval distinct from the one containing \(c\) and \(y\).  We have thus that \(\left.\delta_{\varepsilon}\pi\delta_{\varepsilon}\right|_{\left\{a,b,c,d,x,y\right\}}=\left(a,c\right)\left(b,d\right)\left(x,y\right)\).  However, by our arguments, \(\left.\lambda_{x,y}\right|_{\left\{a,b,c,d\right\}}=\left(a,b,c,d\right)\), violating the third annular-noncrossing condition.

We find thus that each interval in an orbit \(\gamma_{\pi}\) must be connected by \(\delta_{\varepsilon}\pi\delta_{\varepsilon}\) to exactly one other, which must be in the other orbit of \(\gamma_{\pi}\).  We conclude that \(p=q\) for any nonvanishing covariance.

We now show that the diagram on the intervals will be a spoke diagram: if \(x\in I_{k}\) and \(y\in\varepsilon_{\pi} J_{l}\) share an orbit of \(\delta_{\varepsilon}\pi\delta_{\varepsilon}\), then \(I_{k+1}\) must be connected to \(\varepsilon_{\pi} J_{l-\varepsilon_{\pi}}\).  Assume not: then \(\delta_{\varepsilon}\pi\delta_{\varepsilon}\) must connect \(I_{k+1}\) to another \(\varepsilon_{\pi} J_{l^{\prime}}\) and \(\varepsilon_{\pi} J_{l-\varepsilon_{\pi}}\) to another \(I_{k^{\prime}}\).  Let \(a\in I_{k+1}\) and \(c\in\varepsilon_{\pi} J_{l^{\prime}}\) share an orbit of \(\delta_{\varepsilon}\pi\delta_{\varepsilon}\), and let \(d\in\varepsilon_{\pi} J_{l-\varepsilon_{\pi}}\) and \(b\in I_{k^{\prime}}\) share another.  Then \(\left.\delta_{\varepsilon}\pi\delta_{\varepsilon}\right|_{\left\{a,b,c,d,x,y\right\}}=\left(a,c\right)\left(b,d\right)\left(x,y\right)\).  However, \(\left.\lambda_{x,y}\right|_{\left\{a,b,c,d\right\}}=\left(a,b,c,d\right)\), violating the third annular-noncrossing condition.

We now show that each spoke consists of a noncrossing diagram on \(I_{k}\) and \(\varepsilon_{\pi}J_{l}\).  The permutation \(\left.\lambda_{x,y}\right|_{I_{k}\cup\left(\varepsilon_{\pi}J_{l}\right)}\) is identical for any choice of \(x\in\gamma_{\mathrm{ext}}\) and \(y\in\gamma_{\mathrm{int}}\) not in \(I_{k}\cup\left(\varepsilon_{\pi}J_{l}\right)\), and since we are considering \(p=q\geq 2\), such an \(x\) and \(y\) will exist.  We will thus refer to this permutation without specifying \(x\) and \(y\).  The second and third annular-noncrossing conditions become the disc-standard and disc-noncrossing conditions relative to \(\left.\lambda_{x,y}\right|_{I_{k}\cup\left(\varepsilon_{\pi}J_{l}\right)}\).

Conversely, we show that any premap \(\pi\) such that \(\left.\delta_{\varepsilon}\pi\delta_{\varepsilon}\right|_{\gamma_{\pi}}\) has such a spoke arrangement for some choice of sign \(\varepsilon_{\pi}\) (that is, one which connects \(I_{i}\) to only \(\varepsilon_{\pi}J_{k-\varepsilon_{\pi}i}\) and {\em vice versa} for some \(k\) and such that \(\left.\delta_{\varepsilon}\pi\delta_{\varepsilon}\right|_{I_{i}\cup\left(\varepsilon_{\pi}J_{k-\varepsilon_{\pi}i}\right)}\in S_{\mathrm{disc-nc}}\left(\left.\lambda_{x,y}\right|_{I_{i}\cup\left(\varepsilon_{\pi}J_{k-\varepsilon_{\pi}i}\right)}\right)\) for \(x\in\gamma_{\mathrm{ext}}\) and \(y\in\gamma_{\mathrm{int}}\) with \(x,y\notin I_{i}\cup\left(\varepsilon_{\pi}J_{k-\varepsilon_{\pi}i}\right)\)) must be in \(S_{\mathrm{ann-nc}}\left(\gamma_{\pi}\right)\).

Any restriction of \(\lambda_{x,y}\) to \(I_{k}\) or \(\varepsilon_{\pi}J_{l}\) is equal to \(\gamma_{\pi}\) restricted to the same domain.  Any \(\delta_{\varepsilon}\pi\delta_{\varepsilon}\) satisfying the disc-standard condition on \(\lambda_{x,y}\) will then satisfy the first annular-standard condition.

If \(a\), \(b\), \(c\) and \(d\) share a cycle of \(\delta_{\varepsilon}\pi\delta_{\varepsilon}\), with \(a\) and \(b\) in one cycle of \(\gamma_{\pi}\) and \(c\) and \(d\) in the other, then \(a\) and \(b\) must be encountered in some order in \(\lambda_{x,y}\) before any elements of the cycle of \(\gamma_{\pi}\) containing \(c\) and \(d\), so either \(\left(a,b,c\right)\) or \(\left(a,d,b\right)\) is disc nonstandard on \(\lambda_{x,y}\).  Thus \(\delta_{\varepsilon}\pi\delta_{\varepsilon}\) must satisfy the second annular-standard condition.

The elements of any spoke in one of the cycles of \(\gamma_{\pi}\) comprise an interval of that cycle, so if \(\left.\gamma_{\pi}\right|_{\left\{a,b,c,d\right\}}=\left(a,b,c,d\right)\) and \(\left.\delta_{\varepsilon}\pi\delta_{\varepsilon}\right|_{\left\{a,b,c,d\right\}}=\left(a,c\right)\left(b,d\right)\), then one of \(b\) and \(d\) (and hence the other) must be in the same spoke as \(a\) and \(c\).  The elements of this spoke must be disc-crossing relative to \(\lambda_{x,y}\), so \(\delta_{\varepsilon}\pi\delta_{\varepsilon}\) must satisfy the first annular-noncrossing condition.

If \(x^{\prime}\) and \(y^{\prime}\) sharing a cycle of \(\delta_{\varepsilon}\pi\delta_{\varepsilon}\) are on different cycles of \(\gamma_{\pi}\), they divide the other elements of their spoke into two intervals in the cycle \(\left.\lambda_{x,y}\right|_{I_{k}\cup\left(\varepsilon_{\pi}J_{l}\right)}\), say \(K_{1}\) and \(K_{2}\).  Since \(\delta_{\varepsilon}\pi\delta_{\varepsilon}\) is disc-noncrossing relative to \(\lambda_{x,y}\), any cycle of \(\delta_{\varepsilon}\pi\delta_{\varepsilon}\) other than the one containing \(x^{\prime}\) and \(y^{\prime}\) must be contained in one of the \(K_{j}\).  We note that each \(K_{j}\) intersects each cycle of \(\gamma_{\pi}\) in an interval not containing \(x\), \(y\), \(x^{\prime}\) or \(y^{\prime}\), so both \(\lambda_{x,y}\) and \(\lambda_{x^{\prime},y^{\prime}}\) induce a permutation on \(K_{i}\) in which elements are mapped under \(\gamma_{\pi}\) within the interval belonging to a cycle of \(\gamma_{\pi}\), then the last element of an interval is mapped to the first of the other.  Thus \(\left.\lambda_{x^{\prime},y^{\prime}}\right|_{K_{j}}=\left.\lambda_{x,y}\right|_{K_{j}}\).  Since \(\delta_{\varepsilon}\pi\delta_{\varepsilon}\) must then be disc standard relative to \(\lambda_{x,y}\) for any \(x\) and \(y\), even those in the same spoke, it must satisfy the second annular-noncrossing condition.

We note that in the cycle of \(\lambda_{x,y}\), where \(x\in I_{i}\) and \(y\in\varepsilon_{\pi}J_{k-\varepsilon_{\pi}i}\), we see elements from the intervals in the cyclic order \(I_{i},I_{i+1},\ldots,I_{i-1},I_{i},J_{k-\varepsilon_{\pi}i},\allowbreak J_{k-\varepsilon_{\pi}\left(i-1\right)},\ldots,J_{k-\varepsilon_{\pi}\left(i+1\right)},J_{k-\varepsilon_{\pi}i}\) (recalling that subscripts are taken modulo their appropriate range).  If there is an instance of the third annular-crossing condition (variables as given in the definition) where \(a\) and \(c\) belong to a different spoke from \(x\) and \(y\), then \(b\) and \(d\) must belong to the same spoke as \(a\) and \(c\): if \(a\) and \(c\) belong to the same interval of their spoke in \(\lambda_{x,y}\), then one of \(b\) or \(d\) must appear in that interval; and if \(a\) and \(c\) belong to the two different intervals, then no two elements of any other spoke may appear in the correct order.  In either case, they are disc-crossing relative to \(\lambda_{x,y}\), so this configuration cannot occur.  If \(a\) and \(c\) belong to the same spoke as \(x\) and \(y\), then they must be contained in one of the \(K_{j}\), as above.  We note that the \(K_{j}\) coincide with the intervals of this spoke in \(\lambda_{x,y}\), so \(b\) and \(d\) must also be contained within this interval.  However, if \(\delta_{\varepsilon}\pi\delta_{\varepsilon}\) is crossing relative to \(\lambda_{x,y}\), it is crossing relative to \(\lambda_{x^{\prime},y^{\prime}}\) for \(x^{\prime}\) and \(y^{\prime}\) in another spoke, so this configuration also cannot occur.

We now rearrange our expression for the asymptotic covariance in terms of the spokes.  Since a \(\pi\) contributing to the asymptotic value of the covariance must connect \(\left[m\right]\) to exactly one of \(\left[m+1,n\right]\) and \(-\left[m+1,n\right]\), \(\pi\) uniquely determines the sign \(\varepsilon_{\pi}\) and the value \(k\) such that each \(I_{i}\) is connected to \(\varepsilon_{\pi}J_{k-\varepsilon_{\pi}i}\) for all \(i\in\left[p\right]\), and thus the asymptotic covariance can be expressed as a sum over all spoke diagrams for each choice of sign and \(k\).

For a given \(\varepsilon_{0}\) and \(k\), the contribution is a sum over terms corresponding to a choice of connected spoke on each \(I_{i}\) and \(\varepsilon_{0}J_{k-\varepsilon_{0}i}\) (where the contributions are multiplied), and as such may be factored into sums of all connected spokes on \(I_{i}\) and \(\varepsilon_{0}J_{k-\varepsilon_{0}i}\) for all \(i\in\left[p\right]\).

A disc-noncrossing permutation on \(I_{k}\) and \(\varepsilon_{0}J_{l}\) which does not connect the two intervals induces disc-noncrossing permutations on each of \(\left.\lambda_{x,y}\right|_{I_{k}}\) and \(\left.\lambda_{x,y}\right|_{\varepsilon_{0}J_{l}}\).  In terms of the expression given in Lemma~\ref{moment}, the sum over those which do connect the two intervals is then
\[\lim_{N\rightarrow\infty}\mathbb{E}\left(\mathrm{tr}\left(A_{k}B_{l}^{\left(\varepsilon_{0}\right)}\right)\right)-\lim_{N\rightarrow\infty}\mathbb{E}\left(\mathrm{tr}\left(A_{k}\right)\right)\lim_{N\rightarrow\infty}\mathbb{E}\left(\mathrm{tr}\left(B_{l}^{\left(\varepsilon_{0}\right)}\right)\right)\textrm{,}\]
where Lemma~\ref{moment} guarantees the existence of this limit if \(v\left(k\right)=w\left(l\right)\) and it is equal to zero for all \(N\) if \(v\left(k\right)\neq w\left(l\right)\).  The result follows.
\end{proof}

\section{The main definitions}
\label{real}

We take the property satisfied in Theorem~\ref{main} as our definition of asymptotic real second-order freeness:
\begin{definition}
Let \(\left\{X_{c}^{\left(\lambda\right)}\right\}_{\lambda\in\Lambda_{c}}\) be an ensemble of random \(N\times N\) matrices for each colour \(c\in\left[C\right]\).  We say that the ensembles are {\em asymptotically real second-order free} if they are asymptotically free, have a second-order limit distribution, and if, for any \(v:\left[p\right]\rightarrow\left[C\right]\) and \(w:\left[q\right]\rightarrow\left[C\right]\) cyclically alternating words (or words of length \(1\)) in the set of colours \(\left[C\right]\) and \(A_{1},\ldots,A_{p}\) and \(B_{1},\ldots,B_{q}\) random matrices with \(A_{k}\) in the algebra generated by the ensemble associated with \(v\left(k\right)\) and \(B_{k}\) in the algebra generated by the ensemble associated with \(w\left(k\right)\), we have for \(p\neq q\),
\[\lim_{N\rightarrow\infty}k_{2}\left(\mathrm{Tr}\left(\mathaccent"7017{A}_{1}\cdots\mathaccent"7017{A}_{p}\right),\mathrm{Tr}\left(\mathaccent"7017{B}_{1}\cdots\mathaccent"7017{B}_{q}\right)\right)=0\]
and for \(p=q\geq 2\),
\begin{multline*}
\lim_{N\rightarrow\infty}k_{2}\left(\mathrm{Tr}\left(\mathaccent"7017{A}_{1}\cdots\mathaccent"7017{A}_{p}\right),\mathrm{Tr}\left(\mathaccent"7017{B}_{1}\cdots\mathaccent"7017{B}_{p}\right)\right)
\\=\sum_{k=0}^{p-1}\prod_{i=1}^{p}\left(\lim_{N\rightarrow\infty}\mathbb{E}\left(\mathrm{tr}\left(A_{i}B_{k-i}\right)\right)-\mathbb{E}\left(\mathrm{tr}\left(A_{i}\right)\right)\mathbb{E}\left(\mathrm{tr}\left(B_{k-i}\right)\right)\right)\\+\sum_{k=0}^{p-1}\prod_{i=1}^{p}\left(\lim_{N\rightarrow\infty}\mathbb{E}\left(\mathrm{tr}\left(A_{i}B_{k+i}^{T}\right)\right)-\mathbb{E}\left(\mathrm{tr}\left(A_{i}\right)\right)\mathbb{E}\left(\mathrm{tr}\left(B_{k+i}^{T}\right)\right)\right)\textrm{.}
\end{multline*}
\end{definition}

Noting that
\[\mathbb{E}\left(\mathrm{tr}\left(A_{k}B_{l}^{\left(\pm 1\right)}\right)\right)-\mathbb{E}\left(\mathrm{tr}\left(A_{k}\right)\right)\mathbb{E}\left(\mathrm{tr}\left(B_{l}^{\left(\pm 1\right)}\right)\right)=\mathbb{E}\left(\mathrm{tr}\left(\mathaccent"7017{A}_{k}\mathaccent"7017{B}_{l}^{\left(\pm 1\right)}\right)\right)\textrm{,}\]
the above condition is equivalent to the following condition on the algebra generated by the second-order limit distributions of the matrices, which we take as our definition of second-order freeness:
\begin{definition}
\label{real second-order freeness}
Let \(A_{1},\ldots,A_{C}\) be subalgebras of \(A\), \(\left(A,\varphi_{1},\varphi_{2}\right)\) a second-order noncommutative probability space equipped with an involution \(a\mapsto a^{t}\) reversing the order of multiplication.  Then \(A_{1},\ldots,A_{C}\) are {\em real second-order free} if they are free and if, for every \(a_{1},\ldots,a_{p},b_{1},\ldots,b_{q}\in A\) such that \(a_{k}\in A_{v\left(k\right)}\) and \(b_{k}\in A_{w\left(k\right)}\) for \(v:\left[p\right]\rightarrow\left[C\right]\) and \(w:\left[q\right]\rightarrow\left[C\right]\) cyclically alternating words (or words of length \(1\)) in \(\left[C\right]\), for \(p\neq q\),
\[\varphi_{2}\left(\mathaccent"7017{a}_{1}\cdots\mathaccent"7017{a}_{p},\mathaccent"7017{b}_{1}\cdots\mathaccent"7017{b}_{q}\right)=0\]
and for \(p=q\geq 2\)
\[\varphi_{2}\left(\mathaccent"7017{a}_{1}\cdots\mathaccent"7017{a}_{p},\mathaccent"7017{b}_{1}\cdots\mathaccent"7017{b}_{p}\right)=\sum_{k=0}^{p-1}\prod_{i=1}^{p}\varphi_{1}\left(\mathaccent"7017{a}_{i}\mathaccent"7017{b}_{k-i}\right)+\sum_{k=0}^{p-1}\prod_{i=1}^{p}\varphi_{1}\left(\mathaccent"7017{a}_{i}\mathaccent"7017{b}_{k+i}^{t}\right)\textrm{.}\]
\end{definition}

\bibliography{paper}
\bibliographystyle{plain}

\end{document}

%% file: spoke_0.pdf_t
\begin{picture}(0,0)%
\includegraphics{spoke_0.pdf}%
\end{picture}%
\setlength{\unitlength}{3947sp}%
\begingroup\makeatletter\ifx\SetFigFontNFSS\undefined%
\gdef\SetFigFontNFSS#1#2#3#4#5{%
  \reset@font\fontsize{#1}{#2pt}%
  \fontfamily{#3}\fontseries{#4}\fontshape{#5}%
  \selectfont}%
\fi\endgroup%
\begin{picture}(2498,2569)(2393,-4034)
\put(3301,-2386){\makebox(0,0)[lb]{\smash{{\SetFigFontNFSS{12}{14.4}{\rmdefault}{\mddefault}{\updefault}{\color[rgb]{0,0,0}\(a_{2}\)}%
}}}}
\put(3301,-3211){\makebox(0,0)[lb]{\smash{{\SetFigFontNFSS{12}{14.4}{\rmdefault}{\mddefault}{\updefault}{\color[rgb]{0,0,0}\(a_{3}\)}%
}}}}
\put(3901,-2836){\makebox(0,0)[lb]{\smash{{\SetFigFontNFSS{12}{14.4}{\rmdefault}{\mddefault}{\updefault}{\color[rgb]{0,0,0}\(a_{1}\)}%
}}}}
\put(4876,-2836){\makebox(0,0)[lb]{\smash{{\SetFigFontNFSS{12}{14.4}{\rmdefault}{\mddefault}{\updefault}{\color[rgb]{0,0,0}\(b_{3}\)}%
}}}}
\put(2851,-1636){\makebox(0,0)[lb]{\smash{{\SetFigFontNFSS{12}{14.4}{\rmdefault}{\mddefault}{\updefault}{\color[rgb]{0,0,0}\(b_{2}\)}%
}}}}
\put(2851,-3961){\makebox(0,0)[lb]{\smash{{\SetFigFontNFSS{12}{14.4}{\rmdefault}{\mddefault}{\updefault}{\color[rgb]{0,0,0}\(b_{1}\)}%
}}}}
\end{picture}%

%% file: spoke_1.pdf_t
\begin{picture}(0,0)%
\includegraphics{spoke_1.pdf}%
\end{picture}%
\setlength{\unitlength}{3947sp}%
\begingroup\makeatletter\ifx\SetFigFontNFSS\undefined%
\gdef\SetFigFontNFSS#1#2#3#4#5{%
  \reset@font\fontsize{#1}{#2pt}%
  \fontfamily{#3}\fontseries{#4}\fontshape{#5}%
  \selectfont}%
\fi\endgroup%
\begin{picture}(2498,2569)(2393,-4034)
\put(3301,-2386){\makebox(0,0)[lb]{\smash{{\SetFigFontNFSS{12}{14.4}{\rmdefault}{\mddefault}{\updefault}{\color[rgb]{0,0,0}\(a_{2}\)}%
}}}}
\put(3301,-3211){\makebox(0,0)[lb]{\smash{{\SetFigFontNFSS{12}{14.4}{\rmdefault}{\mddefault}{\updefault}{\color[rgb]{0,0,0}\(a_{3}\)}%
}}}}
\put(3901,-2836){\makebox(0,0)[lb]{\smash{{\SetFigFontNFSS{12}{14.4}{\rmdefault}{\mddefault}{\updefault}{\color[rgb]{0,0,0}\(a_{1}\)}%
}}}}
\put(4876,-2836){\makebox(0,0)[lb]{\smash{{\SetFigFontNFSS{12}{14.4}{\rmdefault}{\mddefault}{\updefault}{\color[rgb]{0,0,0}\(b_{3}\)}%
}}}}
\put(2851,-1636){\makebox(0,0)[lb]{\smash{{\SetFigFontNFSS{12}{14.4}{\rmdefault}{\mddefault}{\updefault}{\color[rgb]{0,0,0}\(b_{2}\)}%
}}}}
\put(2851,-3961){\makebox(0,0)[lb]{\smash{{\SetFigFontNFSS{12}{14.4}{\rmdefault}{\mddefault}{\updefault}{\color[rgb]{0,0,0}\(b_{1}\)}%
}}}}
\end{picture}%

%% file: spoke_2.pdf_t
\begin{picture}(0,0)%
\includegraphics{spoke_2.pdf}%
\end{picture}%
\setlength{\unitlength}{3947sp}%
\begingroup\makeatletter\ifx\SetFigFontNFSS\undefined%
\gdef\SetFigFontNFSS#1#2#3#4#5{%
  \reset@font\fontsize{#1}{#2pt}%
  \fontfamily{#3}\fontseries{#4}\fontshape{#5}%
  \selectfont}%
\fi\endgroup%
\begin{picture}(2498,2569)(2393,-4034)
\put(3301,-2386){\makebox(0,0)[lb]{\smash{{\SetFigFontNFSS{12}{14.4}{\rmdefault}{\mddefault}{\updefault}{\color[rgb]{0,0,0}\(a_{2}\)}%
}}}}
\put(3301,-3211){\makebox(0,0)[lb]{\smash{{\SetFigFontNFSS{12}{14.4}{\rmdefault}{\mddefault}{\updefault}{\color[rgb]{0,0,0}\(a_{3}\)}%
}}}}
\put(3901,-2836){\makebox(0,0)[lb]{\smash{{\SetFigFontNFSS{12}{14.4}{\rmdefault}{\mddefault}{\updefault}{\color[rgb]{0,0,0}\(a_{1}\)}%
}}}}
\put(4876,-2836){\makebox(0,0)[lb]{\smash{{\SetFigFontNFSS{12}{14.4}{\rmdefault}{\mddefault}{\updefault}{\color[rgb]{0,0,0}\(b_{3}\)}%
}}}}
\put(2851,-1636){\makebox(0,0)[lb]{\smash{{\SetFigFontNFSS{12}{14.4}{\rmdefault}{\mddefault}{\updefault}{\color[rgb]{0,0,0}\(b_{2}\)}%
}}}}
\put(2851,-3961){\makebox(0,0)[lb]{\smash{{\SetFigFontNFSS{12}{14.4}{\rmdefault}{\mddefault}{\updefault}{\color[rgb]{0,0,0}\(b_{1}\)}%
}}}}
\end{picture}%

%% file: spoke_0_reversed.pdf_t
\begin{picture}(0,0)%
\includegraphics{spoke_0_reversed.pdf}%
\end{picture}%
\setlength{\unitlength}{3947sp}%
\begingroup\makeatletter\ifx\SetFigFontNFSS\undefined%
\gdef\SetFigFontNFSS#1#2#3#4#5{%
  \reset@font\fontsize{#1}{#2pt}%
  \fontfamily{#3}\fontseries{#4}\fontshape{#5}%
  \selectfont}%
\fi\endgroup%
\begin{picture}(2498,2569)(2393,-4034)
\put(3301,-2386){\makebox(0,0)[lb]{\smash{{\SetFigFontNFSS{12}{14.4}{\rmdefault}{\mddefault}{\updefault}{\color[rgb]{0,0,0}\(a_{2}\)}%
}}}}
\put(3301,-3211){\makebox(0,0)[lb]{\smash{{\SetFigFontNFSS{12}{14.4}{\rmdefault}{\mddefault}{\updefault}{\color[rgb]{0,0,0}\(a_{3}\)}%
}}}}
\put(3901,-2836){\makebox(0,0)[lb]{\smash{{\SetFigFontNFSS{12}{14.4}{\rmdefault}{\mddefault}{\updefault}{\color[rgb]{0,0,0}\(a_{1}\)}%
}}}}
\put(4876,-2836){\makebox(0,0)[lb]{\smash{{\SetFigFontNFSS{12}{14.4}{\rmdefault}{\mddefault}{\updefault}{\color[rgb]{0,0,0}\(b_{1}^{t}\)}%
}}}}
\put(2851,-1636){\makebox(0,0)[lb]{\smash{{\SetFigFontNFSS{12}{14.4}{\rmdefault}{\mddefault}{\updefault}{\color[rgb]{0,0,0}\(b_{2}^{t}\)}%
}}}}
\put(2851,-3961){\makebox(0,0)[lb]{\smash{{\SetFigFontNFSS{12}{14.4}{\rmdefault}{\mddefault}{\updefault}{\color[rgb]{0,0,0}\(b_{3}^{t}\)}%
}}}}
\end{picture}%

%% file: spoke_1_reversed.pdf_t
\begin{picture}(0,0)%
\includegraphics{spoke_1_reversed.pdf}%
\end{picture}%
\setlength{\unitlength}{3947sp}%
\begingroup\makeatletter\ifx\SetFigFontNFSS\undefined%
\gdef\SetFigFontNFSS#1#2#3#4#5{%
  \reset@font\fontsize{#1}{#2pt}%
  \fontfamily{#3}\fontseries{#4}\fontshape{#5}%
  \selectfont}%
\fi\endgroup%
\begin{picture}(2498,2569)(2393,-4034)
\put(3301,-2386){\makebox(0,0)[lb]{\smash{{\SetFigFontNFSS{12}{14.4}{\rmdefault}{\mddefault}{\updefault}{\color[rgb]{0,0,0}\(a_{2}\)}%
}}}}
\put(3301,-3211){\makebox(0,0)[lb]{\smash{{\SetFigFontNFSS{12}{14.4}{\rmdefault}{\mddefault}{\updefault}{\color[rgb]{0,0,0}\(a_{3}\)}%
}}}}
\put(3901,-2836){\makebox(0,0)[lb]{\smash{{\SetFigFontNFSS{12}{14.4}{\rmdefault}{\mddefault}{\updefault}{\color[rgb]{0,0,0}\(a_{1}\)}%
}}}}
\put(4876,-2836){\makebox(0,0)[lb]{\smash{{\SetFigFontNFSS{12}{14.4}{\rmdefault}{\mddefault}{\updefault}{\color[rgb]{0,0,0}\(b_{1}^{t}\)}%
}}}}
\put(2851,-1636){\makebox(0,0)[lb]{\smash{{\SetFigFontNFSS{12}{14.4}{\rmdefault}{\mddefault}{\updefault}{\color[rgb]{0,0,0}\(b_{2}^{t}\)}%
}}}}
\put(2851,-3961){\makebox(0,0)[lb]{\smash{{\SetFigFontNFSS{12}{14.4}{\rmdefault}{\mddefault}{\updefault}{\color[rgb]{0,0,0}\(b_{3}^{t}\)}%
}}}}
\end{picture}%

%% file: spoke_2_reversed.pdf_t
\begin{picture}(0,0)%
\includegraphics{spoke_2_reversed.pdf}%
\end{picture}%
\setlength{\unitlength}{3947sp}%
\begingroup\makeatletter\ifx\SetFigFontNFSS\undefined%
\gdef\SetFigFontNFSS#1#2#3#4#5{%
  \reset@font\fontsize{#1}{#2pt}%
  \fontfamily{#3}\fontseries{#4}\fontshape{#5}%
  \selectfont}%
\fi\endgroup%
\begin{picture}(2498,2569)(2393,-4034)
\put(3301,-2386){\makebox(0,0)[lb]{\smash{{\SetFigFontNFSS{12}{14.4}{\rmdefault}{\mddefault}{\updefault}{\color[rgb]{0,0,0}\(a_{2}\)}%
}}}}
\put(3301,-3211){\makebox(0,0)[lb]{\smash{{\SetFigFontNFSS{12}{14.4}{\rmdefault}{\mddefault}{\updefault}{\color[rgb]{0,0,0}\(a_{3}\)}%
}}}}
\put(3901,-2836){\makebox(0,0)[lb]{\smash{{\SetFigFontNFSS{12}{14.4}{\rmdefault}{\mddefault}{\updefault}{\color[rgb]{0,0,0}\(a_{1}\)}%
}}}}
\put(4876,-2836){\makebox(0,0)[lb]{\smash{{\SetFigFontNFSS{12}{14.4}{\rmdefault}{\mddefault}{\updefault}{\color[rgb]{0,0,0}\(b_{1}^{t}\)}%
}}}}
\put(2851,-1636){\makebox(0,0)[lb]{\smash{{\SetFigFontNFSS{12}{14.4}{\rmdefault}{\mddefault}{\updefault}{\color[rgb]{0,0,0}\(b_{2}^{t}\)}%
}}}}
\put(2851,-3961){\makebox(0,0)[lb]{\smash{{\SetFigFontNFSS{12}{14.4}{\rmdefault}{\mddefault}{\updefault}{\color[rgb]{0,0,0}\(b_{1}^{t}\)}%
}}}}
\end{picture}%

%% file: pairing.pdf_t
\begin{picture}(0,0)%
\includegraphics{pairing.pdf}%
\end{picture}%
\setlength{\unitlength}{3947sp}%
\begingroup\makeatletter\ifx\SetFigFont\undefined%
\gdef\SetFigFont#1#2#3#4#5{%
  \reset@font\fontsize{#1}{#2pt}%
  \fontfamily{#3}\fontseries{#4}\fontshape{#5}%
  \selectfont}%
\fi\endgroup%
\begin{picture}(5378,4140)(1434,-5771)
\put(4960,-4324){\makebox(0,0)[b]{\smash{{\SetFigFont{10}{12.0}{\familydefault}{\mddefault}{\updefault}{\color[rgb]{0,0,0}\(i_{7}\)}%
}}}}
\put(4959,-4545){\makebox(0,0)[b]{\smash{{\SetFigFont{10}{12.0}{\familydefault}{\mddefault}{\updefault}{\color[rgb]{0,0,0}\(i_{8}\)}%
}}}}
\put(5432,-5018){\makebox(0,0)[b]{\smash{{\SetFigFont{10}{12.0}{\familydefault}{\mddefault}{\updefault}{\color[rgb]{0,0,0}\(j_{8}\)}%
}}}}
\put(5652,-5057){\makebox(0,0)[b]{\smash{{\SetFigFont{10}{12.0}{\familydefault}{\mddefault}{\updefault}{\color[rgb]{0,0,0}\(j_{9}\)}%
}}}}
\put(6134,-4543){\makebox(0,0)[b]{\smash{{\SetFigFont{10}{12.0}{\familydefault}{\mddefault}{\updefault}{\color[rgb]{0,0,0}\(i_{9}\)}%
}}}}
\put(5656,-3833){\makebox(0,0)[b]{\smash{{\SetFigFont{10}{12.0}{\familydefault}{\mddefault}{\updefault}{\color[rgb]{0,0,0}\(j_{10}\)}%
}}}}
\put(5485,-3823){\makebox(0,0)[b]{\smash{{\SetFigFont{10}{12.0}{\familydefault}{\mddefault}{\updefault}{\color[rgb]{0,0,0}\(j_{7}\)}%
}}}}
\put(3931,-3430){\makebox(0,0)[b]{\smash{{\SetFigFont{10}{12.0}{\familydefault}{\mddefault}{\updefault}{\color[rgb]{0,0,0}\(X\)}%
}}}}
\put(4032,-2707){\makebox(0,0)[b]{\smash{{\SetFigFont{10}{12.0}{\familydefault}{\mddefault}{\updefault}{\color[rgb]{0,0,0}\(X^{T}\)}%
}}}}
\put(3618,-2273){\makebox(0,0)[b]{\smash{{\SetFigFont{10}{12.0}{\familydefault}{\mddefault}{\updefault}{\color[rgb]{0,0,0}\(X\)}%
}}}}
\put(2992,-2386){\makebox(0,0)[b]{\smash{{\SetFigFont{10}{12.0}{\familydefault}{\mddefault}{\updefault}{\color[rgb]{0,0,0}\(X^{T}\)}%
}}}}
\put(2858,-3133){\makebox(0,0)[b]{\smash{{\SetFigFont{10}{12.0}{\familydefault}{\mddefault}{\updefault}{\color[rgb]{0,0,0}\(X\)}%
}}}}
\put(5785,-4193){\makebox(0,0)[b]{\smash{{\SetFigFont{10}{12.0}{\familydefault}{\mddefault}{\updefault}{\color[rgb]{0,0,0}\(X\)}%
}}}}
\put(5790,-4703){\makebox(0,0)[b]{\smash{{\SetFigFont{10}{12.0}{\familydefault}{\mddefault}{\updefault}{\color[rgb]{0,0,0}\(X^{T}\)}%
}}}}
\put(5409,-4670){\makebox(0,0)[b]{\smash{{\SetFigFont{10}{12.0}{\familydefault}{\mddefault}{\updefault}{\color[rgb]{0,0,0}\(X\)}%
}}}}
\put(5358,-4171){\makebox(0,0)[b]{\smash{{\SetFigFont{10}{12.0}{\familydefault}{\mddefault}{\updefault}{\color[rgb]{0,0,0}\(X^{T}\)}%
}}}}
\put(4359,-2995){\makebox(0,0)[b]{\smash{{\SetFigFont{10}{12.0}{\familydefault}{\mddefault}{\updefault}{\color[rgb]{0,0,0}\(j_{1}\)}%
}}}}
\put(4238,-2338){\makebox(0,0)[b]{\smash{{\SetFigFont{10}{12.0}{\familydefault}{\mddefault}{\updefault}{\color[rgb]{0,0,0}\(i_{1}\)}%
}}}}
\put(4041,-2132){\makebox(0,0)[b]{\smash{{\SetFigFont{10}{12.0}{\familydefault}{\mddefault}{\updefault}{\color[rgb]{0,0,0}\(i_{2}\)}%
}}}}
\put(3353,-1916){\makebox(0,0)[b]{\smash{{\SetFigFont{10}{12.0}{\familydefault}{\mddefault}{\updefault}{\color[rgb]{0,0,0}\(j_{2}\)}%
}}}}
\put(3020,-1973){\makebox(0,0)[b]{\smash{{\SetFigFont{10}{12.0}{\familydefault}{\mddefault}{\updefault}{\color[rgb]{0,0,0}\(j_{3}\)}%
}}}}
\put(2513,-2512){\makebox(0,0)[b]{\smash{{\SetFigFont{10}{12.0}{\familydefault}{\mddefault}{\updefault}{\color[rgb]{0,0,0}\(i_{3}\)}%
}}}}
\put(2446,-2785){\makebox(0,0)[b]{\smash{{\SetFigFont{10}{12.0}{\familydefault}{\mddefault}{\updefault}{\color[rgb]{0,0,0}\(i_{4}\)}%
}}}}
\put(2637,-3559){\makebox(0,0)[b]{\smash{{\SetFigFont{10}{12.0}{\familydefault}{\mddefault}{\updefault}{\color[rgb]{0,0,0}\(j_{4}\)}%
}}}}
\put(3513,-3939){\makebox(0,0)[b]{\smash{{\SetFigFont{10}{12.0}{\familydefault}{\mddefault}{\updefault}{\color[rgb]{0,0,0}\(i_{5}\)}%
}}}}
\put(3840,-3859){\makebox(0,0)[b]{\smash{{\SetFigFont{10}{12.0}{\familydefault}{\mddefault}{\updefault}{\color[rgb]{0,0,0}\(i_{6}\)}%
}}}}
\put(4327,-3345){\makebox(0,0)[b]{\smash{{\SetFigFont{10}{12.0}{\familydefault}{\mddefault}{\updefault}{\color[rgb]{0,0,0}\(j_{6}\)}%
}}}}
\put(6091,-4261){\makebox(0,0)[b]{\smash{{\SetFigFont{10}{12.0}{\familydefault}{\mddefault}{\updefault}{\color[rgb]{0,0,0}\(i_{10}\)}%
}}}}
\put(5552,-4891){\makebox(0,0)[b]{\smash{{\SetFigFont{10}{12.0}{\familydefault}{\mddefault}{\updefault}{\color[rgb]{0,0,0}\(Y_{4}\)}%
}}}}
\put(6032,-4426){\makebox(0,0)[b]{\smash{{\SetFigFont{10}{12.0}{\familydefault}{\mddefault}{\updefault}{\color[rgb]{0,0,0}\(D_{5}\)}%
}}}}
\put(5592,-3993){\makebox(0,0)[b]{\smash{{\SetFigFont{10}{12.0}{\familydefault}{\mddefault}{\updefault}{\color[rgb]{0,0,0}\(Y_{5}\)}%
}}}}
\put(4006,-2326){\makebox(0,0)[b]{\smash{{\SetFigFont{10}{12.0}{\familydefault}{\mddefault}{\updefault}{\color[rgb]{0,0,0}\(D_{1}\)}%
}}}}
\put(3219,-2067){\makebox(0,0)[b]{\smash{{\SetFigFont{10}{12.0}{\familydefault}{\mddefault}{\updefault}{\color[rgb]{0,0,0}\(Y_{1}\)}%
}}}}
\put(2598,-2680){\makebox(0,0)[b]{\smash{{\SetFigFont{10}{12.0}{\familydefault}{\mddefault}{\updefault}{\color[rgb]{0,0,0}\(D_{2}\)}%
}}}}
\put(3625,-3695){\makebox(0,0)[b]{\smash{{\SetFigFont{10}{12.0}{\familydefault}{\mddefault}{\updefault}{\color[rgb]{0,0,0}\(D_{3}\)}%
}}}}
\put(4212,-3166){\makebox(0,0)[b]{\smash{{\SetFigFont{10}{12.0}{\familydefault}{\mddefault}{\updefault}{\color[rgb]{0,0,0}\(Y_{3}\)}%
}}}}
\put(2765,-3752){\makebox(0,0)[b]{\smash{{\SetFigFont{10}{12.0}{\familydefault}{\mddefault}{\updefault}{\color[rgb]{0,0,0}\(j_{5}\)}%
}}}}
\put(3183,-3577){\makebox(0,0)[b]{\smash{{\SetFigFont{10}{12.0}{\familydefault}{\mddefault}{\updefault}{\color[rgb]{0,0,0}\(X^{T}\)}%
}}}}
\put(2832,-3600){\makebox(0,0)[b]{\smash{{\SetFigFont{10}{12.0}{\familydefault}{\mddefault}{\updefault}{\color[rgb]{0,0,0}\(Y_{2}\)}%
}}}}
\put(5109,-4447){\makebox(0,0)[b]{\smash{{\SetFigFont{10}{12.0}{\familydefault}{\mddefault}{\updefault}{\color[rgb]{0,0,0}\(D_{4}\)}%
}}}}
\end{picture}%

%% file: Wishart.pdf_t
\begin{picture}(0,0)%
\includegraphics{Wishart.pdf}%
\end{picture}%
\setlength{\unitlength}{3947sp}%
\begingroup\makeatletter\ifx\SetFigFont\undefined%
\gdef\SetFigFont#1#2#3#4#5{%
  \reset@font\fontsize{#1}{#2pt}%
  \fontfamily{#3}\fontseries{#4}\fontshape{#5}%
  \selectfont}%
\fi\endgroup%
\begin{picture}(4450,4004)(1734,-5635)
\put(3711,-2621){\makebox(0,0)[b]{\smash{{\SetFigFont{12}{14.4}{\familydefault}{\mddefault}{\updefault}{\color[rgb]{0,0,0}\(W_{1}\)}%
}}}}
\put(3571,-3396){\makebox(0,0)[b]{\smash{{\SetFigFont{12}{14.4}{\familydefault}{\mddefault}{\updefault}{\color[rgb]{0,0,0}\(W_{3}\)}%
}}}}
\put(3061,-2816){\makebox(0,0)[b]{\smash{{\SetFigFont{12}{14.4}{\familydefault}{\mddefault}{\updefault}{\color[rgb]{0,0,0}\(W_{2}\)}%
}}}}
\put(3226,-2131){\makebox(0,0)[b]{\smash{{\SetFigFont{12}{14.4}{\familydefault}{\mddefault}{\updefault}{\color[rgb]{0,0,0}\(Y_{1}\)}%
}}}}
\put(2896,-3496){\makebox(0,0)[b]{\smash{{\SetFigFont{12}{14.4}{\familydefault}{\mddefault}{\updefault}{\color[rgb]{0,0,0}\(Y_{2}\)}%
}}}}
\put(4126,-3136){\makebox(0,0)[b]{\smash{{\SetFigFont{12}{14.4}{\familydefault}{\mddefault}{\updefault}{\color[rgb]{0,0,0}\(Y_{3}\)}%
}}}}
\put(5311,-4436){\makebox(0,0)[b]{\smash{{\SetFigFont{12}{14.4}{\familydefault}{\mddefault}{\updefault}{\color[rgb]{0,0,0}\(W_{4}\)}%
}}}}
\put(5781,-4441){\makebox(0,0)[b]{\smash{{\SetFigFont{12}{14.4}{\familydefault}{\mddefault}{\updefault}{\color[rgb]{0,0,0}\(W_{5}\)}%
}}}}
\put(5571,-4976){\makebox(0,0)[b]{\smash{{\SetFigFont{12}{14.4}{\familydefault}{\mddefault}{\updefault}{\color[rgb]{0,0,0}\(Y_{5}\)}%
}}}}
\put(5546,-3941){\makebox(0,0)[b]{\smash{{\SetFigFont{12}{14.4}{\familydefault}{\mddefault}{\updefault}{\color[rgb]{0,0,0}\(Y_{4}\)}%
}}}}
\end{picture}%

%% file: mixed.pdf_t
\begin{picture}(0,0)%
\includegraphics{mixed.pdf}%
\end{picture}%
\setlength{\unitlength}{3947sp}%
\begingroup\makeatletter\ifx\SetFigFont\undefined%
\gdef\SetFigFont#1#2#3#4#5{%
  \reset@font\fontsize{#1}{#2pt}%
  \fontfamily{#3}\fontseries{#4}\fontshape{#5}%
  \selectfont}%
\fi\endgroup%
\begin{picture}(7312,3332)(-471,-6107)
\put(449,-4227){\makebox(0,0)[b]{\smash{{\SetFigFont{10}{12.0}{\familydefault}{\mddefault}{\updefault}{\color[rgb]{0,0,0}\(Z\)}%
}}}}
\put(3159,-4657){\makebox(0,0)[b]{\smash{{\SetFigFont{10}{12.0}{\familydefault}{\mddefault}{\updefault}{\color[rgb]{0,0,0}\(Z^{T}\)}%
}}}}
\put(5309,-4167){\makebox(0,0)[b]{\smash{{\SetFigFont{10}{12.0}{\familydefault}{\mddefault}{\updefault}{\color[rgb]{0,0,0}\(Z^{T}\)}%
}}}}
\put(5206,-4726){\makebox(0,0)[b]{\smash{{\SetFigFont{10}{12.0}{\familydefault}{\mddefault}{\updefault}{\color[rgb]{0,0,0}\(W_{2}^{\left(\lambda_{8}\right)}\)}%
}}}}
\put(5791,-4731){\makebox(0,0)[b]{\smash{{\SetFigFont{10}{12.0}{\familydefault}{\mddefault}{\updefault}{\color[rgb]{0,0,0}\(W_{1}^{\left(\lambda_{9}\right)}\)}%
}}}}
\put(2629,-4407){\makebox(0,0)[b]{\smash{{\SetFigFont{10}{12.0}{\familydefault}{\mddefault}{\updefault}{\color[rgb]{0,0,0}\(Z^{T}\)}%
}}}}
\put(451,-4676){\makebox(0,0)[b]{\smash{{\SetFigFont{10}{12.0}{\familydefault}{\mddefault}{\updefault}{\color[rgb]{0,0,0}\(W_{2}^{\left(\lambda_{2}\right)}\)}%
}}}}
\put(2976,-4151){\makebox(0,0)[b]{\smash{{\SetFigFont{10}{12.0}{\familydefault}{\mddefault}{\updefault}{\color[rgb]{0,0,0}\(W_{1}^{\left(\lambda_{3}\right)}\)}%
}}}}
\put(5811,-4206){\makebox(0,0)[b]{\smash{{\SetFigFont{10}{12.0}{\familydefault}{\mddefault}{\updefault}{\color[rgb]{0,0,0}\(W_{2}^{\left(\lambda_{6}\right)}\)}%
}}}}
\end{picture}%